\documentclass[a4paper,reqno,11pt]{amsart}
\usepackage{amsmath, amsfonts, amssymb, amsthm, amscd}
\usepackage[a4paper,scale={0.78,0.78},marginratio={1:1},footskip=7mm,headsep=10mm]{geometry}
\usepackage{tabularx}
\usepackage{graphicx}
\usepackage{psfrag}
\usepackage{perpage}
\usepackage{url}
\usepackage{color}
\usepackage{mathrsfs}
\usepackage{dsfont} 
\usepackage{mathrsfs}
\usepackage{hyperref}
\usepackage{color}
\usepackage{multicol}
\usepackage{mathtools}
\usepackage{mathdots}
\usepackage{tikz}
\usepackage{mathabx}
\usetikzlibrary{arrows,decorations.pathmorphing,backgrounds,positioning,fit,petri} 
\usepackage{tikz,tikz-cd}\usepackage{caption,subcaption}
\usepackage[utf8]{inputenc}
\usepackage[T1]{fontenc}
\usepackage{microtype}

\usepackage{hyperref}
\usepackage[vcentermath]{youngtab}

\makeatletter
\def\@secnumfont{\bfseries\scshape}

\def\section{\@startsection{section}{1}%
  \z@{.7\linespacing\@plus\linespacing}{.5\linespacing}%
  {\normalfont\large\bfseries\scshape\centering}}

\def\subsection{\@startsection{subsection}{2}%
  \z@{.5\linespacing\@plus.7\linespacing}{-.5em}%
  {\normalfont\bfseries\scshape}}
  
\def\subsubsection{\@startsection{subsubsection}{3}%
  \z@{.5\linespacing\@plus.7\linespacing}{-.5em}%
  {\normalfont\scshape}}

\def\specialsection{\@startsection{section}{1}%
  \z@{\linespacing\@plus\linespacing}{.5\linespacing}%
  {\normalfont\centering\large\bfseries\scshape}}
\makeatother

\makeatletter

\renewenvironment{proof}[1][\proofname]{\par
\pushQED{\qed}%
\normalfont \topsep4\p@\@plus4\p@\relax
\trivlist
\item[\hskip\labelsep
\bfseries
#1\@addpunct{.}]\ignorespaces
}{%
\popQED\endtrivlist\@endpefalse
}
\makeatother

\makeatletter
\newcommand \Dotfill {\leavevmode \leaders \hb@xt@ 6pt{\hss .\hss }\hfill \kern \z@}
\makeatother

\makeatletter
\def\@tocline#1#2#3#4#5#6#7{\relax
  \ifnum #1>\c@tocdepth 
  \else
    \par \addpenalty\@secpenalty\addvspace{#2}%
    \begingroup \hyphenpenalty\@M
    \@ifempty{#4}{%
      \@tempdima\csname r@tocindent\number#1\endcsname\relax
    }{%
      \@tempdima#4\relax
    }%
    \parindent\z@ \leftskip#3\relax \advance\leftskip\@tempdima\relax
    \rightskip\@pnumwidth plus4em \parfillskip-\@pnumwidth
    #5\leavevmode\hskip-\@tempdima
      \ifcase #1
       \or\or \hskip 1.65em \or \hskip 3.3em \else \hskip 4.95em \fi%
      #6\nobreak\relax
    \Dotfill
    \hbox to\@pnumwidth{\@tocpagenum{#7}}\par
    \nobreak
    \endgroup
  \fi}
\makeatother

\makeatletter
\def\l@section{\@tocline{1}{0pt}{1pc}{}{\scshape}}
\renewcommand{\tocsection}[3]{%
\indentlabel{\@ifnotempty{#2}{\ignorespaces#1 #2.\hskip 0.7em}}#3}
\def\l@subsection{\@tocline{2}{0pt}{1pc}{5pc}{}}

\def\l@subsubsection{\@tocline{3}{0pt}{1pc}{7pc}{}}

\makeatother

\numberwithin{equation}{section}

\newtheoremstyle{mytheorem}{.7\linespacing\@plus.3\linespacing}{.7\linespacing\@plus.3\linespacing}%
     {\itshape}
     {}
     {\bfseries}
     {. }
     {0.3ex}
     {\thmname{{\bfseries #1}}\thmnumber{ {\bfseries #2}}\thmnote{ (#3)}}  

\theoremstyle{mytheorem}

\newtheorem{theorem}{Theorem}[section]

\newtheorem{proposition}[theorem]{Proposition}

\newtheorem{remark}[theorem]{Remark}
\newtheorem{definition}[theorem]{Definition}

\newtheorem{assumption}[]{Assumption}

\newtheorem{question}[]{Question}

\newcommand{\bbE}{{\ensuremath{\mathbb E}} }

\newcommand{\bbN}{{\ensuremath{\mathbb N}} }

\newcommand{\bbP}{{\ensuremath{\mathbb P}} }

\newcommand{\bbR}{{\ensuremath{\mathbb R}} }

\newcommand{\bbV}{{\ensuremath{\mathbb V}} }

\newcommand{\bbZ}{{\ensuremath{\mathbb Z}} }

\newcommand{\cC}{{\ensuremath{\mathcal C}} }

\newcommand{\cE}{{\ensuremath{\mathcal E}} }
\newcommand{\cF}{{\ensuremath{\mathcal F}} }

\newcommand{\cN}{{\ensuremath{\mathcal N}} }

\newcommand{\cP}{{\ensuremath{\mathcal P}} }

\newcommand{\cS}{{\ensuremath{\mathcal S}} }

\newcommand{\cW}{{\ensuremath{\mathcal W}} }

\newcommand{\cY}{{\ensuremath{\mathcal Y}} }

\newcommand{\ga}{\alpha}
\newcommand{\gb}{\beta}

\newcommand{\gl}{\lambda}

\newcommand{\go}{\omega}

\newcommand\sfA{\mathsf A}
\newcommand\sfB{\mathsf B}

\newcommand\sfI{\mathsf I}

\newcommand\sfR{\mathsf R}

\newcommand\sfX{\mathsf X}

\newcommand\sfa{\mathsf a}
\newcommand\sfb{\mathsf b}

\newcommand\sff{\mathsf f}

\newcommand{\mvw}{\boldsymbol{w}}\newcommand{\mvx}{\boldsymbol{x}}
\newcommand{\mvt}{\boldsymbol{t}}

\renewcommand{\tilde}{\widetilde}          
\DeclareMathSymbol{\leqslant}{\mathalpha}{AMSa}{"36} 
\DeclareMathSymbol{\geqslant}{\mathalpha}{AMSa}{"3E} 
\DeclareMathSymbol{\eset}{\mathalpha}{AMSb}{"3F}     
\newcommand{\dd}{\text{\rm d}}             

\newcommand{\sumtwo}[2]{\sum_{\substack{#1 \\ #2}}} 

\newcommand{\R}{\mathbb{R}}

\newcommand{\Z}{\mathbb{Z}}
\newcommand{\N}{\mathbb{N}}

\def\bs{\boldsymbol}

\newcommand{\PEfont}{\mathrm}

\newcommand{\p}{\ensuremath{\PEfont P}}
\newcommand{\e}{\ensuremath{\PEfont E}}
\newcommand{\E}{\e}
\renewcommand{\P}{\p}

\newcommand\bE{\ensuremath{\bs{\mathrm{E}}}}

\DeclareMathOperator{\bbvar}{\ensuremath{\mathbb{V}ar}}
\DeclareMathOperator{\bbcov}{\ensuremath{\mathbb{C}ov}}

\newcommand{\ind}{\mathds{1}}

\renewcommand{\epsilon}{\varepsilon}
\renewcommand{\rho}{\varrho}

\newenvironment{myenumerate}{%
\renewcommand{\theenumi}{\arabic{enumi}}%
\renewcommand{\labelenumi}{{\rm(\theenumi)}}%
\begin{list}{\labelenumi}
	{%
	\setlength{\itemsep}{0.4em}%
	\setlength{\topsep}{0.5em}%
	\setlength\leftmargin{2.45em}%
	\setlength\labelwidth{2.05em}%
	\setlength{\labelsep}{0.4em}%
	\usecounter{enumi}%
	}%
	}%
{\end{list}
}

{\end{list}
}

{\end{list}
}

{\end{myenumerate}}

\newenvironment{myitemize}{%
\begin{list}{$\bullet$}%
 	{%
	\setlength{\itemsep}{0.4em}%
	\setlength{\topsep}{0.5em}%
	\setlength\leftmargin{2.45em}%
	\setlength\labelwidth{2.05em}%
	\setlength{\labelsep}{0.4em}%
	}%
	}%
{\end{list}}

\renewenvironment{itemize}{
\begin{myitemize}}%
{\end{myitemize}}

\MakePerPage[2]{footnote} 

\def\dd{\mathrm{d}}

\def\hbeta{{\hat{\beta}}}

\usepackage{bbold} 
\renewcommand{\ind}{\mathbb{1}}

\begin{document}

\title
[ Phase transitions of the DPRE]
{Directed polymers in a random environment: a review of the phase transitions}


\author[N.Zygouras]{Nikos Zygouras }
\address{Department of Mathematics\\
University of Warwick\\
Coventry CV4 7AL, UK}
\email{N.Zygouras@warwick.ac.uk}

\begin{abstract}
The model of {\it directed polymer in a random environment} is a fundamental model of
interaction between a simple random walk and ambient disorder. This interaction 
gives rise to complex phenomena and transitions from a central limit theory 
to novel statistical behaviours. Despite its intense study, there are still many aspects and phases
which have not yet been identified. In this review we focus on the current status
of our understanding of the transition between weak and strong disorder phases, 
give an account of some of the methods that the study of the model has motivated
and highlight some open questions.

\end{abstract}

\date{\today}


\maketitle
\begin{center}
{\it Dedicated to the memory of Francis Comets}
\end{center}
\tableofcontents

\section{Introduction}
The model of directed polymer in a random environment was conceived almost forty
years ago as a model of one dimensional interfaces of disordered Ising magnets. The interest
in the model soon went beyond its original purpose. The reason for this was that, on the one hand, 
it demonstrated a non trivial interaction between disorder and the underlying simple random walk,
which indicated new exponents different from those dictated by the central limit theorem.
On the other hand, it exhibited links to universality of stochastic growth phenomena, what is nowadays 
known as the Kardar-Parisi-Zhang universality class. One of the most intriguing features is the transition
that this model exhibits between the central limit theory behaviour of that of a simple random walk
 and new statistical phenomena. 
 Besides dimension one, where there is now a good anticipation of several aspects of
the novel statistical theory (even if a mathematical theory is still not fully in place), 
in higher dimensions the situation as to when the transition happens and 
what new statistics emerge is nebulous. In this review we will try to summarise the current status of our
understanding and to present some prominent methods that have emerged during this endeavour, as well
as highlight some questions. 

Let us define the model of the directed polymer in random environment (DPRE) on $\Z^d$: 
We assume that $S=(S_n)_{n\geq 0}$ is a random walk on $\bbZ^d$. 
We will denote its law when starting at $S_0=x\in\bbZ^d$ by $\P_x$ and we will often
drop the subscript $x$ if this is $0$.
Typically, we will consider the case of 
a simple, symmetric random walk but the option of more general lattice random walks, e.g. performing long
jumps, is also available, although less studied in the literature. In addition to $S$, we consider a family of random variables $\omega=(\omega_{n,x})_{n\in \bbN, x\in \bbZ^d}$, which plays the role of disorder. We denote its distribution by $\bbP$ and the associated expectation by 
$\bbE$. Again, typically, $\omega$ is taken to be a family of independent, identically distributed 
(i.i.d.) random variables with
\begin{align}\label{omegaass}
\bbE[\omega_{n,x}]=0,\quad \bbV ar(\omega_{n,x})=1, \quad \text{and}\quad 
\lambda(\beta):=\log\bbE[e^{\beta \omega_{n,x}}]<\infty, \quad \text{for $\beta\in \bbR$}.
\end{align}
The possibility of correlated disorder or disorder without exponential moments has also been
considered  but unless we assume explicitly such hypotheses, we will be working with the
 assumptions of i.i.d. and \eqref{omegaass}.
 The directed polymer measure is a probability measure on random walk paths $S$ defined as 
\begin{align}\label{def:DPRE}
\dd \mu_{n,x}^{\beta, \omega}(S) 
&:=\frac{1}{Z_{n}^{\beta, \omega}(x)} e^{\sum_{i=1}^{n} \beta\omega_{i,S_i}} \,\dd \P_x(S) \quad
\text{with $\,\,\, Z_{n}^{\beta, \omega}(x):=\E_x[e^{\sum_{i=1}^{n} \beta\omega_{i,S_i}}]$} .
\end{align}
The normalisation $ Z_{n,x}^{\beta, \omega}(x)$ is called the partition function. It contains
a significant amount of information and will be discussed extensively in these notes. 
Both $\mu_{n,x}^{\beta, \omega}$ and $ Z_{n,x}^{\beta, \omega}(x)$ are random, depending on the
disorder $\omega$, but
to lighten notation we will drop the superscript $\omega$ and simply write 
$\mu_{n,x}^\beta$ and $Z_{n}^\beta(x)$ and when the starting point $x=0$, we will simply write
$\mu_{n}^\beta$ and  $Z_{n}^\beta$.

We will also want to normalise the partition function and the Gibbs factor so that they have mean $1$
with respect to $\bbE$ and so we subtract from the exponential the term $n\lambda(\beta)$, 
thus writing
\begin{align}\label{defW}
\dd \mu_{n,x}^{\beta, \omega}(S) 
&:=\frac{1}{W_n^{\beta,\omega}(x)} e^{\sum_{i=1}^{n} (\beta\omega_{i,S_i}-\lambda(\beta))} \,\dd \P_x(S) \quad
\text{with $\,\,\, W_{n}^{\beta, \omega}(x):=\E_x[e^{\sum_{i=1}^{n} (\beta\omega_{i,S_i} -\lambda(\beta))} ]$} .
\end{align}
The measure $\mu_{n,x}^\beta$ models the situation where a random 
walk path $S$ moves inside a random ``landscape'', which affects its motion. In particular, we see
from its definition that $\mu_{n,x}^\beta$ gives more weight to trajectories $(S_1,...,S_n)$ that visit
sites that maximise the ``energy'' $\sum_{i=1}^n \beta \omega_{i,S_i}$, i.e. the cumulative 
disorder encountered along the random walk trajectory up to time $n$. 
This interaction might affect the diffusive behaviour of the random walk since the latter will try
to reach possibly remote and spatially restricted locations, where disorder takes atypically high values. In such a situation we say that the walk becomes {\it localised} as it tends to move in 
narrow corridors, which carry favourable disorder. Of course,
an atypical path trajectory, will have a cost in ``entropy'', quantified by $\P_x(S)$ and so the
problem of determining whether the walk will adopt a different long time asymptotic is a matter
of a complex energy-entropy balance. 

Understanding when a phase transition between diffusive and super-diffusive
 behaviour takes place has been a matter of study for very long time but still many 
questions remain to be settled. This will be the focus of this review.
We plan to start from the classical results that dominated the first 15-20 years of the lifespan of the model
and which demonstrated the existence of a transition between a {\it weak disorder} phase (where the walk maintains its diffusive behaviour) and a {\it strong disorder} phase (where deviation from diffusivity is expected). 
We will then pass to recent progress in the understanding of this transition both in higher dimensions
as well as at the {\it critical dimension} $2$, where the transition is hidden. 
We also discuss phase diagrams when disorder has heavy tails.

The directed polymer model has demanded the deployment of a variety of methods
and has also motivated the development of new ones. The first 15-20 years focused on martingale 
methods and a partial characterisation of the weak disorder regime. 
At the turn of the century there was the boom of combinatorial and representation theoretic
 methods related to the random matrix theory and integrability,
which gave the rigorous derivation of the conjectured $1/3,2/3$ 
exponents for certain one-dimensional, {\it integrable} polymers at {\it zero temperature} and the 
identification of the limit theory with statistics of eigenvalues of  random matrices. The expansion
of what is now named {\it integrable probability} allowed for the treatment of a few integrable
polymer models at positive temperature, as well as a number of other models in the KPZ universality
class. 

Around that time and still continuing till today, methods more grounded to probability are developing 
with the purpose of better understanding the transitions in dimensions higher than one and towards
building a universality framework. These methods combine a variety of geometric and probabilist intuition
and tools. Some of the keywords in this direction are: pinning interfaces, fractional moment methods, concentration estimates, coarse graining, Lindeberg principles. Furthermore, the original martingale techniques from the nineties and beginning of the millennium have very recently been revisited and expanded, opening new prospects in the understanding of the phase transitions in higher dimensions.

The above set of methods will be the main focus of this review. Unfortunately, we will not discuss 
methods and aspects relating to integrable probability and one-dimensional KPZ asymptotics. This
is a rather separate and extensive topic to which a number of other reviews have been devoted
and to which we refer \cite{J05, C12, QS15, BDS16, BG16, BP14, Z22}.
We will also not discuss other pertinent polymer models, such as pinning and wetting models or
polymers with self interactions or large deviation type methods.
Some references on this topic are the books \cite{G07, dH09, G11, CdHP12} and we refer
to them. Complementary features and further references of the topics we discuss here can be found in
Francis Comets' St. Flour Lectures \cite{C16}.
\vskip 2mm
This review grew out of a very sad event, that of the loss of Francis Comets.
Francis had been one of the most passionate advocates of the directed polymer model (alongside
his many other interests in disordered systems, complex random walks and more) and his early works 
in collaboration with Shiga and Yoshida played an important role in setting a framework and popularising the model. 
In fact I, myself, came into contact with the subject
 and the surrounding problems from his review paper with Shiga and Yoshida \cite{CSY04}, when
I was a graduate student. This certainly had an impact on me. Francis continued throughout his career to
investigate, develop and popularise the subject, through novel contributions, exposition and 
nurturing of young mathematicians. A summary of his efforts and signature on the topic can be seen in his
   St. Flour lecture notes \cite{C16}. The current review may just be seen as an addendum to Francis'
   St. Flour notes, which includes some of the developments that have taken place since then.
   
   Francis Comets was a dedicated and passionate mathematician, who genuinely cared about the 
   development of the field and the community. He will be missed but the impact of his works and
   efforts will stay. His kindness and generosity will also be remembered by those who met him.
   
   \vskip 4mm
{\bf Acknowledgements.} I would like to thank Thierry Bodineau, Patrick Cattiaux and 
Giambattista Giacomin for the invitation to prepare this review and in particular Thierry Bodineau for his encouragement and patience
during the preparation. I also thank Rongfeng Sun for feedback on the manuscript as well as Amol Aggarwal,  Ofer Busani, Ron Peled and
the anonymous referee for their useful comments. I am grateful to Quentin Berger, Stefan Junk and Hubert Lacoin,  for their
very detailed comments which improved the review and in particular S. Junk and H. Lacoin for settling a number of the questions posed in the first arxiv version of this review. 
The writing of these notes started at the National Taiwan University and I thank the institute and Jhih-huang Li for the hospitality.
The work was partially supported by EPSRC through grant EP/R024456/1.
   \section{A brief historical overview.} 

\subsection{ The physical origins.}
The directed polymer model was first introduced in a one-dimensional setting by Huse and Henley
\cite{HH85} as a model of the interface separating the $+$ and the $-$ states in a two-dimensional 
Ising model
with coupling disorder. It is instructive
to see how the model emerges from this angle, so let us briefly review this. 

Let $\Lambda_n:=[0,n]^2\cap \Z^2$
and consider the {\it spin field} $\sigma=(\sigma_x)_{x\in \Lambda_n}$ with
 $\sigma_x\in\{+1,-1\}$ for all $x\in \Lambda_n$. On top of this field, consider the disordered field of couplings $(J_{x,y}^\omega)_{x,y\in \Lambda_n}$, which is assumed to be an i.i.d. family of random
 variables, which are non-zero only if $x,y$ are neighbouring sites, i.e. $|x-y|=1$.  
 Upon the field $\sigma$, we will impose the probability distribution
\begin{align*}
\dd \mu^{\rm Ising}_{n}(\sigma):=\frac{1}{Z^{\rm Ising}_n} e^{\sum_{x,y\in \Lambda_n, \, |x- y|=1} J^\omega_{x,y} \sigma_x \sigma_y} 
\quad \text{with} \quad
Z^{\rm Ising}_n:= \sum_{\sigma \in \{\pm 1\}^{\Lambda_n}}
 e^{\sum_{x,y\in \Lambda_n, \, |x- y|=1} J^\omega_{x,y} \sigma_x \sigma_y} .
\end{align*}
To simplify the situation, we assume that there are just two regions inside $\Lambda_n$:
 one which has all spin values equal to $+1$ and the other with all spins $-1$, and we assume that the two regions are 
 separated by an interface which does not have loops or overhangs, i.e. it is a path, which we denote by $\pi$,  see Figure \ref{fig:ising}. To be consistent with this, one should also assign $-$ boundary condition on the part of the boundary
 above the interface and $+$ to the part of the boundary below.
 Since $\sigma_x\sigma_y=1$ if $\sigma_x, \sigma_y$ have the same sign, we can write
 \begin{align*}
 \sum_{x,y\in \Lambda_n, \, |x- y|=1} J^\omega_{x,y} \sigma_x \sigma_y
 = \sum_{x,y\in \Lambda_n} J^\omega_{x,y} -2\sum_{x,y\sim \pi} J^\omega_{x,y},
 \end{align*}
 where $x,y\sim \pi$ means that sites $x,y$ are adjacent and on opposite sites of path $\pi$.
 The Ising measure, then, takes the form
 \begin{align*}
 \dd \mu^{\rm Ising}_{n}:=\frac{1}{\hat Z^{\rm Ising}_n} e^{-2\sum_{x,y\in \Lambda_n, \, x,y\sim\pi} J^\omega_{x,y}} 
\quad \text{with} \quad
\hat Z^{\rm Ising}_n(\sigma):= \sum_{\sigma \in \{\pm 1\}^{\Lambda_n}}
 e^{-2\sum_{x,y\in \Lambda_n, \, x,y\sim \pi} J^\omega_{x,y} } ,
 \end{align*}
 since the term $\sum_{x,y\in\Lambda_n} J^\omega_{x,y}$ cancels out in the ratio.
 Even though paths $\pi$ can be non-nearest neighbour, the similarity to the directed
 polymer measure (where paths $S$ are considered to be simple random walks) 
 is apparent and in fact one should expect similar behaviour given that both
 $S$ and $\pi$ fall within the scope of Donsker's invariance principle.
 In fact, Huse and Henley assumed such invariance principle and replaced the discrete interface with
 a continuum model to which they performed numerical calculation, which revealed a localisation-delocalisation transition and new scaling (or ``roughness'') exponents, which indicated deviations
 from diffusive behaviour. Huse and Henley also considered d-dimensional lattices. Then the
 connection to the directed polymer is lost, since paths are replaced $(d-1)$-dimensional interfaces.
 Predictions of a localisation phase were put forward in \cite{HH85} for dimensions $5/3<d<5$ but
 most conclusive and striking were their numerical predictions in $d=2$ 
 (where the links to polymer are present). In this case \cite{HH85} predicted the standard deviations of a  polymer path of length
 $n$ are of order $n^{2/3}$ as opposed to the $n^{1/2}$ deviations of a simple random walk.
 Moreover, they predicted that the fluctuations of the logarithm of the partition function $Z_n^\beta$ are of
 order $n^{1/3}$.
    \begin{figure}[t]
 \begin{center}
\begin{tikzpicture}[scale=.6]
\draw [ultra thick, red] (0,0)--(1.5, 0) -- (1.5,2.5) -- (2.5,2.5) --(2.5,1.5) -- (3.5, 1.5)--(3.5, 4.5)--
(6.5, 4.5)--(6.5, 2.5)--(7.5,2.5)--(7.5, 3.5) -- (8.5, 3.5) --(8.5, 1.5)--(9.5, 1.5)--(9.5, 2.5)--(10.5, 2.5)--
(11.5, 2.5)--(11.5, 0.5)--(12, 0.5);
\foreach \i in {0,...,8}{
  	\foreach \j in {0,...,12}{
		\node[draw,circle,inner sep=1pt,fill] at (\j,\i) {};
	}
	\fill[white] (4.5,0.5)--(4.5,2.5)--(6.5,2.5)--(6.5,0.5)--(4.5,0.5); 
	\node at (5.5, 1.5) {$+$};
	\fill[white] (8.5,5.5)--(8.5,7.5)--(10.5,7.5)--(10.5,5.5)--(8.5,5.5); 
	\node at (9.5, 6.5) {$-$};
	}
\end{tikzpicture}
 \end{center}  
\caption{ \small  An interface separating a $+$ from a $-$ phase in an Ising model. We assume that
phases do not contain isolated islands of opposite sign and the the interface does not have overhangs. }
\label{fig:ising}
\end{figure}

 After the work of Huse and Henley a considerable interest emerged as to the universality of the
 $1/3, 2/3$ exponents observed numerically in \cite{HH85} and further queried
  by Kardar \cite{K85} (see also Huse-Henley-Fischer \cite{HHF85}). 
  The reason that led to querying the universality 
 of these exponents was that similar exponents had already been observed in the study of hydrodynamics and Burgers turbulence by Foster, Nelson and Stephen \cite{FNS77}
 through a more theoretical, dynamic renormalisation approach. 
 The proposal of the universality of the new exponents
  was then put forward by Kardar, Parisi and Zhang in the
 famous KPZ paper \cite{KPZ86}. There they proposed that fluctuations of random (growing)
$d$-dimensional interfaces are governed by the stochastic PDE
 \begin{align}\label{eq:KPZ}
 \partial_t h= \frac{1}{2} \Delta h +\frac{1}{2} |\nabla h|^2 + \xi,\qquad t\geq 0 , \,\,x\in \bbR^d,
 \end{align}
 where $h(t,x)$ denotes the height of the interface at time $t$ and spatial location $x$ and
 $\xi=\xi(t,x)$ is a space-time white noise, that is a Gaussian process with delta correlations
 as $\bbE[\xi(t,x)\, \xi(s,y)] = \delta(t-s) \delta(x-y)$.
 
 The relation between the KPZ equation and the directed polymer model arises through
 a Cole-Hopf transformation: $h(t,x)=:\log u(t,x)$, which transforms \eqref{eq:KPZ} to the stochastic
 heat equation (SHE)
 \begin{align}\label{eq:SHE}
 \partial_t u= \frac{1}{2}\Delta u +\xi u,  \,\,x\in \bbR^d,
 \end{align}
 whose solution can be written via the (stochastic version of the) Feynman-Kac formula 
 \cite{BC95} as
 \begin{align*}
 u(t,x) = \E_x\Big[\, e^{\int_0^t \beta \xi(t-s,B(s)) \,\dd s - \frac{\beta^2}{2}
 \bbE\big[\big( \int_0^t \xi(t-s,B(s)) \,\dd s \big)^2\big]}\, \Big],
 \end{align*}
 with $\E_x$ is the expectation over a Brownian motion starting from $x\in\bbR^d$.
 Even though neither term in the exponential formally makes sense (they are infinities) and 
 a proper interpretation is required, one can see the connection to the directed polymer model:
 one just needs to replace the white noise $\xi$, which is the continuum analog of i.i.d. family,
 with i.i.d. variables $(\omega_{n,x})$ on the discrete lattice, the Brownian motion with the trajectory
 of a simple random walk and 
 $\frac{\beta^2}{2} \bbE\big[\big( \int_0^t \xi(t-s,B(s)) \,\dd s \big)^2\big]$
 with the corresponding log-moment generating function $\lambda(\beta)$, which in the Gaussian case is just $\frac{\beta^2}{2}$.
 
\subsection{Passage into mathematics.} In the mathematical literature the directed polymer model was introduced by Imbrie and Spencer
 \cite{IS88}. Analysing certain expansions of $W_n^\beta$ they showed that diffusivity can still persist. More precisely, in dimensions $d\geq 3$ and small enough $\beta$
it holds that $\E^{\mu_n^\beta}[|S_n|^2/n]\to1$, a.s. with respect to $\omega$.  
Diffusivity was then upgraded to a central limit theorem by Bolthausen \cite{B89}, who also
introduced a martingale approach. In particular, 
he showed that in the same regime of $\beta$ as \cite{IS88}, $d\geq 3$ and 
$a.e.$ realisation of the disorder $\omega$, 
 $S_n/\sqrt{n}$ converges to a normal distribution with 
covariance matrix $\frac{1}{d} \sfI$ and $\sfI$ the identity matrix. 
The martingale approach was further used to upgrade the central limit theorem to an invariance
principle \cite{AZ96}, while the expansion approach of \cite{IS88} led Sinai \cite{S95} to prove a
local limit theorem in $d\geq 3$ and small $\beta$.
The diffusive regime that the above works addressed was identified (see also \cite{S95, SZ96})
as $\{\beta \colon \lambda(2\beta)-2\lambda(\beta) < \log\frac{1}{\pi_d}\}$, where $\pi_d$ is the probability that 
a simple $d$-dimensional random walk which starts at $0$ returns to $0$ in finite time. 
We will later see that this condition emerges from the requirement that the second moment
of $W_n^\beta$ remains bounded.

The regime where the polymer exhibits or it is expected to exhibit diffusive behaviour is 
called {\it weak disorder regime}. Opposite to that is the {\it strong disorder regime} where
 localisation and super-diffusivity behaviour is expected to take place. A dichotomy
between weak and strong disorder was identified according to whether $W_n^\beta$
converges a.s. to a random variable $W_\infty^\beta$, which is a.s. strictly positive (weak disorder)
or a.s. identically equal to $0$ (strong disorder). Identifying the critical value $\beta_c$ 
(in high dimensions) is still a challenging problem and large part of this review will consider
the progress around this. Below the critical $\beta_c$ Comets and Yoshida \cite{CY06} managed
to extend Bolthausen's central limit theorem, showing that the polymer measure $\mu_n^{\beta,\omega}$ 
satisfies a CLT  in probability, with respect to $\omega$,
up to the critical temperature as $n\to\infty$. A CLT for $\mu_n^{\beta,\omega}$ 
that would hold for $a.e.$ disorder $\omega$, for general disorder,
in the whole weak disorder regime is still an open problem. In a Brownian environment such a CLT has been established in
\cite{J23b}. In the strong disorder regime Carmona-Hu \cite{CH02}, 
Comets-Shiga-Yoshida \cite{CSY03} and Vargas \cite{V07}
 demonstrated\footnote{the result on the $\sup_{x}\mu_n^{\beta, \omega}(S_n=x)$ not converging to zero
 was proved in \cite{CH02, CSY03} in $d=1,2$ or under the additional assumption of {\it ``very'' strong disorder},
 which we now know that it is equivalent to strong disorder \cite{JL24a}} localisation phenomena in the form
of $\sup_{x}\mu_n^\beta(S_n=x)$ not converging to zero  
- contrary to what happens for simple random walks.

 \subsection{The rise of integrable probability.}
 Going beyond the weak disorder regime has been a very challenging problem as a new type 
 of probability theory, which would cross beyond the central limit theory, is required. 
 In fact the exponents $1/3$ and $2/3$ predicted by \cite{HH85, KPZ86, FNS77} indicated 
 a non central limit universality. The first breakthrough in this direction was achieved through
 the works of Baik-Deift-Johansson \cite{BDJ99}, Johansson \cite{J00}, Prahofer-Spohn \cite{PS02},
 where the exponents $1/3, 2/3$ were firmly established and the limiting distribution was discovered
 to be identical to distributions arising in random matrix theory. More precisely, it was established that
 for the {\it ground state} of a directed polymer model (more widely known as {\it last passage percolation})
 in an environment with exponential or geometric distribution the following limit theorem takes place:
 \begin{align}\label{LPPGUE}
\frac{ \max_{S:(0,0)\to(2n,0)} \sum_{i=1}^{2n}{\omega_{i,S_i}}- \mu n}{\sigma n^{1/3}}
 \xrightarrow[n\to\infty]{d} {\rm TW}_{\rm GUE},
 \end{align} 
 where the maximum is over all simple random walk paths starting from $0$ at time $0$ and ending 
 also at $0$ at time $2n$, $\mu$ and $\sigma$ are model specific constants and ${\rm TW}_{\rm GUE}$
 is the Tracy-Widom GUE ({\it Gaussian Unitary Ensemble}), which described (in the same (!) $n^{1/3}$
 scaling) the fluctuations of the largest eigenvalue of an ensemble of unitary random matrices
 with independent (up to unitary symmetry) Gaussian entries.
 
 The $n^{2/3}$ exponent emerges through the limit at a process level
 \begin{align}\label{LPPAiry}
\Bigg\{\frac{ \max_{S:(0,0)\to(2n,\lfloor n^{2/3} x \rfloor )} \sum_{i=1}^{2n}{\omega_{i,S_i}}- \mu n}{\sigma n^{1/3}} \Bigg\}_{x\in \R}
 \xrightarrow[n\to\infty]{d} \big\{ {\rm Airy}_2(x)-x^2 \big\}_{x\in \R},
 \end{align} 
 where now the end point of the path varies through $\lfloor n^{2/3} x \rfloor$ 
 (respecting the parity of the walk)  and ${\rm Airy}_2(\cdot)$
 is the {\it Airy-two} process: a continuous, stationary process with one-dimensional marginal given
 by the Tracy-Widom GUE distribution. This process was constructed in \cite{PS02, J03} and the
 finite dimensional distributions were fully determined.
 
 To achieve these novel limit theorems, novel links to algebraic combinatorics, symmetric function theory
  and integrability were required. Some of the key words here are {\it Robinson-Schensted-Knuth correspondence, Young tableaux, Schur functions, Fredholm determinants, steepest descent asymptotics}.
 This endeavour gave rise to the ``new'' probability theory, which is now known as {\it Integrable Probability} and which allowed the study of a large number of models in the one-dimensional KPZ universality class. Unfortunately,  we will not be able to discuss the methods and results of this exciting field but we can
 refer to a number of relevant reviews \cite{J05, C12, QS15, BDS16, BG16, BP14, Z22}.
 
 The analogue of \eqref{LPPGUE} for the directed polymer model:
   \begin{align}\label{logAiry}
\frac{\log Z_{2n}^\beta(0, 0) - \sff(\beta) n}{\sigma(\beta) n^{1/3}} 
  \xrightarrow[n\to\infty]{d} {\rm TW}_{\rm GUE},
  \end{align}
  (the argument $(0,0)$ in $Z_n^\beta(0,0)$ denotes that the path starts and ends at $0$) 
  was achieved later through \cite{COSZ14, OSZ14, BCR13} for an integrable polymer model, 
  known as the log-gamma polymer. This model was introduced by Sepp\"al\"ainen \cite{S12}
  where the disorder $\omega$ has a log-gamma distribution. Analogous results have also 
  been obtained for a semi-discrete, Brownian polymer model, called the O'Connell-Yor
  polymer \cite{OY01, OC12}.
The analogue of \eqref{LPPAiry}, i.e.
  \begin{align}\label{logAiry}
 \Bigg\{ \frac{\log Z_n^\beta(0, \lfloor n^{2/3}x \rfloor) - \sff(\beta) n}{\sigma(\beta) n^{1/3}} \Bigg\}
  \xrightarrow[n\to\infty]{d} \big\{ {\rm Airy}_2(x)-x^2 \big\}_{x\in \R},
  \end{align}
  was recently obtained by Arrarwal and Huang (\cite{AH23}, see Corollary 25.2 therein) employing also geometric and probabilistic methods
  on top of integrable structures.
 \subsection{Probability strikes back.}
 Integrable probability has been very successful in obtaining very detailed results for a large range
of models, which possess some specific integrability (exponential last passage percolation, log-gamma polymer, TASEP and ASEP, six vertex and tiling models etc). Aspects of universality, though, cannot be
covered by the integrable probability approach as integrability is lost by any change of parameters, however
small this might be. 

Nevertheless, there has recently been efforts to combine the detailed information
that emerges through integrability with more probabilistic methods.  
Some\footnote{the list will not be complete but the readers can guide themselves through additional references therein} of the highlights in this direction are:
\begin{itemize}
\item The {\it Airy and KPZ line ensembles} \cite{CH14, CH16, AH23}. The Airy line ensemble is 
a family of non-intersecting Brownian motions, the top line of which is the Airy process from
 \eqref{LPPAiry}.  The KPZ line ensemble is an ensemble of Brownian motions, which are exponential penalised when crossing each other. It emerges from the integrable O'Connell-Yor polymer 
 \cite{OY01, OC12}. 
 A characterisation of the Airy line ensemble has been proposed in  \cite{AH23} as a step towards
 universality considerations.
\item The {\it KPZ fixed point} \cite{MQR21}, which was established via integrability methods as 
a more general than \eqref{LPPAiry} space-time scaling limit of the totally asymmetric exclusion process (TASEP). The KPZ fixed
point is considered to determine a universality class, which is wider than the 
KPZ equation itself. Universal aspects of the KPZ fixed point were established in \cite{QS23}. There it was shown, via probabilistic methods, that TASEP with long range jumps has a KPZ fixed point limit.
\item The {\it Directed Landscape} \cite{DOV22}, which was established as a universal object describing
geodesic curves in last passage percolation. Universal aspects of it have been proposed in \cite{V20}.
\item Well-posedness of KPZ and weak universality. From the stochastic analysis 
point of view the well-posedness of the one-dimensional KPZ equation has been established
from a number of perspectives: the theory of regularity structures \cite{H13}, the theory of paracontrolled
distributions \cite{GIP15}, a renormalisation point of view \cite{K14} and the approach of energy 
solutions \cite{GJ14, GP18}. A framework of weak universality of KPZ, where the term $|\nabla h|^2$
is augmented by small perturbations was also treated in \cite{HQ18}
\item Weak universality at intermediate disorder regime. In \cite{AKQ14} it was established that 
the partition function of the directed polymer model converges in distribution, when the temperature
is suitably scaled with the size of the polymer, to the solution to the stochastic heat equation, irrespective
 the underlying disorder. Such type of temperature scaling is now known as {\it intermediate disorder scaling}. In dimension two, which is the {\it critical} dimension richer phenomena have shown to 
 emerge \cite{CSZ17a, CSZ23}, including certain phase transition. Moreover, in higher dimensions  
 our understanding of the weak-to-strong disorder regime has improved \cite{J22a,J22b, JL24a, JL24b} and further 
 developments seem that should be expected. Some of these aspects will occupy significant 
 parts of these notes.
\end{itemize}
\subsection{Outline of the paper.}
In Section \ref{sec:general}
 we will start with a review of the general framework of weak and strong disorder and its
martingale formulation. We will discuss the presence of the phase transition and provide 
the standard reformulation in terms of uniform integrability of the associated partition function.
We will then investigate relations to random walk pinning models, which provided some  first 
estimates on the critical temperature, and then proceed to more recent characterisations of it. Section 
\ref{sec:general} will close with an overview of the powerful {\it fractional moment method}.
Section \ref{sec:interm} will investigate hidden phase transitions at the critical dimension $2$ and the
framework of {\it intermediate disorder regime}. Section \ref{sec:heavy} will expose phase transitions and
conjectures when disorder has heavy tails. Finally, in Section \ref{sec:gengraph} we will make a very quick 
reference to directed polymers and their phase transitions on more general graph structures.

\section{Weak and strong disorder}\label{sec:general}

\subsection{The general framework and some basic results.}
Let us start by introducing some notation, complementary to \eqref{defW}, 
encoding when the polymer starts at different times and has fixed end point. To this end, we define, 
\begin{align}
	W_{(m,n)}^{\beta}(x,y) &:= \E_x \bigg[
	e^{\sum_{i=m+1}^{n-1} \{\beta \omega_{i,S_i} - \lambda(\beta)\}}
	\, \ind_{\{S_n = y\}}  \bigg]  \,, \qquad \text{and} \label{eq:Zab1} \\
	W_{(m,n]}^{\beta}(x,y) &:= \E_x \bigg[
	e^{\sum_{i=m+1}^{n} \{\beta \omega_{i,S_i} - \lambda(\beta)\}}
	\, \ind_{\{S_n = y\}}  \bigg]  \,,  \label{eq:Zab2}
\end{align}
with the convention that $\sum_{n=m+1}^{n-1} \{\ldots\} := 0$ for $n \le m+1$. 
We will also extend the above definition to $W_{(m,\infty)}$, if this quantity is well defined.
If there is only one spatial and/or 
temporal argument, then we will always refer to definition \eqref{defW}, i.e. consider a fixed starting but
free ending point. We will follow similar conventions  for the unnormalised partition function $Z^\beta$.
\vskip 2mm
 {\bf The martingales.} Consider the filtration 
$\cF_n=\sigma\big\{ \omega_{i,x} \colon i\leq n, x\in \bbZ^d\big\}$ for $n\geq 1$. It is easy to 
check that $W_n^\beta$ is a martingale with respect to $\cF_n$. Given that it is also non-negative,
it will have, by standard martingale theory, an a.s. limit as $N\to\infty$, which we denote by $W_\infty^\beta$. The events $\{W_\infty^\beta>0\}$ and $\{W_\infty^\beta=0\}$ 
belongs to the tail $\sigma$-field because whether the limit of $W_\infty^\beta$ is zero or not
clearly does not dependent on the value of any $\omega_{i,x}$ for any finite index $(i,x)$.
By Kolmogorov's 0-1 law, these events will have probability either equal to $1$ or equal to $0$. 
The dichotomy
between weak disorder (WD) and strong disorder (SD) is identified with the positivity or triviality
of $W_\infty^\beta$. In particular, we have
\begin{definition}\label{def:weak}
Let $W_\infty^\beta$ be the a.s. limit of $W_n^\beta$.
If $\bbP(W_\infty^\beta>0)=1$, we say that weak disorder {\rm (WD)}
takes place and if $\bbP(W_\infty^\beta=0)=1$, we have strong disorder {\rm (SD)}. 
\end{definition}

Uniform integrability of $(W_n^\beta)_{n\geq 1}$ implies weak disorder. This is 
straightforward since uniform integrability will allow to interchange limits and expectations and lead
to $\bbE[W_\infty^\beta]=1$.
It is an interesting fact that 
the converse is also true for the directed polymer partition function $W_n^\beta$ on $\Z^d$.
In particular, we have the following proposition:
\begin{proposition}\label{ui}
$(W_n^\beta)_{n\geq 1}$ is uniformly integrable if and only if the limit $W_\infty^\beta$ is a.s.
strictly positive.
\end{proposition}
\begin{proof}
If $(W_n^\beta)_{n\geq 1}$ is uniformly integrable, then we can interchange limit and expectation in
\begin{align*}
\bbE[W_\infty^\beta]=\lim_{N\to\infty} \bbE[W_N^\beta] =1,
\end{align*}
and therefore $W_\infty^\beta$ cannot be a.s. equal to $0$.

For the opposite claim we start by deriving a self-similar identity for $W_\infty^\beta$.
Using the Markov property and notation \eqref{eq:Zab2}, we derive the identity
\begin{align}\label{linear1}
W_{n+m}^\beta =\sum_{x\in \bbZ^d} W_n^\beta(0,x) W_{(n,m]}^\beta(x),
\end{align}
for any $n,m\in\bbN$
and taking the limit $m\to\infty$ we obtain that
\begin{align}\label{linear2}
W_{\infty}^\beta =\sum_{x\in \bbZ^d} W_n^\beta(0,x) W_{(n,\infty)}^\beta(x).
\end{align}
Let us now compute
\begin{align*}
\bbE\Big[ W_\infty^\beta \, \Big | \, \cF_n\Big] 
=\bbE\Big[ \sum_{x\in\bbZ^d} W_n^\beta(0,x) W_{(n,\infty)}^\beta(x)  \,\Big| \, \cF_n\Big]
= \bbE[\,W_\infty^\beta ]\sum_{x\in\bbZ^d}W_n^\beta(x) 
=  \bbE[\,W_\infty^\beta ] W_n^\beta ,
\end{align*}
Taking expectation with respect to disorder $\cF_n$ on both sides and then the limit $n\to\infty$ we have that
\begin{align*}
\bbE\Big[ W_\infty^\beta \Big] 
=  \bbE\Big[ \bbE[W_\infty^\beta] W_\infty^\beta \Big] 
=\bbE\Big[ W_\infty^\beta \Big]^2
\end{align*}
 and if  $W_\infty>0$ a.s. this implies that $\bbE[W_\infty^\beta]=1$. 
This, combined with the fact that $W_n^\beta$ is a non-negative martingale,
 implies uniform integrability, see \cite{D19}, Chapter 4, Theorem 5.2.
\end{proof}
\begin{remark}
{\rm
It is worth noting that the above equivalence between uniform integrability and positivity of the limit
of the martingale is a more general feature shared among martingales which have a ``linear'' structure
as that in \eqref{linear1} and \eqref{linear2}. Systems possessing such property often go by the name ``linear systems'' and they appear in branching and interacting particle systems,
 see \cite{L85}, Chapter IX.2. We should point out, though, that there are examples of directed polymer models
 on more exotic graph structures for which this equivalence fails \cite{CoSeZe21}.
 }
 \end{remark}
The dichotomy between weak and strong disorder is marked by a critical value of the temperature.
More precisely, we define
\begin{align}\label{criticalb}
\beta_c:=\sup\{ \beta>0 \colon W_\infty^\beta >0\}.
\end{align}
In other words, $ W_\infty^\beta >0$ for $\beta<\beta_c$ and
 $ W_\infty^\beta =0$ for $\beta>\beta_c$. The fact that this phase transition takes
 places only at a single value $\beta_c$ is a consequence of a monotonicity property. 
 More precisely, as we will show below, for $\gamma\in (0,1)$, 
 the function $\beta\to \bbE[(W_n^\beta)^\gamma]$ is non-increasing. 
 Then, since for $\gamma<1$, $(W_n^\beta)^\gamma$ is uniformly integrable, we have that 
 $\lim_{n\to \infty}\bbE[(W_n^\beta)^\gamma] =\bbE[(W_\infty^\beta)^\gamma] $ and by the 
 mentioned monotonicity,
 if $ \bbE[(W_\infty^\beta)^\gamma]=0$ (which would imply that $W_\infty^\beta=0$ a.s.)
 for some $\beta$, then it will also be zero for all larger ones.
 
 Let us now prove the monotonicity as part of a more general monotonicity property.
\begin{proposition}\label{prop:mon}
Consider a function $\phi\colon \bbR\to\bbR$. The function $\beta\to \bbE[\phi(W_n^\beta)]$
is non-decreasing if $\phi$ is convex and non-increasing if $\phi$ is concave.
\end{proposition}
\begin{proof}
We prove the statement for convex $\phi$. Let us start by differentiating:
\begin{align*}
\frac{\dd}{\dd \beta} \bbE[\phi(W_n^\beta)] 
&= \bbE\Big[ \E\big[ \sum_{i=1}^n (\omega_{i,S_i}-\lambda'(\beta)) e^{H_n^\beta(S) }\big] \, \phi'(W_n^\beta) \Big] \\
&= \E\Big[ \bbE\big[ \sum_{i=1}^n (\omega_{i,S_i}-\lambda'(\beta)) \phi'(W_n^\beta) 
\, e^{H_n^\beta(S) }\big] \,  \Big] .
\end{align*}
For a fixed path $S$ we define the distribution $\bbP_S$ on $\omega$ under which 
$(\omega_{i,x})_{i\in\bbN, x\in \bbZ^d}$ are independent with individual laws given by
\begin{align*}
\dd \bbP_S(\omega_{i,x})= \exp\Big\{\big(\beta\omega_{i,x}-\lambda(\beta)\big) \,\ind_{\{S_i=x\}} \Big\} \dd \bbP(\omega_{i,x}).
\end{align*}
This measure tilts the law of $\omega$ on sites visited by the random walk and, thus, 
provides a measure of the correlations between walk and disorder. Similar change
of measure will play an important role later in the notes.

Using the tilted measure we can write 
\begin{align*}
\bbE\big[ \sum_{i=1}^n (\omega_{i,S_i}-\lambda'(\beta)) \phi'(W_n^\beta) \, e^{H_n^\beta(S) }\big]
= \bbE_S \big[ \sum_{i=1}^n\sum_{x\in\bbZ^d} (\omega_{i,x}-\lambda'(\beta)) 
\ind_{\{S_i=x\}} \phi'(W_n^\beta) \big].
\end{align*}
Note that the function $\sum_{i=1}^n\sum_{x\in\bbZ^d} (\omega_{i,x}-\lambda'(\beta)) 
\ind_{\{S_i=x\}} $ is an increasing function in $\omega_{i,x}$ for every $(i,x)$
and so is $\phi'(W_n^\beta)$, if $\phi$ is assumed to be convex. Therefore by FKG
inequality (\cite{L85}, Chapter 2 or \cite{G99}, Section 2.2)
and given that $\bbP_S$ is still a product measure we have that
\begin{align*}
\bbE_S \big[ \sum_{i=1}^n\sum_{x\in\bbZ^d} (\omega_{i,x}-\lambda'(\beta)) \ind_{\{S_i=x\}} \phi'(W_n^\beta) \big]
\geq 
\bbE_S \big[ \sum_{i=1}^n\sum_{x\in\bbZ^d} (\omega_{i,x}-\lambda'(\beta)) \ind_{\{S_i=x\}}  \big] 
\cdot \bbE_S \big[ \phi'(W_n^\beta) \big],
\end{align*}
which is equal to $0$ since
\begin{align*}
\bbE \big[ \sum_{i=1}^n\sum_{x\in\bbZ^d} (\omega_{i,S_i}-\lambda'(\beta)) e^{H_n^\beta(S)} \big]
= \frac{\dd}{\dd \beta} \bbE \big[ e^{H_n^\beta(S)} \big] 
=0
\end{align*}
since $\bbE \big[ e^{H_n^\beta(S)} \big]=1$. Therefore, $\frac{\dd}{\dd \beta} \bbE[\phi(W_n^\beta)] \geq 0$, and the monotonicity follows.
\end{proof}
Characterising the critical value $\beta_c$ has been a problem of major interest since the onset 
of the directed polymer model but progress has been slow. However, there has been some important
developments, which we will describe in the next sections.

The equivalence between weak disorder and uniform integrability provides a path to provide
some information on $\beta_c$. Uniform integrability can be provided through boundedness
of moments higher than 1 of $W_n^\beta$. The easiest case is to estimate the second moment
and in fact this is the regime that works \cite{IS88, B89, S95, SZ96} considered. Let us state:
\begin{theorem}\label{thm:L2regime}
In dimensions $d\geq3$, there exist $\beta_2>0$ such that 
$\sup_{n\geq 1} \bbE[W_n^2]<\infty$ for $\beta<\beta_2$. In particular, this implies
uniform integrability for $\beta<\beta_2$ and, therefore, that $\beta_c>0$. 
\end{theorem}
The proof of this result, which is essentially the original approach of Imbrie-Spencer \cite{IS88},
 is simple but it contains some elements, including notions of expansions,
whose ramifications will be very useful later on, in particular in Section \ref{sec:interm}.
\begin{proof}
The second moment of $W_n^\beta$ can be easily computed. More precisely, if $S,S'$ are two
independent simple random walks and $\E^{S,S'}$ denotes the joint expectation, 
then we have by simple Fubini that
\begin{align}\label{eq:secondmom}
\bbE[(W_n^\beta)^2] 
&= \bbE \Big[ \E^{S,S'}\big[ e^{\sum_{i=1}^n ( \beta\omega_{i,S_i} + \beta \omega_{i,S_i'} -2\lambda(\beta) )}\big]\Big] \notag\\
&=\E^{S,S'}\Big[ e^{\lambda_2(\beta) \sum_{i=1}^n \ind_{\{S_i=S_i'\}} } \Big],
\end{align}
where we have denoted by
\begin{align}\label{l2}
\lambda_2(\beta):=\lambda(2\beta)-2\lambda(\beta).
\end{align}
We can deduce that 
\begin{align*}
\sup_{n\geq 1} \bbE[(W_n^\beta)^2] = \E^{S,S'}\Big[ e^{\lambda_2(\beta) \sum_{i=1}^\infty \ind_{\{S_i=S_i'\}} } \Big].
\end{align*}
Denoting the right-hand side by $\mathfrak{E}$ and $\tau:=\inf\{n\geq 1 \colon S_n=S_n'\}$ and using the strong Markov property
we have that 
\begin{align*}
\mathfrak{E}=1+\P^{S,S'}(\tau<\infty) \, e^{\lambda_2(\beta)} \, \mathfrak{E},
\end{align*}
from which it follows that $\mathfrak{E}<\infty$ if $\P^{S,S'}(\tau<\infty) \, e^{\lambda_2(\beta)}<1$. Since $\tau$ has the same distribution 
as that of the return time to $0$ of a single random walk, we have that 
$\P^{S,S'}(\tau=\infty)=\P(S \, \text{does {\it not} return to zero})=:\pi_d$, which is strictly positive in $d\geq 3$. 
Therefore, $\P^{S,S'}(\tau<\infty) =1-\pi_d <1$ and, thus, there exists $\beta>0$ such that $\mathfrak{E}<\infty$. In particular,
this is the case if $e^{\lambda_2(\beta)}<\frac{1}{1-\pi_d}$.
\end{proof}
The question of whether $\beta_2<\beta_c$, i.e. whether the $L^2$ regime coincides with the weak
disorder regime concerned the physics literature \cite{ED92, MG06a, MG06b}. 
We now know that $\beta_2<\beta_c$ in $d\geq 3$
\cite{BGH10, BS10, BT10, BS11} but we will detail on this in the next section. 
We will also discuss recent characterisations
of the weak disorder regime via moments in Section \ref{sec:momchar}. 
For the moment let us record the central
limit behaviour of the polymer path in the weak disorder regime.
\begin{theorem}
Assume weak disorder, that is, $\beta<\beta_c$. Then for every bounded, continuous
$f\colon \bbR^d\to\bbR$ it holds that
\begin{align*}
\sum_{x\in\bbZ^d} f(x/\sqrt{n}) \,\mu_n^\beta(x) \xrightarrow[n\to\infty]{\bbP} \bE[f(\cN_d)],
\end{align*}
where the convergence is in probability and 
 $\cN_d$ is a multivariate normal variable on $\bbR^d$ with covariance matrix $\frac{1}{d} \sfI$.
 Moreover, if $\beta<\beta_2$ the above convergence is a.s.
\end{theorem}
The a.s. convergence in the $L^2$ regime is the original result of Bolthausen \cite{B89}, while 
the extension of the central limit behaviour in the whole weak disorder regime (even if we do not
know $\beta_c$ itself) was proved by Comets-Yoshida \cite{CY06}. 
An extension to an a.s. central limit theorem in the whole weak disorder regime 
 for a Brownian disorder has been proved by Junk in \cite{J23b}. 

Let us comment on the existence of a strong disorder regime and a formulation of localisation.
We have
\begin{theorem}\label{thm:SD}
In dimensions $d=1,2$ strong disorder takes place for all $\beta>0$, that is $\beta_c=0$.
Moreover, strong disorder takes place also in $d\geq 3$, if $\beta$ is large enough so that
$\beta\lambda'(\beta)-\lambda(\beta)>\log (2 d)$.
\end{theorem}
This result was proved in \cite{CH02, CSY03} (see also \cite{CSY04}). We will refer there for
 the details of the proof. Here we will give a sketch of the main points. In particular we want to 
 highlight the estimates on fractional moments - an approach that has, more recently,
  been developed significantly 
and has been very fruitful. The fully developed 
fractional moment method will be detailed in Section \ref{sec:frac}. 
Before getting to the sketch of the proof of the above theorem,
let us comment that the set of $\beta$ such that $\beta\lambda'(\beta)-\lambda(\beta)>\log (2d)$ is empty
 when the disorder $\omega$ has finite support and gives mass $p$ to the right boundary of its support such that $\log\frac{1}{p}<\log (2d)$. 
\begin{proof}[Proof (main points) of Theorem \ref{thm:SD}]
A first main point is to establish the existence of numbers $\alpha_n$ diverging to infinity such that
\begin{align}\label{fraclog}
\limsup_{n\to\infty} \frac{1}{\alpha_n} \log \bbE[(W_n^\beta)^\gamma] <0,
\end{align}
for some fractional power $\gamma\in (0,1)$. This would imply that $W_n$ converges to zero in probability and 
since we know that $W_n$ converges a.s., the a.s. convergence to $0$ follows.
The second main point is then to estimate the
fractional moments to obtain the above limiting behaviour.  
The use of fractional moment will become a recurring theme in these notes.

The first step here is to
use the linearity \eqref{linear1} of the partition function and the elementary inequality 
$(\sum_{i=1}^k a_i)^\gamma \leq  \sum_{i=1}^k a_i^\gamma $ for $\gamma\in (0,1)$
and $a_i$ non-negative numbers, to obtain that
\begin{align*}
(W_n^\beta)^\gamma
&=\Big(
\frac{1}{2d}\sum_{x\colon |x|=1} e^{\beta\omega_{1,x}-\lambda(\beta)} \, W^\beta_{(1,n]}(x)
\Big)^\gamma \\
&\leq \frac{1}{(2d)^\gamma}\sum_{x\colon |x|=1} 
e^{\gamma\beta\omega_{1,x}-\gamma\lambda(\beta)} \, \big(W^\beta_{(1,n]}(x)\big)^\gamma ,
\end{align*}
and by taking expectation this leads to 
\begin{align*}
\bbE\big[ (W_n^\beta)^\gamma \big] \leq r(\gamma) \bbE\big[ (W_{n-1}^\beta)^\gamma\big],
\quad \text{with} \quad r(\gamma):= (2d)^{1-\gamma} e^{\lambda(\gamma\beta)-\gamma\lambda(\beta)}.
\end{align*}
Condition $\beta\lambda'(\beta)-\lambda(\beta) > \log (2d)$ implies that 
$\frac{\dd r(\gamma)}{\dd \gamma}\Big|_{\gamma=1}>0$ and this, together with the log-convexity
of $r(\gamma)$ and its continuous differentiability, implies the existence of a $\gamma\in (0,1)$,
such that $r(\gamma)<1$. Then, in this case, \eqref{fraclog} is satisfied with a choice of 
$\alpha_n=n$.

We refer on how to extend this method to all $\beta>0$ in dimensions one and two to \cite{CSY03, CSY04}.
\end{proof}
We now mention the first reference to localisation as this was established in \cite{CH02, CSY03}
(see also \cite{CSY04})
\begin{theorem}\label{thm:vsdloc}
Let 
\begin{align}\label{overlap}
I_n:=\sum_{x\in\bbZ^d} \mu_{n-1}^\beta (S_n=x)^2
\end{align}
then a.s. we have that
\begin{align}\label{SDoverlap}
\big\{ W_\infty^\beta =0 \big\} = 
\Big\{ \sum_{n\geq 1} I_n =\infty\Big\}
\end{align}
Moreover, if strong disorder holds, then there exist constants $c_1,c_2 \in (0,\infty)$ such that a.s.
\begin{align}\label{logoverlap}
-c_1\log W_n^\beta  \leq \sum_{k=1}^n I_k \leq -c_2 \log W_n^\beta,
\end{align}
for all large $n$.
\end{theorem}
Let us observe that $I_n$ can also be written in the form 
$I_n=\mu_{n-1}^{\beta, \otimes 2}(S_n=S_n')$, where
\begin{align*}
\mu_{n-1}^{\beta, \otimes 2}(S,S'):=\frac{1}{(W_{n-1}^\beta)^2} \,
e^{\sum_{i=1}^{n-1} (\beta \omega_{i,S_i} +\beta\omega_{i,S'_{i}} -2\lambda(\beta)) }\, \P(S)\P(S')
\end{align*}
is the joint law of two polymer paths run in the same environment $\omega$.
In other words, $I_n$ is the probability that two independent 
polymer paths, which run in the random environment $\omega$, meet at time $n$. 
Relation \eqref{SDoverlap} says that in the strong disorder regime, the  number
of times that two independent polymers in the same environment meet is infinite. This is
in contrast to what would happen if considering two independent simple random walks in $d\geq 3$.
This has a flavour of localisation as the two independent paths will tend to meet at sites at the 
vicinity of which disorder is favourable.
Relation \eqref{logoverlap} can provide a stronger localisation statement if one could
show that $\lim_{n\to\infty}\frac{1}{n} \log W_n^\beta <0$. This would mean that 
the C\'esaro means limit of $I_n$ is positive, which
could be interpreted (but not directly implying) that two independent
polymers will actually meet a positive fraction of times. This motivates the following definition
\begin{definition}\label{def:vSD}
We say that very strong disorder {\rm (vSD)} holds if
\begin{align}\label{eq:vSD}
\sff(\beta):=\lim_{n\to\infty}\frac{1}{n}\log W_n^\beta <0.
\end{align}
\end{definition}
The fact that the limit in \eqref{eq:vSD} exists and that it is, in fact, self-averaging i.e. equal to 
$\lim_{n\to\infty}\frac{1}{n}\bbE\log W_n^\beta$, was established in \cite{CH02,CSY03}.
In  \cite{CV06} it was established that in $d=1$ very strong disorder holds for all $\beta>0$.
In \cite{L10} based on an upscaled {\it fractional moment method} it was proved that very strong
disorder also holds in dimension $2$. Moreover, asymptotics on $\sff(\beta)$ for $\beta$ small 
were derived. 
The question of whether strong disorder coincides with very strong disorder in dimensions $d\geq 3$, equivalently 
whether the decay of $W_n$ to $0$ always happens at an exponential rate, was a longstanding question. In a recent breakthrough
Junk-Lacoin \cite{JL24a} have resolved this question in the case of bounded disorder:
\begin{theorem}[\cite{JL24a}]\label{thm:JL}
In $d\geq 3$ and for bounded disorder $\omega$, strong disorder and very strong disorder are equivalent.
Furthermore, weak disorder holds at the critical temperature $\beta_c$.
\end{theorem}
One would expect that the boundedness assumption is not necessary but let us state this as a question:
\begin{question}
Are strong and very strong disorder equivalent for unbounded disorder $\omega$ ?
\end{question} 
We will further discuss the points of strong and very strong disorder in Section \ref{sec:frac}.

Under the assumption of very strong disorder and additionally that $\sff(\beta)$ is differentiable and $\sff'(\beta)< 0$,
Bates \cite{B19,B21} and Bates-Chatterjee \cite{BC20} have given quantitative estimates on localisation in terms
of overlaps and in terms of C\'esaro means of end-point distributions \cite{BC20b}. In particular, they have identified the existence of a finite number of trajectories with which a directed polymer
path will have a positive fraction of intersections. We refer to \cite{B19,BC20, B21} for details. The
relation between the derivative of the free energy and localisation properties of the polymer was first observed  
Carmona-Hu \cite{CH02}. In relation to the discussion above about strong and very strong disorder we can also pose the question:
\begin{question}
Does $\sff'(\beta)<0$ hold in the whole strong disorder regime ? 
\end{question}
The most detailed result on localisation is probably due to Comets-Nguyen \cite{CN16} for the exactly solvable, one-dimensional,
  log-gamma polymer. Without presenting the specifics of the log-gamma setting (as this will require a digression) let us state (in rough terms) the main content of the Comets-Nguyen localisation theorem (we refer to \cite{CN16} for the details and some of the
  terminology):
  \begin{theorem}
  For the one-dimensional, stationary log-gamma polymer measure $\mu_n^{\rm log-gamma}$, 
  with equilibrium boundary condition, denote by
$l_n:={\rm argmax} \{ \mu_n^{\rm log-gamma}(x) \colon x\in \Z  \} $. Then
\begin{align*}
\frac{l_n}{n}\xrightarrow[n\to\infty]{d} {\rm argmax}_{t\in[0,1]} B_t,
 \end{align*}
 where $B$ is a one-dimensional Brownian motion with certain diffusion coefficient depending on the parameters of the log-gamma
 polymer.
\end{theorem}

\subsection{Random walk pinning model and non-coincidence of $\beta_2$ and $\beta_c$.}
\label{sec:RWpin}
The question of whether the $L^2$ regime characterises the weak disorder regime has been 
a subject of interest since the early directed polymer era. In this section we will see that this is not
the case in $d\geq 3$, through linking to intersection properties of random walks. In particular, in the
second part of this section we will introduce an interesting link to a study of intersections of 
random walks, where one path trajectory is considered fixed. This will be the {\it random walk pinning
model}.

Before proceeding, let us record:
\begin{align}\label{def:beta2}
\beta_2:=\sup\{\beta\colon \sup_{n\geq 1} \bbE[(W_n^\beta)^2]<\infty \}.
\end{align}
Probably the first study of the relation 
between the $L^2$ and weak disroder regimes is that of Evans-Derrida 
\cite{ED92}. There they attempted to estimate fractional moments of 
$(W_n^\beta)^{2\gamma}$ and obtained the following result:
\begin{theorem}\label{thm:ED}
Let $S,S'$ be two independent, simple random walks on $\bbZ^d$ and denote by $P^{S,S'}$ their
joint law. Weak disorder holds for $\beta$ such that
\begin{align}\label{eq:EDcond}
e^{\lambda(2\gamma \beta)-2\gamma\lambda(\beta)} 
\sum_{n\geq 1,\, x\in \bbZ^d} 
\P^{S,S'}(\text{$S$ and $S'$ meet for the {\it first time} at site $(n,x)$} )^\gamma
<1,
\end{align}
for some $\gamma<1$.
\end{theorem}
\begin{proof}
The idea is to estimate moments of $W_n^\beta$ between $1$ and $2$.
To this end we start by writing
\begin{align*}
\bbE[(W_n^\beta)^{2\gamma}] 
&= \bbE\Big[ \Big( \E^{S,S'}\big[ e^{\sum_{i=1}^n(\beta\omega_{i,S_i} +\beta\omega_{i,S'_i}-2\lambda(\beta) )}  \big] \Big)^\gamma \Big] .
\end{align*}
Next, we identify the set of space-time points where $S,S'$ intersect. 
For $(n_1,x_1),...,(n_k,x_k) \in \bbN\times\bbZ^d$, we denote by 
\begin{align*}
\sfA_{(n_1,x_1),...,(n_k,x_k)} :=\big\{ S\cap S' =\{ (n_1,x_1),...,(n_k,x_k) \}\big\},
\end{align*}
and by 
\begin{align*}
\sfA_{(n_{j-1},x_{j-1})}^{(n_j,x_j)}:= \{ 
S_{n_{j-1}}=S'_{n_{j-1}}=x_{j-1}, \, 
S_{n_{j}}=S'_{n_{j}}=x_{j},\,\,
\text{and $S_i\neq S'_i$ for $n_{j-1}<i<n_j$ } \},
\end{align*}
and write
\begin{align*}
\bbE[(W_n^\beta)^{2\gamma}] 
&=\bbE\Big[ \Big( \sum_{k\geq 0} \sumtwo{0<n_1<\cdots<n_k\leq n}{x_1,...,x_k\in\bbZ^d} 
\E^{S,S'}\big[ e^{\sum_{i=1}^n(\beta\omega_{i,S_i} +\beta\omega_{i,S'_i}-2\lambda(\beta) )}
\,;\, \sfA_{(n_1,x_1),...,(n_k,x_k)} 
 \big]   \Big)^\gamma \Big] .
 \end{align*}
 Using, again, the fractional moment inequality 
 $(\sum_i a_i)^\gamma \leq \sum_i a_i^\gamma$, for $\gamma\in (0,1)$, we have
 \begin{align*}
 &\bbE[(W_n^\beta)^{2\gamma}] \\
&  \leq \sum_{k\geq 0} \sumtwo{0<n_1<\cdots<n_k\leq n}{x_1,...,x_k\in\bbZ^d} 
\bbE\Big[ \E^{S,S'}\big[ e^{\sum_{i=1}^n(\beta\omega_{i,S_i} +\beta\omega_{i,S'_i}-2\lambda(\beta) )}
\,;\, \sfA_{(n_1,x_1),...,(n_k,x_k)}  \big]^\gamma \Big] \\
& = \sum_{k\geq 0} \sumtwo{0<n_1<\cdots<n_k\leq n}{x_1,...,x_k\in\bbZ^d} 
 \prod_{j=1}^k \bbE\Big[
\E^{S,S'}_{(n_{j-1}, x_{j-1})} 
\big[ e^{\sum_{i={n_{j-1}+1}}^{n_j}
(\beta\omega_{i,S_i} +\beta\omega_{i,S'_i}-2\lambda(\beta) )}  \, ;\,
 \sfA_{(n_{j-1},x_{j-1})}^{(n_j,x_j)}  \big]^\gamma \Big].
\end{align*}
Singling out from $E^{S,S'}$ the disorder $e^{2\beta \omega_{n_j,x_j}-2\lambda(\beta)}$ 
on the sites where the two walks meet, we can write the above as
\begin{align}\label{eq:ED1}
& \sum_{k\geq 0} \sumtwo{0<n_1<\cdots<n_k\leq n}{x_1,...,x_k\in\bbZ^d} 
 \prod_{j=1}^k \bbE\Big[
e ^{2\gamma\beta \omega_{n_j,x_j}-2\gamma\lambda(\beta)} 
\E^{S,S'}_{(n_{j-1}, x_{j-1})} \big[ e^{\sum_{i={n_{j-1}+1}}^{n_j-1}
(\beta\omega_{i,S_i} +\beta\omega_{i,S'_i}-2\lambda(\beta) )} \,;\,
 \sfA_{(n_{j-1},x_{j-1})}^{(n_j,x_j)}  \big]^\gamma \Big] 
\end{align}
Now compute  $\bbE e ^{2\gamma\beta \omega_{n_j,x_j}-2\gamma\lambda(\beta)} 
= e^{\lambda(2\gamma\beta) -2\gamma\lambda(\beta)}$ 
and we use H\"older in 
$\bbE \big[ \E^{S,S'}_{(n_{j-1}, x_{j-1})} \big[\cdots\big]^\gamma \big]$
and the fact that 
\begin{align*}
&\bbE \big[ \E^{S,S'}_{(n_{j-1}, x_{j-1})} \big[e^{\sum_{i={n_{j-1}+1}}^{n_j-1}
(\beta\omega_{i,S_i} +\beta\omega_{i,S'_i}-2\lambda(\beta) )} 
 \ind_{\{S_i\neq S'_i,\,\,\, \text{for $n_{j-1}< i < n_j \, 
\, S_{n_j}=S'_{n_j}=x_j$\}}} \big] \Big]^\gamma \\
&= \P^{S,S'}_{(n_{j-1},x_{j-1})}\big( S_i\neq S'_i,\,\,\, \text{for $n_{j-1}< i < n_j \, 
\, S_{n_j}=S'_{n_j}=x_j$ }\big)^\gamma
\end{align*}
to obtain that \eqref{eq:ED1} is bounded by
\begin{align*}
\sum_{k\geq 0} \sumtwo{0<n_1<\cdots<n_k\leq n}{x_1,...,x_k\in\bbZ^d} 
 \prod_{j=1}^k  e^{\lambda(2\gamma\beta)-2\gamma\lambda(\beta)}  
\,  \P^{S,S'}_{(n_{j-1},x_{j-1})}\big( S_i\neq S'_i,\,\,\, \text{for $n_{j-1}< i < n_j \, 
\, S_{n_j}=S'_{n_j}=x_j$}\big)^\gamma.
\end{align*}
Letting $n\to\infty$ we can decouple the spatial and temporal summations and have that in the limit
$n\to \infty$ the above equals
\begin{align*}
\sum_{k\geq 0} \Big( \sum_{n\geq 1, x\in \bbZ^d} 
  e^{\lambda(2\gamma\beta)-2\gamma\lambda(\beta)}  
\,  \P^{S,S'}\big( S_i\neq S'_i,\,\,\, \text{for $0< i < n \, 
\, S_{n}=S'_{n}=x$}\big)^\gamma \Big)^k,
\end{align*}
which is finite if condition \eqref{eq:EDcond} is satisfied.
\end{proof}
Evans and Derrida in \cite{ED92} provided numerical evidence that condition \eqref{eq:EDcond}
extends the $L^2$ regime. There has also been further discussion in the physics literature 
\cite{MG06a, MG06b} as to whether the Evans-Derrida condition provides a good estimate on 
$\beta_c$. This is still open but let us remark that just a crude use of fractional inequality
and H\"older without taking into account the correlations between the walk and the disorder 
is known in several circumstances not to be sharp. 
Below we outline another approach to the comparison between $\beta_2$ and $\beta_c$
that was put forward in \cite{B04, BS10, BS11, BT10}. 
This approach share some features to the approach of Evans-Derrida but it more refined.
 In particular, it also considers the intersections of two independent random walks, with a reward 
 assigned each time the walks intersect, but, crucially, the framework is in a ``quenched'' setting. 
 This means that averaging is taking place with respect to only one of the two walks, while the other is
 fixed and one is interested in determining the growth of a (certain) partition for a.e. trajectory of the fixed
 random walk. Proceeding with such estimates, they use the full power of the refined moment method, which we will describe in Section \ref{sec:frac}. Below, let us 
 outline the main framework of the approach of \cite{B04, BS10, BS11, BT10}.
\vskip 2mm
{\bf Random walk pinning model.}
For a simple random walk $S$ denote by $\bbP_S$ the law on $\omega=(\omega_{n,x})_{n\geq 0, x\in \bbZ^d}$ under which 
$(\omega_{n,x})_{n\geq 0, x\in \bbZ^d}$ are independent and with a law for each individual $\omega_{n,x}$ given by
\begin{align}\label{sizebiasO}
\dd \bbP_S(\omega_{n,x})= \exp\Big\{\beta\omega_{n,x}-\lambda(\beta)) \,\ind_{\{S_n=x\}} \Big\} \dd \bbP(\omega_{n,x}).
\end{align}
In other words, the environment admits an exponential tilt on the sites which are visited by the
trajectory of the random walk $S$. 
Note that we have already encountered this random walk - tilted measure in the proof of the monotonicity properties of the partition functions in Proposition \ref{prop:mon}.
For a given random walk trajectory $S$,
we will also denote the realisation of the disorder sampled from \eqref{sizebiasO} by 
$\omega^S:=(\omega^S_{n,x})_{n\geq 0, x\in \bbZ^d}$. The partition function in a disorder tilted along $S$, will be denoted by 
\begin{align}\label{sizebiasZ}
W_n^{\beta, S} 
&:=\E^{S'}\Big[ e^{ \sum_{i=1}^n ( \beta \omega^{S}_{i,S'_i}   -\lambda(\beta) )} \Big] \notag \\
&= \E^{S'}\Big[ e^{ \sum_{i=1}^n\big( \beta \omega_{i,S'_i} \ind_{\{S'_i\neq S_i\}} + \beta \omega^S_{i,S'_i} \ind_{\{S'_i=S_i\}}  -\lambda(\beta) \big)} \Big]
\end{align}
where in the right-hand-side the expectation is with respect to the random walk $S'$, which is independent of $S$. We also used the fact that on sites $(i,x)$ not visited by $S$, we have
$\omega^S_{(i,x)}=\omega_{(i,x)}$.

For a random variable $X$ with law $\bbP$ and mean equal to $1$, we call the distribution $\dd \bbP_X:=X\dd \bbP$ the {\it size-bias law} of $X$.
We will be able to relate the size-bias law of $W_n^\beta$ to the law of $W_n^{\beta,S}$ as the following simple computation shows:
\begin{align}\label{sizebiasZ2}
\bbE[W_n^\beta f(W_n^\beta)] 
&= \bbE \Big[ \E\big[ e^{\sum_{i=1}^n (\beta \omega_{i,S_i} - \lambda(\beta)) }\big] \, f(W_n^\beta) \Big]
\notag \\
&= \E\Big[ \bbE\big[ e^{\sum_{i=1}^n (\beta \omega_{i,S_i} - \lambda(\beta)) }\big] \, f(W_n^\beta) \Big] 
\notag \\
&=\E\Big[ \bbE_{S}\big[ \, f(W_n^{\beta, S}) \big] \Big] \notag \\
&=:\E\otimes \bbE_{S}\big[ \, f(W_n^{\beta, S}) \big].
\end{align} 
The above identity can then be interpreted as that
\begin{align*}
\text{the size-bias law of $W_n^\beta$ is identical to the law of $W_n^{\beta, S}$ 
under the measure $\P\otimes \bbP_S$.}
\end{align*}
The significance of this fact, with regards to weak disorder, is that it relates, 
via choosing functions $f$ that grow to infinity, the uniform integrability
of $W_n^\beta$ under $\bbP$  to the tightness of $W_n^{\beta, S}$ under $\P\otimes \bbP_S$.
Using this fact, Birkner-Sun \cite{BS10,BS11} were able to show that the weak disorder regime extends beyond the $L^2$ regime in dimensions $d\geq 3$. A similar result was also derived
in $d=3$ by Berger-Toninelli \cite{BT10} (see also \cite{JL24a}, Theorem B).
In order to state the result, we need to introduce the {\it random walk pinning} partition function:
Given a path of a simple random walk $S$, define
\begin{align}\label{RWpin}
W_n^{\beta, {\rm pin}(S)}
:=\E^{S'}\Big[ e^{\big(\lambda(2\beta)-2\lambda(\beta)\big)\sum_{i=1}^n \ind_{\{S_i=S_i'\}}} \Big].
\end{align}
The random walk pinning partition arises from $W_n^{\beta, S}$ by integrating out the disorder, with respect to $\bbP_S$ 
 (again, the random walk trajectory $S$ is fixed), as: 
\begin{align}\label{RWpinemerge}
\bbE_S\big[W_n^{\beta, S}\big]
&= \bbE_S\Big[ \E^{S'} \Big[ e^{ \sum_{i=1}^n\big( \beta \omega_{i,S'_i}
 -\lambda(\beta) \big)} \Big] \Big] \notag \\
&= \E^{S'}\Big[ \bbE_S \Big[ e^{ \sum_{i=1}^n\big( \beta \omega_{i,S'_i} \ind_{\{S'_i\neq S_i\}} + \beta \omega^S_{i,S'_i} \ind_{\{S'_i=S_i\}}  -\lambda(\beta) \big)} \Big] \Big] \notag \\
&= \E^{S'}\Big[ \bbE_S \Big[ e^{ \sum_{i=1}^n\big( \beta \omega^S_{i,S'_i}  
 -\lambda(\beta) \big)\ind_{\{S'_i=S_i\}}} \Big] \Big] \notag \\
&=\E^{S'}\Big[ e^{\big(\lambda(2\beta)-2\lambda(\beta)\big)\sum_{i=1}^n \ind_{\{S_i=S_i'\}}} \Big] = W_n^{\beta, {\rm pin}(S)},
\end{align}
where in the third equality we used that for sites $(i,S'_i)$, which are not visited by $S$, we
have that $\bbE_S\big[e^{\beta\omega_{i,S'_i} -\lambda(\beta)}\big] =1 $, while in the fourth that
on sites $(i,S'_i)$, which are visited by $S$, we have 
$\bbE_S\big[e^{(\beta\omega^S_{i,S'_i} -\lambda(\beta))\ind_{S_i=S'_i}}\big] =
\bbE\big[e^{(2\beta\omega_{i,S'_i} -2\lambda(\beta) ) \ind_{\{S_i=S'_i}\} }\big] = 
e^{(\lambda(2\beta)-2\lambda(\beta))\ind_{\{S_i=S'_i\}}}$.

We note that averaging $W_n^{\beta,{\rm pin}(S)}$ with respect to $S$ gives
\begin{align*}
\E^S\big[ W_n^{\beta,{\rm pin}(S)}\big] 
= \E^{S,S'}\Big[ e^{\big(\lambda(2\beta)-2\lambda(\beta)\big)\sum_{i=1}^n \ind_{\{S_i=S_i'\}}} \Big] ,
\end{align*}
which coincides with $\bbE[(W_n^\beta)^2]$. 
Averaging with respect to some form of disorder is often called {\it annealing} and, therefore, the above indicates that 
the annealed critical temperature of the random walk pinning partition
coincides with the $L^2$ critical temperature $\beta_2$ of $W_n^\beta$. The results of Birkner-Sun \cite{BS10} for $d\geq 4$ and of \cite{BT10, BS11} for $d=3$ are summarised in
the following theorem:
\begin{theorem}\label{thm:BS}
Let $d\geq 3$, $\beta_2$ be the $L^2$ critical temperature as in \eqref{def:beta2} and  
\begin{align}\label{RWcrit}
\beta_c^{\rm RWpin}:= \sup\big\{\beta \colon \sup_{n\geq 1} W_n^{\beta, {\rm pin(S)}} < \infty, \quad \text{for $a.e.$ random walk $S$} \big\},
\end{align}
then $\beta_2< \beta_c^{\rm RWpin}$. 
\end{theorem} 
Given the discussion around \eqref{sizebiasZ2}, we have that $\beta_c^{\rm RWpin} \leq \beta_c$ and the above result
shows that $\beta_2< \beta_c^{\rm RWpin}\leq \beta_c$, thus the weak disorder regime extends beyond the $L^2$ regime.
The fact that the weak disorder regime extends beyond the $L^2$ regime in dimensions $d\geq 5$ was earlier established
by Birkner-Greven-den Hollander \cite{BGH10} (see also \cite{BGH23}), 
but in a less quantitative form
and via different methods, based on large deviations. 
The approaches of both Birkner-Sun and Berger-Toninelli relied on the fractional moment method. In fact, they were able to establish that
for a $\gamma\in (0,1)$ it holds that
\begin{align}\label{BSfrac}
\sup_{n\geq 1}  \E^S\big[ (W_n^{\beta, {\rm pin}(S)})^\gamma \big] <\infty, \qquad \text{for a.e. realisation of $S$.}
\end{align}
Taking $f(x)=x^\gamma$ in \eqref{sizebiasZ2} and via a simple Jensen we obtain via \eqref{sizebiasZ2} and \eqref{RWpinemerge}
that
\begin{align*}
\bbE\big[ \big( W_n^{\beta}\big)^{1+\gamma} \big] 
&=\bbE\big[ W_n^{\beta} \cdot \big( W_n^{\beta}\big)^{\gamma} \big] 
= \E^S\otimes \bbE_S \big[ \big( W_n^{\beta}\big)^{\gamma}\big] \\
&\leq \E^S\big[ \big( \bbE_S W_n^{\beta}\big)^{\gamma}\big]
=  \E^S\big[ (W_n^{\beta, {\rm pin}(S)})^\gamma \big] <\infty
\end{align*}
for $\beta<\beta_c^{\rm RWpin}$.
Therefore, the relation to the random walk pinning model already provides a regime of $\beta$ beyond the $L^2$,
where $W_n^\beta$ has moments strictly greater than $1$.

Birkner-Sun \cite{BS10, BS11} were also able to give quantitative bounds on the gap between the random walk pinning critical temperature and the $L^2$ one in the case where the
simple, $d$-dimensional  random walks $S,S'$ are replaced by continuous, $d$-dimensional random walks $\cS^{(1)}, \cS^{(\rho)}$
with jump intensities $1$ and $\rho>0$, respectively, and identical, symmetric and irreducible transition probabilities. In particular, if $\tilde\beta_c^{\,\rm RWpin}$ and $\tilde\beta_2$ are the 
corresponding critical temperatures in this continuous setting, then 
\begin{align*}
\tilde\beta_c^{\,\rm RWpin} - \tilde\beta_2 \geq 
\begin{cases}
a\rho, &\qquad \text{if $d\geq 5$ and for some constant $a$}, \\
a_\delta \rho^{1+\delta}, &\qquad \text{if $d=4$ and for any $\delta>0$ and suitable constant $a_\delta$}, \\
e^{-a'_\delta \rho^{-\delta}},  &\qquad \text{if $d=3$ and for any $\delta>2$ and suitable constant $a'_\delta$}.
\end{cases}
\end{align*}
\begin{question}
Can one obtain upper bounds matching the lower bounds ? 
\end{question}

\subsection{Moment characterisation of the weak disorder regime}\label{sec:momchar}
The understanding of the weak disorder regime beyond the $L^2$ regime received substantial 
boost recently via the very insightful works of Junk \cite{J22a, J22b, J23}, Fukushima-Junk \cite{FJ23} and 
more recently Junk-Lacoin \cite{JL24a, JL24b}.
What was first
shown in \cite{J22a} was that in $d\geq 3$, for any $0<\beta<\beta_c$  the partition function $W_n^\beta$ admits
moments strictly greater than $1$, which are obviously strictly less than $2$ if $\beta>\beta_2$.

Before going into details let us define:
\begin{align}\label{pstar}
p^*(\beta):=\sup\big\{ p\geq 1 \colon \sup_n \bbE\big[ (Z_{n}^{\beta})^p\big] < \infty \big\}.
\end{align}
The original result of Junk \cite{J22a} was, in rough terms and under certain assumptions,
 that for every $\beta<\beta_c$, we have $p^*(\beta)>1$. This result opened new paths in the
 understanding of properties of the directed polymer model beyond the $L^2$ regime. Some of these
have already been explored in \cite{J22b,J23, JL24a, JL24b} but the understanding of the 
moment properties of the partition function offers tools for further exploration.  

The results in this section will be based on one of the following two assumptions.
\begin{assumption}[upper-bounded]\label{uBd}
Disorder $\omega$ is upper bounded. That is, there exists $M>0$, such $\bbP(\omega_{i,x}>M)=0$
for every $(i,x)\in \bbN \times \bbZ^d$. 
\end{assumption}
 
\begin{assumption}[controlled overshoot]\label{overshoot}
For all $\beta>0$, there exist $A_1=A_1(\beta)>1$ and $C=C(\beta)>0$ such that for all $A>A_1$
\begin{align}\label{eq:overshoot}
\bbE\big[ e^{\beta \omega} \,\big| \, \omega >A \big] \leq C e^{\beta A}.
\end{align}
\end{assumption}
Clearly, Assumption \ref{overshoot} implies Assumption \ref{uBd} but some of the results in this
section currently require the stronger upper bound of Assumption \ref{uBd}.
The significance of this assumption lies on imposing some control on the overshoot that
certain martingales can have when crossing a level. This plays a role in being able to 
obtain certain {\it maximal inequalities}, i.e. relating to $\sup_n M_n$, of martingales. We will 
come to this point later.

Some of the insights in the first works of Junk came from revisiting the martingale nature of
$W_n^\beta$ and drawing inspiration from classical (but forgotten) works from the 
onset of martingale theory \cite{G69}. In particular, the relations between moments 
of martingales which possess a linear structure such as the one encountered in \eqref{linear1}
and the integrability of their supremum. These relations were then applied to reveal the links
between moments of $W_n^\beta$ and integrability of $\sup_n W_n^\beta$, i.e. the understanding of {\it maximal inequalities}. Let us state
Junk's theorem \cite{J22a}:
\begin{theorem}\label{p-mom}
Let $\beta$ be such that weak disorder holds. Then the following hold:
\begin{itemize}
\item[{\rm (i)}] $\bbE\big[ \sup_{n\geq 1} W_n^\beta \big] <\infty$
\item[{\rm (ii)}] If Assumption \ref{uBd} holds and $p^*(\beta)$ is as defined in \eqref{pstar}, 
then $p^*(\beta) >1$.
\end{itemize}
\end{theorem}
We will present the proof of Theorem \ref{p-mom} under the more restrictive boundedness
Assumption  \ref{uBd} as originally done in \cite{J22a}. The proof under the more general 
Assumption \ref{overshoot} was provided in \cite{FJ23}, to which we refer. 
Before presenting the proof of the above theorem, let us state some properties of $p^*(\beta)$
as these were presented and proved in \cite{J23}. We note that the weak monotonicity part of 
Property (ii) is a consequence of the
monotonicity proved in Proposition \ref{prop:mon}. We will also
discuss the proof of Property (i)  after the proof of the main theorem. For the proofs of Properties
(iii) and (iv) we refer to \cite{J23}, Theorem 1.2.
\begin{theorem}[Properties of $p^*(\beta)$] \label{thm:propp*}
Let Assumption \ref{uBd} hold true and also assume weak disorder. The $L^p$-critical exponent $p^*(\beta)$ has the
following properties:
\begin{itemize}
\item[{\rm (i)}]  it holds that $p^*(\beta)\geq 1+\frac{2}{d}$.

\item[{\rm (ii)}]  the function $\beta \to p^*(\beta)$ is strictly decreasing, if $\bbP$ has finite support.

\item[{\rm (iii)}] if $\beta$ is such that $p^*(\beta)\in (1+2/d, 2]$, then the function  $p^*(\cdot)$
 is right-continuous at $\beta$.
 
 \item[{\rm (iv)}]  if $\beta > \beta_2$,  and $p^*(\beta)>1$, then $p^*(\cdot)$ is left-continuous at $\beta$.
\end{itemize}
\end{theorem}
The value of $p^*(\beta_c)$ at the critical temperature $\beta_c$ has been identified more recently in \cite{JL24a}:
\begin{theorem}[\cite{JL24a}, Corollary 2.2]
Under Assumption \ref{uBd}, it holds that in $d\geq 3$,
  \begin{align*}
  p^*(\beta_c)=\lim_{\beta\uparrow \beta_c} p^*(\beta)=1+\frac{2}{d}.
  \end{align*}
\end{theorem}
The last theorem is a consequence of the results summarised in Theorem \ref{thm:propp*} and Theorem \ref{thm:JL}. Indeed,
by the latter theorem we know that weak disorder holds at $\beta_c$ and by (i) of Theorem \ref{thm:propp*} 
that $p^*(\beta_c)\geq 1+\frac{2}{d}$. If $p^*(\beta_c) > 1+\frac{2}{d}$, then by (iii) of Theorem \ref{thm:propp*}, $p^*(\cdot)$ would have
to be right-continuous at $\beta_c$ but this is not the case as for $\beta>\beta_c$ the partition function only has first moments uniformly bounded. So we must have that $p^*(\beta_c) = 1+\frac{2}{d}$. The left-continuity follows form (iv) of  Theorem \ref{thm:propp*}.

A second critical exponent was defined in \cite{J22b}:
\begin{align*}
q^*(\beta):=\inf\big\{ p\geq 1 \colon \lim_{n\to\infty}\frac{1}{n} \log \bbE\big[ (W_n^\beta)^p\big]>0 \big\}.
\end{align*}
The following theorem proved in \cite{J22b}, Theorem 1.5,
showed that in the weak disorder regime, the $p$ moment of $W_n^\beta$ grows exponentially
for $p>p^*(\beta)$:
\begin{theorem}\label{thm:q*}
Assume $d\geq 3$, weak disorder, $\beta>\beta_2$ and that the disorder is upper bounded 
as in Assumption \ref{uBd}. Then for every $p>p^*(\beta)$, with $p^*(\beta)$ as in \eqref{pstar}, 
there exist constants $C_1>0$ and $C_2>1$, such that
\begin{align*}
\bbE\big[ \big( W_n^\beta\big)^p\big] \geq C_1 C_2^n.
\end{align*}
\end{theorem}
In \cite{JL24b}, Corollary 2.9, the following, stronger result was proved:
\begin{theorem}
For $\beta$ in the weak disorder regime, we have $p^*(\beta)=q^*(\beta)$.
\end{theorem}
The question of whether this theorem extends to the strong disorder regime is currently open for general disorder. We refer
to \cite{JL24b}, Remark 2.10 for a discussion.

\vskip 2mm
Let us now move towards the proof of Theorem \ref{p-mom}. As already mentioned, the proof of that
theorem relies on the understanding of maximal inequalities for martingales with a linear structure.
The general, relevant theorem proved in \cite{J22a} is the following:
\begin{theorem}\label{thm:martingale}
Let $(M_n)_{n\geq 0}$ with $M_0=1$ 
be a non-negative martingale with respect to a filtration $(\cF_n)_{n\geq 1}$ and with the property that
for every $k,\ell \in \bbN$ and convex $f\colon \bbR_+\to \bbR$, almost surely on $M_k>0$, it holds that
\begin{align}\label{convexity}
\bbE\Big[ f\Big( \frac{M_{k+\ell}}{M_k}\Big) \, \Big| \, \cF_k\Big] \leq \bbE [f(M_\ell)].
\end{align}
Denote by $M_n^*:=\sup_{0\leq k \leq n} M_k$, by $M_\infty:=\lim_{n\to \infty} M_n$ and by 
$M^*_\infty:=\lim_{n\to \infty} M^*_n$. Then the following hold:
\begin{itemize}

\item[{\rm (i)}] if $\bbP(M_\infty>0)>0$, then $\bbE[M_\infty^*]<\infty$.

\item[{\rm (ii)}] if $\bbP(M_\infty>0)>0$ and there exists $K>1$ such that 
\begin{align}\label{Mupper}
\bbP(M_{n+1} \leq K M_n) = 1, \qquad \text{for all $n\in \bbN$}, 
\end{align}
then there exists $p>1$ such that 
\begin{align}\label{eq:lp}
\sup_{n\geq 1} \bbE[ M_n^p ] <\infty.
\end{align}
The interval of $p>0$ satisfying \eqref{eq:lp} is open.
\item[{\rm (iii)}] if $\bbP(M_\infty=0)=1$ and \eqref{Mupper} holds, then\
footnote{the original bound in \cite{J22a} was $\frac{1}{4K^2t}$. 
The improved bound here and its proof was suggested by H. Lacoin.}
\begin{align}\label{Mstar-tail}
\bbP(M^*_\infty > t) >\frac{1}{Kt}, \qquad \text{for all $t>1$}.
\end{align} 
\item[{\rm (iv)}] if $\bbP(M_\infty >0)=1$ and there exists $K>1$ such that
\begin{align}
\bbP(M_{n+1}\geq M_n/K)=1, \qquad \text{for all $n\in \bbN$},
\end{align} 
then there exists $p>0$ such that 
\begin{align}\label{Mlower}
\sup_n \bbE[M_n^{-p}]<\infty.
\end{align}
The interval of $p>0$ satisfying \eqref{Mlower} is open.
\end{itemize}
\end{theorem}
\begin{proof}
(i) The proof is simple and ingenious. It hinges upon cleverly 
choosing a suitable concave test function that takes advantage of property
 \eqref{convexity}. The desired function is
\begin{align*}
f_{\delta, \epsilon}(x):= \Big(\delta\big(\frac{x}{\epsilon}-1\big)\Big) \wedge 1.
\end{align*}
We notice that for all $x\geq 0$ the inequalities 
\begin{align}\label{fineq}
\ind_{(\epsilon,\infty)} \geq f_{\delta, \epsilon}(x)
 \geq \ind_{[ (\delta^{-1}+1)\epsilon, \infty)}(x)- \delta \ind_{[0, \epsilon)}(x),
\end{align}
hold. For $t>0$, consider the crossing time $\tau:=\inf\{n\in \bbN \colon M_n>t\}$. 
On the even $\tau<\infty$ we, naturally,  have that $M_\tau\geq t >0$ and, thus,
\begin{align}\label{Mnlower}
\bbP(M_n>t\epsilon) 
& \geq \bbP(\tau \leq n, \frac{M_n}{M_\tau} >\epsilon) \notag \\
& = \sum_{k=1}^n
 \bbE\Big[ \ind_{\{\tau=k\}} \,\bbE\Big[ \ind_{\frac{M_n}{M_k} >\epsilon }\Big| \cF_k \Big] \Big] \notag \\
&\geq  \sum_{k=1}^n
 \bbE\Big[ \ind_{\{\tau=k\}} \, \bbE\Big[ f_{\delta, \epsilon}\Big( \frac{M_n}{M_k}\Big)\Big| \cF_k \Big] \Big] \notag \\
 &\geq  \sum_{k=1}^n
 \bbE\Big[ \ind_{\{\tau=k\}} \, \bbE\big[ f_{\delta, \epsilon}\big( M_{n-k} \big) \big] \Big] \\
 &\geq \bbP(\tau\leq n) \, \inf_{k} \bbE[  f_{\delta, \epsilon}\big( M_{k} \big) ],
 \end{align}
 where in the second inequality we used \eqref{fineq} and in the third \eqref{convexity}.
 Using now the second inequality in \eqref{fineq} we have that
 \begin{align*}
  \inf_{k} \bbE[  f_{\delta, \epsilon}\big( M_{k} \big) ] 
  &\geq \bbE[ \inf_{k} f_{\delta, \epsilon}\big( M_{k} \big) ] \\
  &\geq \bbP(\inf_{k}M_k \geq (\delta^{-1}+1)\epsilon) ) - \delta \bbP(\inf_{k} M_k <\epsilon)\\
  &\xrightarrow[\epsilon\to 0]{} \bbP(M_\infty>0)- \delta \bbP(M_\infty =0),
 \end{align*}
 where the limit is a consequence of the fact that $0$ is an absorbing state for the non-negative martingale $M_n$ and so $\{M_\infty>0\} = \{\inf_n M_n >0 \}$.
 Given the assumption that $\bbP(M_\infty>0)>0$ and the above estimate, we can choose
 $\epsilon, \delta$ small enough such that 
 \begin{align*}
 \bbE \big[ f_{\delta,\epsilon}(M_k) \big] =:\eta >0.
 \end{align*}
 Going back to \eqref{Mnlower} we have that
 \begin{align*}
 \bbP(M_n>t\epsilon) 
 \geq \eta \bbP(\tau\leq n)
 = \eta \bbP(M^*_n>t) 
 \end{align*}
 and the result follows upon integrating over $t\in \bbR_+$, giving
 \begin{align*}
 \bbE[M^*_n] \leq \frac{1}{\epsilon \eta} \bbE[M_n] = \frac{1}{\epsilon \eta} ,
 \end{align*}
 and the monotone convergence theorem.
 \vskip 2mm
 {\rm (ii)} For $t>0$ that will be suitably chosen later, consider the crossing time
 $\tau=\inf\{ n \colon M_n > t\}$. The assumption that a.s. $M_{n+1}\leq K M_{n}$ for all $n\in \bbN$
 implies that, on $\{\tau\leq n\}$, it holds 
 \begin{align*}
 M_\tau = \frac{M_\tau}{M_{\tau-1}} M_{\tau-1} \leq K t,
 \end{align*}
 and, thus, $M_n\leq tK \frac{M_n}{M_\tau}$. The latter implies the following estimate for any
  $p>1$:
 \begin{align}\label{Mnepsilon}
 \bbE[M_n^{p}] 
 &\leq  t^{p} +  \bbE[M_n^{p} \, \ind_{\{\tau\leq n\}}] \notag \\
 &\leq t^{p} + (Kt)^{p} 
    \sum_{k=1}^n \bbE\Big[ \Big( \frac{M_n}{M_\tau}\Big)^{p} \, \ind_{\{\tau =k\}} \Big] \notag\\
 & =t^{p} + (Kt)^{p} 
    \sum_{k=1}^n \bbE\Big[\ind_{\{\tau =k\}} \, 
    \bbE\Big[ \Big( \frac{M_n}{M_\tau}\Big)^{p} \,\Big| \, \cF_k\Big]\, \Big] \notag\\
 &\leq  t^{p} + (Kt)^{p} 
    \sum_{k=1}^n \bbE\Big[\ind_{\{\tau =k\}} \, 
    \bbE\big[ \big( M_{n-k}\big)^{p}  \big] \Big]  \notag\\
 &\leq     t^{p} + (Kt)^{p} \, \bbP(\tau\leq n) \, \bbE [ M_n^{p} ] \notag\\
 &\leq     t^{p} + (Kt)^{p}  \, \bbP(M^*_\infty>t) \,\bbE [ M_n^{p} ] 
 \end{align}
 where in the third inequality we used assumption \eqref{convexity}, in the fourth that
 $\bbP(\tau\leq n)= \bbP(M^*_n>t)\leq  \bbP(M^*_\infty>t)$ 
and in the last the fact for $p>1$, $(M_n^{p})_{n\geq 1}$ is a sub-martingale. 
Since by Chebyshev inequality
\begin{align}\label{maximal}
\bbP(M^*_\infty>t) \leq \frac{1}{t} \bbE[ M^*_\infty ; M^*_\infty >t],
\end{align}
and $\bbE[M^*_\infty] <\infty$, we can choose, using monotone convergence,
 $t$ large enough such that, given $K$ as in \eqref{Mupper}, we have
\begin{align}\label{maxchoice}
\bbP(M^*_\infty>t) \leq \frac{1}{4K^p t}.
\end{align}
 Inserting this into 
\eqref{Mnepsilon} and choosing $p$ sufficiently close to $1$, such that $t^{p-1}\leq 2$, we have that
\begin{align*}
\bbE[M_n^{p}]  \leq t^{p} +\frac{1}{2} \bbE[M_n^{p}] ,
\end{align*} 
thus, implying that $\sup_n \bbE[M_n^{p}] \leq 2 t^{p}$ for $t$ and $p>1$ chosen as above.

The fact that the interval of $p>1$ such that $\sup_n\bbE[M_n^p]<\infty$ is open goes as follows.
Let $p>1$ such that $\sup_n\bbE[M_n^p]<\infty$ and let $q>p$ to be chosen. Repeat the steps leading to \eqref{Mnepsilon} with $p$ there replaced by $q>p$. Instead of \eqref{maximal} use
the $L^q$ inequality to conclude
\begin{align}\label{maximal-p}
\bbP(M^*_\infty>t) \leq \frac{1}{t^p} \bbE[ (M^*_\infty)^p ; M^*_\infty >t] \leq \frac{1}{4K^q t^p},
\end{align}
which leads to 
\begin{align*}
\bbE[M_n^q] \leq t^q + \frac{t^{q-p}}{4} \bbE[M_n^q],
\end{align*}
and it remains to choose $q>p$ close enough to $p$ such that $t^{q-p}<2$.
\vskip 2mm
(iii) 
Let $\tau:=\inf\{ n\colon M_n>t\}$. By the martingale property we have that $1=\bbE[ M_{n\wedge \tau} ]$ and by letting $n\to\infty$, 
dominated convergence (using that $M_{n\wedge \tau}\leq Kt$) and that $M_\infty=0$  a.s., we have that
\begin{align*}
1=\bbE[ M_\tau\,;\, \tau<\infty] \leq Kt \,\P(\tau<\infty),
\end{align*}
which implies the result.
\vskip 2mm
(iv) The existence of negative moments follows similar steps as the proof of (ii). We refer to \cite{J22a}
for details.
\end{proof}
We are now ready to provide the proof of Theorem \ref{p-mom}.
\begin{proof}[Proof of Theorem \ref{p-mom}]
It suffices to apply Theorem \ref{thm:martingale} for the martingale $W_n^\beta$ and, thus, it
suffices to check conditions \eqref{convexity} and  \eqref{Mupper}. For the latter we have
\begin{align*}
\frac{W^\beta_{n+1}}{W^\beta_n}= \sum_{x\in \bbZ^d} \mu_n^\beta(x) 
\frac{1}{2d}\sum_{y\in \bbZ^d \colon |y-x|=1} e^{\beta \omega_{n+1,y}-\lambda(\beta)}
\leq e^{\beta M},
\end{align*}
and, therefore, \eqref{Mupper} is satisfied with $K=e^{\beta M}$. Similarly, 
for any $f\colon \bbR_+\to\bbR$ it holds that
\begin{align*}
\bbE\Big[ f\Big( \frac{W_{k+\ell}}{W_\ell} \Big) \,\Big| \, \cF_k\Big]
&=\bbE\Big[ f\Big( \sum_{x\in \bbZ^d} \mu_k^\beta(x) W_{(k,k+\ell]}^\beta(x)  \Big) \,\Big| \, \cF_k\Big] \\
&\leq \sum_{x\in \bbZ^d} \mu_k^\beta(x)  \bbE\Big[ f\Big(W_{(k,k+\ell]}^\beta(x)  \Big) \,\Big| \, \cF_k\Big] \\
&= \sum_{x\in \bbZ^d} \mu_k^\beta(x) \bbE\big[ f\big(W_\ell^\beta \big)\big]  \\
&= \bbE\big[ f\big(W_\ell^\beta \big)\big] ,
\end{align*}
where the inequality is due to Jensen. Therefore, condition \eqref{convexity} is also satisfied.
\end{proof}
Let us now discuss the bound $p^*(\beta)\geq 1+\frac{2}{d}$ from Theorem \ref{thm:propp*}, (i).
This follows from the asymptotic behaviour of $(W_n^\beta(x))_{x\in\bbZ^d}$ when viewed as
a field indexed by the starting point $x\in\bbZ^d$ of the polymer path. The fact that
 the moments of $W_n^\beta$ are related to
the correlations of the field $(W_n^\beta(x))_{x\in\bbZ^d}$ is probably not a surprise as moments
capture some tail properties of $W_n^\beta(x)$ and can encode localisation and, thus, correlations.
Still, this interplay is very interesting and warrants further investigation. 

We need the following theorem, which combines results from three works. Part (i) was proved in
\cite{CNN22}, Part (ii) in \cite{CNN22} in the continuous setting and in \cite{LZ22} in the 
discrete and Part (iii) in \cite{J22b}.

\begin{theorem}\label{thm:flucd3}
For a test function $\phi\in C_c(\bbR^d)$, we define the averaged field as
\begin{align*}
W_n^\beta(\phi):=\frac{1}{n^{d/2}} \sum_{x\in \bbZ^d} \phi\big(\tfrac{x}{\sqrt{n}}\big) \big( W_n^{\beta}(x)-1\big).
\end{align*}
Let $d\geq 3$. Then
\begin{itemize}
\item[{\rm (i)}] If weak disorder holds, then for every $\phi\in C_c(\bbR^d)$,
\begin{align}\label{homog}
W_n^\beta(\phi)\xrightarrow[n\to\infty]{L^1} 0.
\end{align}
\item[{\rm (ii)}] If $\beta < \beta_2$, then for every $\phi\in C_c(\bbR^d)$,
\begin{align}\label{flucW3}
n^{\frac{d-2}{4}} W_n^\beta(\phi) \xrightarrow[n\to \infty]{d} \cN(0,\sigma_\phi^2(\beta)),
\end{align}
where $\cN(0,\sigma_\phi^2(\beta))$ is a Gaussian random variable with mean zero and
variance 
\begin{align*}
\sigma_\phi^2(\beta)&
= \cC_\beta\int_0^1\int_{\bbR^d\times \bbR^r} \phi(x) g_{\tfrac{2t}{d}}(x-y) \phi(y) \,\dd x\dd y \quad \text{with} \\
&\,\quad  \cC_\beta=\big( e^{\lambda(2\beta)-2\lambda(\beta)}-1\big) \bbE\big[ (W_\infty^\beta)^2\big].
\end{align*}
\item[{\rm (iii)}]
Assume that weak disorder {\rm (WD)} holds and $\beta>\beta_2$ and, further, assume 
Assumption \ref{uBd}.
For $\xi(\beta):=\tfrac{d}{2}-\tfrac{2+d}{2p^*(\beta)}$ and every $\phi\in C_c(\bbR^d)$, we then have that
\begin{align}\label{WDbeyond}
\lim_{n\to \infty} \bbP\big( n^{-\xi(\beta)-\epsilon} \leq W_n^\beta(\phi) \leq n^{-\xi(\beta)+\epsilon}\big) =1.
\end{align}
\end{itemize}
\end{theorem}
The derivation of the bound $p^*(\beta)\geq 1+\frac{2}{d}$ comes, now,
 from the combination of (i) and
(iii) of the above theorem. In particular, the conclusion that $W_n^\beta(\phi)\to 0$ from (i) is 
consistent with the conclusion of (iii) only if $\xi(\beta)\geq 0$, which then translates to the
bound $p^*(\beta)\geq 1+\frac{2}{d}$.

Let us remark that results analogous to \eqref{flucW3} for the log of $W_n^{\beta}$ and the 
KPZ equation in $d\geq 3$ have been proved in \cite{CNN22,  CCM18, CCM20, LZ22}. 
A result of the form \eqref{WDbeyond} for the $\log W_n^{\beta}$ or the KPZ equation is still
missing.

\subsection{Very strong disorder and the fractional moment method}\label{sec:frac}
In this section we will discuss the phenomenon of {\it very strong disorder} as this was presented in Definition \ref{def:vSD}. 
That is, the question of whether the partition function
$W_n^\beta$ decays exponentially fast at a rate given by the free energy
\begin{align*}
\sff(\beta):=\lim_{n\to\infty}\frac{1}{n}\log W_n^\beta.
\end{align*}
We talk about very strong disorder if $\sff(\beta)<0$. As we already saw in 
Theorem \ref{thm:vsdloc}, very strong disorder is significant as it translates to localisation
estimates. 

Very strong disorder was first established by Comets-Vargas \cite{CV06} in dimension $1$ and for every $\beta>0$ via computation of fractional moments, in spirit similar to what
we saw in the proof of Theorem \ref{thm:SD}. 
Lacoin \cite{L10} proved that very strong disorder also holds in dimension 2 by 
making use of an upgraded formulation of the fractional moment, which constitutes a method,
now known as {\it the fractional moment method}. 
This method is very powerful, allowing for very precise estimate on free energies and
 other critical point shifts and has found several applications. 
 Recently, it was shown in \cite{JL24a} that strong disorder is equivalent to very strong disorder in any dimension. For this
 result the fractional moment method was combined with martingale tools similar to what were used in Section \ref{sec:momchar}
 and motivated by those objectives.
Outlining the factional moment method is the main objective of this section. 
Before doing so, let us mention the following theorem.
\begin{theorem}
Very strong disorder holds in dimensions $1,2$ for every $\beta>0$. Moreover, the 
following asymptotics holds:
\begin{itemize}
\item[\rm (i)] in dimension $1$, we have that
\begin{align}\label{eq:naka}
\lim_{\beta\to0} \frac{\sff(\beta)}{\beta^4} =-\frac{1}{6}.
\end{align}
\item[\rm (ii)]in dimension $2$, we have that 
\begin{align}\label{eq:BL}
\lim_{\beta\to0} \beta^2 \log \sff(\beta) =-\pi.
\end{align}
\end{itemize} 
\end{theorem}
The above theorem is the summary and culmination of several works concerning dimensions $1,2$. Lacoin \cite{L10}
first proved very strong disorder in both dimensions $1$ and $2$ and obtained upper
bounds on the free energy of order $O(\beta^4)$ and $O(e^{-c/\beta^4})$, in $d=1,2$, respectively, using (and developing) the fractional moment method. Lower bounds were also obtained in \cite{L10}. 
In $d=1$ and Gaussian disorder the lower bound was
of the form $O(\beta^4)$ and for $d=2$ and general disorder it was of the form $O(e^{-c/\beta^2})$.
The lower bounds were obtained using the so-called second-to-first moment estimates and concentration.
 The very precise asymptotic \eqref{eq:naka} was obtained by Nakashima in \cite{N19}
 making use of another method, called the intermediate disorder scaling, which we will
 discuss in Section \ref{sec:interm}. Asymptotic \eqref{eq:BL} was obtained by Berger-Lacoin
 in \cite{BL17}, developing further the fractional moment method for the upper bounds
 and performing sharp second-to-first moment estimates for the lower bound. 
\cite{BL17} refined in the directed polymer setting
 previous work of the authors \cite{BL16} on the study of the critical point shift of another
 important polymer model, called the {\it disordered pinning model}. 
  It is worth mentioning that the framework of the fractional moment method
  was actually born out of the attempt to study the critical points of the disordered pinning model, 
 in a series of works which started from \cite{T08, DGLT09, GLT11}.
  
In the following, we first outline the more standard fractional moment framework and afterwards we discuss the recent variation of Junk-Lacoin
\cite{JL24a}, which was used in order to establish Theorem \ref{thm:JL}, that strong and very strong disorder are equivalent in any dimension. 

{\bf Outline of the fractional moment method}.
The pillars of the fractional moment method are the following:
\begin{itemize}
\item[{\bf A.}] Estimate fractional moments.
\item[{\bf B.}] Perform coarse graining by identifying a proper correlation length.
\item[{\bf C.}] Change of measure.
\end{itemize}
Let us remark that the change of measure in Step C will share features with the 
{\it path tilted measures} we have seen in the discussion of the Random Walk Pinning Model
in Section \ref{sec:RWpin} as well as some crucial features with 
the material that we will discuss in Section \ref{sec:interm}. 

The discussion that follows is restricted to dimensions $1$ and $2$.
\vskip 2mm
Let us start with the following elementary computation:
\begin{align}\label{freegamma}
\sff(\beta)
= \lim_{n\to\infty} \frac{1}{n} \log \bbE \log W_n^\beta 
= \lim_{n\to\infty} \frac{1}{n \gamma} \log \bbE \log (W_n^\beta)^\gamma 
\leq  \lim_{n\to\infty} \frac{1}{n \gamma} \log \ \bbE(W_n^\beta)^\gamma ,
\end{align}
where first we trivially inserted a fractional power $\gamma\in (0,1)$ before applying Jensen's inequality.
The first crucial idea is to split, in a sense, the length $n$ of the polymer into 
{\it correlations lengths}. In other words, we will need to identify a particular length scale $\ell=\ell_{\rm corr.}(\beta)$
and split $n=m\ell$ where $m$ will be the number of correlation lengths inside horizon $n$.
The point then will be to identify the correlation length $\ell_{\rm corr.}$ so that (roughly)
$\bbE[(W_\ell^\beta)^\gamma] <e^{-c}$, for some $c>0$ and, therefore (roughly, again)
$\bbE[(W_n^\beta)^\gamma] \leq e^{- c m }$, which would then imply via \eqref{freegamma}
that
\begin{align}\label{frac_result}
\sff(\beta) \leq -\frac{c}{\gamma \, \ell_{\rm corr.}(\beta)}.
\end{align}
Identifying the correlation length is important in order to obtain a sharp estimate on the free energy.
In $d=1$ the correlation length turns out to be of order $\beta^{-4}$, while in dimension $d=2$ it is of order $e^{-\pi/\beta^2}$ \cite{L10, BL17}. 

Estimating a fractional moment $\gamma$ for $\gamma <1$ offers, on the one hand,
a more convenient route than estimating the expected value of the logarithm. Moreover, fractional moments
provide sharper estimates as compared to moments of order $p\geq 1$ as the latter are dominated by the atypical 
high values of the partition function $W_n^\beta$. 
The idea that allows to discount atypical fluctuations in the disorder is to 
{\it change the law} of the disorder at locations visited by the polymer path. This change of measure will
also facilitate estimating the fractional moment as then it will be the right moment to make use of H\"older inequality
(any attempt to use H\"older before a suitable change of measure is doomed to fail). However, we also do not want to change the measure
along all paths since this will have a high cost and here comes the idea of {\it coarse graining}, which is to change the
measure at only macroscopic blocks visited by the path. Let us start going through the details of this strategy.
\vskip 2mm
For $y\in \Z^d$ (for the considerations here $d$ will be either $1$ or $2$) let us define the spatial coarse-grained blocks
\begin{align*}
\Lambda_y:= y\sqrt{\ell} +\Big(-\frac{1}{2}\sqrt{\ell}, \frac{1}{2}\sqrt{\ell} \,\Big]^d
\qquad \text{for $y\in \bbZ^d$}.
\end{align*}
The scale of the coarse-graining is diffusive in the correlation length $\ell=\ell_{\rm corr.}$ (to keep notation simpler
we will drop the ``corr.'' and just use $\ell$). 
Given a vector $\cY:=(y_1,...,y_m)\in (\bbZ^d)^m$ define the event $\cE_\cY$ on the random walk trajectory $S$:
\begin{align*}
\cE_\cY:=\{ S_{i\ell} \in \Lambda_{y_i}, \text{for all $i=1,...m$}\}
\end{align*}
In other words, $\cE_\cY$ is the event that the trajectory of the random walk crosses boxes $\Lambda_{y_i}$ 
at times $i\ell$ for $i=1,...,m$. Let us decompose
\begin{align}\label{CGW}
W_n^\beta = \sum_{\cY\in (\bbZ^d)^m} \E \big[ e^{H^\omega_n(S)} \, \ind_{\{S\in \cE_\cY\}} \big] =: \sum_{\cY\in (\bbZ^d)^m}W^\beta_{n, \cY}
\end{align}
Using the inequality $(\sum_i a_i)^\gamma  \leq \sum_i a_i^\gamma$ for $\gamma \in (0,1)$ and $(a_i)_{i\geq 1}$ non-negative numbers, 
we have 
\begin{align}\label{fracest0}
\bbE [ (W_n^\beta)^\gamma ]  \leq \sum_{\cY\in (\bbZ^d)^m} \bbE[ (W^\beta_{n,\cY})^\gamma ]
\end{align}
The use of the elementary fractional inequality is the first technical advantage of working with powers  $\gamma\in (0,1)$ and we have already seen this in the proof of Theorem \ref{thm:SD}. 
Let us now define the space-time coarse grained blocks over which we will perform the
change of measure. For $R>0$, which will be taken large, we define
\begin{align*}
\tilde\Lambda_y&:= y\sqrt{\ell} +(-R\sqrt{\ell}, R\sqrt{\ell} \,]\cap \bbZ^d \quad 
\text{and} \quad \sfB_{(n,x)}:= [(n-1)\ell, n\ell\,] \times \tilde\Lambda_y
\end{align*}
The choice of $R$ large enough is in order to guarantee that, once the block is visited, the path will keep
moving, with high probability, within this block where disorder will be modified.

Next, come the details of how the change of measure is performed.
We will need to choose a random variable $\sfX(\omega) = \sfX(\{\omega_{n,x}\}_{(n,x)\in \sfB_{(1,0)}})$,
which depends locally only on the disorder inside box $\sfB_{(1,0)}$. The random variable 
$\sfX$ will be chosen so that it has
mean zero and $\bbE[\sfX^2] \leq 1$ but the most important feature will be to capture the correlation structure of the interaction 
of the polymer path with the environment, which leads to atypical fluctuations. 
We will come to this point and the precise choice of $\sfX$ later.  For the moment, let us also introduce 
the functions $\sfX_{i,y}:=\sfX\circ \theta_{(i \ell, y\sqrt\ell)}$, where $\theta$ is the shift operator,
 as well as the functions
\begin{align}\label{g}
g_{i,y}(\omega) = e^{- K \, \ind_{\{{\sfX}_{i,y}(\omega) \geq e^{K^2}\}}} \quad \text{and} \quad
g_\cY(\omega):=\prod_{i=0}^m g_{i,y_i}(\omega) 
 \end{align}
 where in the left $y\in \bbZ^d$ is a generic lattice point and $y_1,...,y_m$ in the right are the elements of the 
 vector $\cY$. The constant $K>0$ will be chosen large so that to penalise atypically large configurations $\sfX$.
 We insert the function $g_{\cY}(\omega)$ in the expectation $\bbE[(W_n^\beta)^\gamma]$ and 
 now is the time to use H\"older to obtain
\begin{align}\label{fracHolder}
\bbE \big[ (W_n^\beta)^\gamma \big] 
&= \bbE \big[ (W_n^\beta)^\gamma \, g_\cY(\omega)^\gamma \, g_{\cY}(\omega)^{-\gamma} \big] 
\leq  \bbE \big[ W_n^\beta \, g_\cY(\omega) \big]^\gamma 
   \, \bbE\big[ g_{\cY}(\omega)^{-\gamma/ (1-\gamma)} \big]^{1-\gamma} 
\end{align}
Using that $\sfX$ is chosen to satisfy $\bbE[\sfX^2]\leq 1$, we have via Chebyshev that
\begin{align}\label{RNbound}
\bbE\big[ g_{\cY}(\omega)^{-\gamma/ (1-\gamma)} \big] 
&\leq 1 +e^{\tfrac{\gamma}{1-\gamma}K} \bbP(\sfX\geq e^{K^2})  \notag \\
&\leq 1 +e^{\tfrac{\gamma}{1-\gamma}K - 2K^2} \leq 2 ,
\end{align}
where the last inequality is obtained by choosing $K$ large enough. 

Let us prepare for the more subtle analysis of the first term in \eqref{fracHolder}.
We start by rewriting it as
\begin{align}\label{change_est1}
 \bbE \big[ W_n^\beta \, g_\cY(\omega) \big]
 = \E \Big[ \bbE \big[ e^{H_n^\beta(S)} \, g_\cY(\omega) \big] \, \ind_{\cE_\cY} \,\Big]
 = \E \Big[ \prod_{i=1}^m \bbE_S \big[ \, g_{i,y_i}(\omega) \big] \,\big] \, \ind_{\cE_\cY} \,\Big],
\end{align}
where in the last we use again the measure $\bbP_S$, which tilts $\bbP$ by the factor $e^{H_n^\beta(S)}$ as defined in \eqref{sizebiasO}. 
We would now like to bring the product outside the expectation. This cannot be done as an identity
since the factors have Markovian dependence. However, we can make use of the Markov property of  $S$ at the entrance
points of block $\Lambda_{y_i}$ and dominate by considering the maximum over these entrance points. In this
way, we can obtain the upper bound
\begin{align}
&\bbE_S \big[ \, g_{i,y_i}(\omega) \big] \,\big] \, \ind_{\cE_\cY} \,\Big] \\
& \leq
 \prod_{i=1}^m \max_{x\in \Lambda_{y_{i-1}}} \E\Big[ \bbE_S\big[ g_{i,y_i}(\omega) \,\big] \ind_{S_{i\ell} \,\in \, \Lambda_{y_{i}}} 
\,\big| \, S_{(i-1)\ell} = x\Big]
=\Big( \max_{x\in \Lambda_{y_{0}}} \E_x\Big[ \bbE_S\big[ g_{0,0}(\omega) 
\,\big] \ind_{\{S_{\ell} \,\in \, \Lambda_{y_{1}}\}} 
\Big] \Big)^m, \notag
\end{align}
where in the last equality we used the shift invariance of $g_{i,y}(\omega)$.
Inserting bounds \eqref{change_est1} and \eqref{RNbound} into \eqref{fracHolder} and then into 
\eqref{fracest0} we obtain
\begin{align}\label{frac_main_est}
\bbE [ (W_n^\beta)^\gamma ]  \leq
\Big( 2^{1-\gamma}
\sum_{y\in \bbZ^2} \max_{x\in \Lambda_{y_{0}}} \E_x\Big[ \bbE_S\big[ g_{0,0}(\omega) \,\big] 
\ind_{\{S_{\ell} \,\in \, \Lambda_{y}\}}  \Big]^\gamma \Big)^m
\end{align}
it only remains to show that, when $\ell$ is suitably chosen, there is a constant $c>0$
\begin{align}\label{frac_desire2}
2^{1-\gamma}
\sum_{y\in \bbZ^2} \max_{x\in \Lambda_{y_{0}}} \E_x\Big[ \bbE_S\big[ g_{0,0}(\omega) \,\big] 
\ind_{\{S_{\ell} \,\in \, \Lambda_{y}\}}  \Big]^\gamma  
\leq e^{-c}.
\end{align}
which then implies the desired estimate \eqref{frac_result}.

The main estimate \eqref{frac_desire2} is obtained in the following way.
First, since $g\leq 1$ we have that
\begin{align*}
\sum_{|y|> A} \max_{x\in \Lambda_{y_{0}}} \E_x\Big[ \bbE_S\big[ g_{0,0}(\omega) \,\big]
\ind_{\{S_{\ell} \,\in \, \Lambda_{y}\}} \Big]
\leq \sum_{|y|> A} \max_{x\in \Lambda_{y_{0}}} \P_x\big(  S_{\ell} \,\in \, \Lambda_{y}  \big)
\leq  \sum_{|y|> A}  e^{-c|y|^2},
\end{align*}
by standard random walk estimates.
The latter can be made small by choosing $A$ large enough. 
Choose such an $A$ and fix it and then proceed with bounding
the remainder sum. Towards this task, we now drop the indicator $\ind_{\{S_{\ell} \,\in \, \Lambda_{y}\}}$
(as it does not offer any further help) and so 
\begin{align*}
\sum_{|y|\leq A} \max_{x\in \Lambda_{y_{0}}} \E_x\Big[ \bbE_S\big[ g_{0,0}(\omega) \,\big] \ind_{\{S_{\ell} \,\in \, \Lambda_{y}\}} \Big]
\leq A^d \max_{x\in \Lambda_{y_{0}}} \E_x\Big[ \bbE_S\big[ g_{0,0}(\omega) \,\big]  \Big].
\end{align*}
It turns out that with high probability with respect to the random walk $S$, the function $\sfX(\omega)$ has high expectation and high concentration with respect to $\bbP_S$. In other words, it turns out that for any $\delta>0$ and $R>0$, there exists $\beta_0(\delta, K)$ such that for $\beta\leq \beta_0$, uniformly in 
$x\in \Lambda_{y_0}$, we have
\begin{align}\label{Xest}
\P_x\big( \bbE_S[\sfX] >2e^{K^2} \big) \geq 1-\delta, \quad \text{and} \quad \P_x\big( \bbV ar_S[\sfX] \ll e^{K^2} \big) \geq 1-\delta.
\end{align}
These conditions would then imply that with high probability with respect to the random and uniformly 
over its starting point $x\in \Lambda_{y_0}$, we have
\begin{align*}
\bbE_S[g_{0,0}(\omega)] 
&\leq e^{-K} + \bbP_S\big( \sfX(\omega) \leq e^{K^2}\big) \\
&\leq e^{-K} + \bbP_S\big( \sfX(\omega) -\bbE_S[\sfX ]\leq -e^{K^2}\big) \\
&\leq e^{-K} + \bbP_S\big( \big|\, \sfX(\omega) -\bbE_S[\sfX ] \,| \geq e^{K^2}\big) \\
&= e^{-K} + e^{-K^2} \, \bbV ar_S[\sfX]
\end{align*}
and the above can be made arbitrarily small by choosing $K$ large enough. We now remain with the most crucial 
task, which is how to choose $\sfX$ in such a way that requirements $\bbE[\sfX]=0$, 
$\bbE[\sfX^2]\leq 1$ and most crucially \eqref{Xest} are satisfied.
\vskip 2mm
{\bf Choice of $\sfX$.} As we mentioned, the driving idea is to manage to discount configurations of the disorder $\omega$
which cause atypical fluctuations in the partition function $W_n^\beta$. This is implemented via the function
$g_\cY =\prod_{i=1}^m \exp\big(-K \ind_{\{ \sfX_{i,y_i} \geq e^{K^2}\} } \big)$ and the computation of 
$\bbE\big[ W_n^\beta g_\cY(\omega) \big]$ in \eqref{change_est1}.

Bearing in mind the coarse-graining that we performed above, it would seem natural to choose $\sfX$ to actually be
the partition function $W^\beta_\ell$ over a single correlation length $\ell$, itself. Then the function $g_\cY$ would
penalise by a factor $e^{-K}$ the expectation $\bbE\big[ W_n^\beta g_\cY(\omega) \big]$  
each time the partition $W^\beta_\ell \circ\theta_{(i\ell,y_i\sqrt{\ell})}$ takes exceptionally high values. 
However, for such a choice, estimates \eqref{Xest} are not easy. What has proven useful is to consider, instead,
an approximation of $W^\beta_\ell$ that carries substantial contributions to its total fluctuations. The suitable functional 
has proven to be of the form
\begin{align}
\sfX(\omega)
:=&\frac{1}{\sqrt{{\rm Vol}(B_{1,0})} \,  \sfR_{\ell}^{k/2}}
\sum_{\substack{x_0,x_1,...,x_k \in \tilde \Lambda_0 \label{Xchoice}\\ 
1\leq n_0< n_1<\cdots < n_k \leq \ell}}
\prod_{j=1}^k  q_{n_i-n_{i-1}}(x_i-x_{i-1}) \, \omega_{n_i,x_i}, \\
&\text{with} \quad \sfR_\ell:=\sum_{i=1}^\ell \sum_{x\in \bbZ^d} q_i(x)^2 
= \sum_{i=1}^\ell \E^{S,S'}\big[ \ind_{\{S_i=S_i'\}}\big] \label{overlap}
\end{align}
and $\text{Vol}(B_{1,0})$ denoting the Euclidean volume of $B_{1,0}$. $\sfR_\ell$ is the expected collision time of 
two independent simple random walks in a time horizon $\ell$, which by the local limit theorem is known to have the
asymptotic behaviour $\sfR_\ell \sim \sqrt{2\ell/\pi} $ is $d=1$ and  $\sfR_\ell \sim \log \ell/\pi$ in $d=2$.
\vskip 2mm
Let us remark that the functional $\sfX$ as defined in \eqref{Xchoice} is (essentially) the $k^{th}$ term in the so called {\it chaos expansion}
of the partition function, suitably normalised so that it satisfies $\bbE[\sfX^2]\leq 1$. We will discuss more about 
chaos expansion and its significance in the analysis of the partition function in Section
 \ref{sec4methods} (see Remark \ref{comparison}) but, for the moment,
one can essentially think of $\sfX$ as the expectation with respect to $\E$ of $k^{th}$ term in the Taylor expansion 
of the exponential 
\begin{align*}
e^{\sum_{i=1}^\ell (\beta \omega_{i,S_i}-\lambda(\beta)) }
= e^{\sum_{i=1}^\ell \sum_{x\in \bbZ^d}(\beta \omega_{i,x}-\lambda(\beta)) \, \ind_{\{S_i=x\}}}.
\end{align*}
normalised by $\sqrt{{\rm Vol}(B_{1,0})} \,  \sfR_{\ell}^{k/2}$.
It turns out \cite{L10} that in dimension one, $\sfX$ carries a significant portion of the fluctuations of $W^\beta_\ell$ even for $k=1$, which is eventually a sufficient choice. 
However, in dimension two, which is the {\it critical dimension}, the main contribution is carried by terms that correspond to $k$ that grows with $\ell$. We will also discuss this point in Remark \ref{comparison} in the next section.
 For a length scale 
$\ell=\ell(\beta)= \exp\big( \tfrac{1+\epsilon}{\pi} \beta^{-2}\big)$, which is near the critical correlation length, corresponding to 
$\epsilon=0$, it turns out to suffice to take $k=k_\ell:=(\log\log \ell)^2$ and this choice was made in \cite{BL17}. 
 To facilitate the estimates, Berger-Lacoin \cite{BL17} chose, moreover, a tailored version of $\sfX$.
 We denote it by $\sfX_{\rm BL}$ and for 
$\ell_\epsilon:=\ell^{1-\epsilon^2}$ it takes the form 
\begin{align*}
\sfX_{\rm BL}(\omega)
&:=\frac{1}{\sqrt{{\rm Vol}(B_{1,0})} \,  \sfR_{\ell_{\epsilon}}^{k/2}}
\sum_{\substack{x_0,x_1,...,x_k \in \tilde \Lambda_0 \\ 
1\leq n_0< n_1<\cdots < n_k \leq \ell_\epsilon \\ n_j-n_{j-1}\leq \ell_{\epsilon} \,,\, j=1,...,k}}
\prod_{j=1}^k \hat q_{n_i-n_{i-1}}(x_i-x_{i-1}) \, \omega_{n_i,x_i}, \\
&\text{with $\hat q_{n}(x) :=q_{n}(x) \ind_{\{|x|\leq \rho(t)} $ and with $\rho(t):=\min\big( \tfrac{t}{2}, (\log t) \sqrt{t}\big)$}.
\end{align*}

The fact that with high probability with respect to $S$, the mean $\bbE_S[\sfX_{\rm BL}]$ is high, is a consequence of the fact that
choosing $R$ large enough in the definition of $\tilde \Lambda_y$, we guarantee that with high probability the random walk path
$S$, which lands in $\tilde \Lambda_y$ will stay there for the whole time $\ell$. Thus, it will be moving inside a tilted 
disorder environment, leading to a high value of $\bbE_S[\sfX_{\rm BL}]$. The computations are subtle
and we refer to \cite{BL17} for the details.
\vskip 2mm
{\bf Outline of the Junk-Lacoin variation \cite{JL24a}.}  The starting point in this variation is also to find a scale $n$ such that
a fractional moment of $W_n^\beta$ is small. In \cite{JL24a} they require $\bbE\big[\sqrt{W_n^\beta}\big]< (2n+1)^{-d}$. 
A main difference with what we outlined above is at the change of measure required to obtain this estimate. Here the change of measure
is via the partition function $W_n^\beta$ itself instead of a chaos approximation proxy. In particular, they consider the size-biased change of measure
$\tilde\bbP(\dd\omega) = W_n^\beta \dd \bbP(\dd \omega)$. Lemma 3.2 in \cite{JL24a} shows that for any event $A$ of the disorder $\omega$, it holds that
\begin{align*}
\bbE\Big[\sqrt{W_N^\beta}\Big] \leq \sqrt{\bbP(A)} +\sqrt{\tilde \bbP(A^c)}.
\end{align*}
The task then is reduced to finding an event $A$ which has small probability with respect to $\bbP$ but high probability with respect
to $\tilde\bbP$. The intuition towards this is something we already discussed at the beginning of the previous change-of-measure outline,
which is to exploit events where the partition function has atypically large fluctuations. The event chosen in \cite{JL24b} is (we use slightly 
different notation)
\begin{align*}
A_{n,s}:=\Big\{ \exists \, (m,y)\in \{1,...,n\}\times\{-n,...,n\}^d  \colon \theta_{m,y} W_s^\beta \geq n^{4d}  \Big\},
\end{align*}
where $\theta$ is the shift operator i.e. $\big(\theta_{m,y}\,  \omega \big)_{n,x} = \omega_{n+m,x+y}$. It is then proved in Proposition 4.1 
in \cite{JL24a} that 
\begin{align*}
&\text{there exist $C>0, n_0\in \N$, such that $\forall n\geq n_0$ there exists $s=s_n\in [0, C\log n]$ for which} \\
& \hskip 3.5cm \bbP(A_{n,s})\leq \frac{1}{8(2n+1)^d}  \quad\text{and}\quad  \bbP(A_{n,s}^c)\leq \frac{1}{8(2n+1)^d}
\end{align*}
A further novelty in \cite{JL24a} is the approach to establishing the existence of this event and its probability estimates. This uses
martingale, stochastic calculus techniques and overshoot probability estimates. We refer directly to \cite{JL24a} for details.

\section{Intermediate disorder and phase transition in dimension $2$}\label{sec:interm}
\subsection{The general framework.}
We have so far seen that in dimensions $1$ and $2$, for every $\beta>0$, strong disorder
takes place, which naturally suggests the absence of a weak disorder regime in these dimensions. However, this is not the full picture. As we will discuss in this section,
in dimension $2$ a weak disorder phase can be exhibited but in order to observe it one
has to zoom into the critical temperature $\beta_c=0$ in a very particular fashion,
which is named {\it intermediate disorder regime}. The precise zooming is related to the
choice of $\beta$ as a function of the polymer length $n$ as
\begin{align}\label{def:betan} 
\beta_n:=\frac{\hat\beta }{ \sqrt{\sfR_n}},\qquad \text{with}\qquad 
\sfR_n:=\sum_{i=1}^n\sum_{x\in\bbZ^2} q_i(x)^2 =\frac{\log n}{\pi}(1+o(1)).
\end{align}
We note that $\sfR_n=\E[\sum_{i=1}^n \ind_{\{S_i=S'_i\}}]$ and, thus, it represents the
expected number of collisions of two independent simple random walks. A quantity
of this sort plays an important role in other disordered system models and often
goes by the name {\it ``replica overlap''}. The fact that the replica overlap grows like 
a slowly varying function is crucial towards seeing a weak-to-strong disorder phase transition.
The phase transition happens when changing the parameter $\hat\beta$. This phenomenon
was first observed in \cite{CSZ17b} and is presented in the following theorem:
\begin{theorem}\label{2dsubcritical}
Let $d=2$ and $W_n^{\beta_n}$ be the partition function of the $2d$ directed polymer at intermediate 
disorder regime with $\beta_n=\hat\beta / \sqrt{\sfR_n}$ as in \eqref{def:betan}.
 We then have
 \begin{equation} \label{conv-subZ}
	W_{n}^{\beta_n}\xrightarrow[\,n\to\infty\,]{d} 
	\begin{cases}
	\exp\big( \sigma_{\hat\beta} \, \cN-
	\frac{1}{2} \sigma_{\hat\beta}^2
	\,\big)
	& \text{if } \hat \beta < 1 \\
	0 & \text{if } \hat\beta \ge 1
	\end{cases} \,.
\end{equation}
where the convergence is in distribution and $\cN$ is a standard normal variable and $ \sigma_{\hat\beta}^2  = \log (1-\hat\gb^2)^{-1}$.
\end{theorem}
The transition between weak and strong disorder phase in the intermediate disorder regime happen at the
critical value $\hat\beta_c=1$. Below this, $W_n^{\beta_n}$ converges in distribution to a
strictly positive random variable, while at and above it to $0$. A difference between the
weak and strong disorder, as this was defined earlier is that the limit  $W_n^{\beta_n}$, is
{\it in distribution} while earlier we were talking about $a.s.$ limits. 
Another very important difference between weak-to-strong disorder transition in $d\geq 3$
and the transition observed in $d=2$ is that the critical point $\hat\beta_c$ is 
explicitly determined as $\hat\beta_c=1$ and moreover it coincides with the $L^2$ critical
point. In fact, it is worth drawing a contrast to the moment characterisation described in Section \ref{sec:momchar}: in $d=2$ and intermediate disorder regime, {\it all}
moments exist for $\hat\beta<\hat\beta_c=1$. This has been proved in \cite{CZ21, CZ23, LZ22, LZ23}. Furthermore, in \cite{CZ23} it has been proved that $W_n^{\beta_n}$
has moments even of growing order, up to $C(\hat\beta) \sqrt{\log n}$ for a particular
constant $C(\hat\beta)$. Such considerations are important towards understanding 
of properties of the field of partition functions, indexed by their starting point, as 
pointed in \cite{CZ23}.

\vskip 2mm
{\bf On the intermediate disorder regime and (marginal) disorder relevance.}
The notion of intermediate disorder was first introduced by Alberts-Khanin-Quastel 
\cite{AKQ14} for the one-dimensional polymer, with temperature scaling 
$\beta_n=\hat\beta n^{-1/4}$. There it was shown that in $d=1$ and with such 
intermediate disorder scaling, $W_n^{\beta_n}$ converges in distribution to 
the solution of the one-dimensional stochastic heat equation (SHE)
\begin{align}\label{eq:SHE2}
\partial_t u=\frac{1}{2}\Delta u +\sqrt{2}\hat\beta \xi u,\qquad t>0, x\in \bbR,
\qquad \text{and}\qquad u(0,x)\equiv 1, \quad \text{for all $x\in\bbR$}. 
\end{align}
with $\xi$ being the space-time white. Contrary to what is observed in dimension two, in 
dimension one there no phase transition of weak-to-strong disorder type, as for any $\hat\beta>0$ the solution to the stochastic heat equation is strictly positive. 
The fact that one sees a phase transition in dimension two is a consequence of the fact
  that dimension two  is a {\it critical dimension} in the sense of renormalisation. 
  To get an idea of what this means, consider the rescaled solution to the SHE 
  $u_\epsilon(t,x):=u(\frac{t}{\epsilon^2}, \frac{x}{\epsilon})$ in
  general dimension $d\geq 1$. Using the scaling
    invariance of the white noise, $u_\epsilon$ then solves the equation
    \begin{align*}
\partial_t u_\epsilon=\frac{1}{2}\Delta u_\epsilon + \epsilon^{\frac{d-2}{2}} \sqrt{2}\hat\beta \xi u,
\end{align*}
and we see that, as $\epsilon\to 0$, the noise term $\epsilon^{\frac{d-2}{2}} \sqrt{2}\hat\beta \xi u$ (formally) converges to $0$ in $d\geq 3$
 (indication of weak disorder as the effect of noise is attenuated) and 
 it blows up when $d=1$ (indication of strong disorder), while it remains invariant when 
 $d=2$. The latter indicates that dimension two is critical and 
 whether disorder is relevant (strong disorder) or irrelevant (weak disorder) 
 cannot be concluded from a renormalisation procedure as under this transform 
 the noise term has neither a clearly dominant effect nor it is clearly negligible. 
 In statistical physics and renormalisation this situation is often termed 
 {\it marginal disorder} \cite{H74, G11}. In a situation like this one has to look at the
 system in a much finer detail, in order to draw conclusions, and this is the framework 
 that the two-dimensional polymer falls into, as also confirmed by Theorem \ref{2dsubcritical}
 
 Theorem \ref{2dsubcritical} was presented in \cite{CSZ17b} as a particular case of 
 a more general framework, which concerned {\it marginally relevant disordered systems},
 with a time direction. Let us record a form of this result here
  \begin{theorem}\label{thm:marginal}
  Let $Z_{n,\gb_n}^{\rm marginal}$ be a partition function of a marginally relevant disordered system, 
  which admits an expansion as a multilinear polynomial of the form
  \begin{align}\label{marginalpart}
  Z_{n,\gb_n}^{\rm marginal} = 1+\sum_{k=1}^n\,\,\gb_n^k\,\,\sumtwo{1\leq i_1<\cdots <i_k\leq n}{x_0=0, x_1,...,x_k\in \bbZ^d} \,\,\prod_{j=1}^k 
  p_{i_j-i_{j-1}}(x_j-x_{j-1}) \,\,\xi_{i_j,x_j},
  \end{align}
  where
   $(\xi_{i,x})_{i\in \N, x\in \bbZ^d}$ is a collection of i.i.d. mean zero, variance one random variables with exponential moments
  and the kernel $(p_{i}(x))_{i\in \N,x\in \bbZ^d}$ satisfies a condition which can be recast as equivalent to marginal relevance:
  \begin{align}\label{RN}
  \sfR_n:=\sum_{i=1}^n\sum_{x\in \bbZ^d} p_i(x)^2 \quad \text{grows to infinity as a slowly varying function}.
  \end{align}
  Assume, also, that the kernel $q_n(x)$ satisfies a type of local limit theorem, i.e.
  \begin{equation} \label{q-LLT}
	\sup_{x\in\bbZ^d} \left\{ n^\gamma \,p_n(x) -
	g\big(\frac{x}{n^a} \big) \right\}
	\xrightarrow[n\to\infty]{} 0 \,,
\end{equation}
for a sufficiently smooth density $g(\cdot)$ and exponents $a,\gamma>0$
  Then, if $\beta_n:=\hat\beta / \sqrt{\sfR_n}$, it holds that 
  \begin{equation} \label{eq:conve0}
	Z_{n,\beta_n}^{\rm marginal} \xrightarrow[\,n\to\infty\,]{d} 
	\begin{cases}
	\exp\big( \sigma_{\hat\beta} \cN-
	\frac{1}{2} \sigma_{\hat\beta}^2
	\,\big)
	& \text{if } \hat \beta < 1 \\
	0 & \text{if } \hat\beta \ge 1
	\end{cases} \,.
\end{equation}
where $\cN$ is a standard normal variable and $ \sigma_{\hat\beta}^2  = \log (1-\hat\gb^2)^{-1}$.
\end{theorem}
We will refer to \cite{CSZ17b, CSZ18+} 
for more discussion about this theorem and its relations
to disorder relevance and the so-called {\it Harris criterion}, \cite{H74}. 
Here, {let us point that \eqref{RN} and \eqref{q-LLT} imply the relation $\gamma>\frac{1+ad}{2}$, which is equivalent to the
scaling relation proposed by Harris \cite{H74}. 
Let us also mention that the above theorem encompasses, besides the two-dimensional
directed polymer model, a one-dimensional directed polymer where the underlying walk
is in the domain of attraction of the Cauchy distribution, the disordered pinning model \cite{G11} and the two-dimensional stochastic heat equation. The extension to statistical
mechanics models which do not have a time directions (e.g. Parabolic Anderson model
or Ising type models) is an interesting open problem.

The extension to the two-dimensional KPZ equation has been accomplished in
\cite{CSZ20, G18}
and there has also been a very important extension of this framework to a certain class of
 seminlinear stochastic heat equations by Dunlap-Gu \cite{DG22} and 
 Dunlap-Graham \cite{DG23}
 where further phenomena have been revealed. We refer to the above works for details.

 \subsection{Field correlations}\label{sef:fieldcorr}
 In this subsection we want to outline the properties of the field $(W_n^{\beta_n}(x))_{x\in\bbZ^2}$, in analogy with what was presented for $d\geq 3$ in
 Theorem \ref{thm:flucd3}. This is important as it reveals the correlation structure
 of the polymer and how localisation can emerge. The following theorem was proved
 in {\cite{CSZ17b}, Theorem 2.13, in greater generality than what is stated below, that includes fluctuations of a wider class of marginally relevant models}.
\begin{theorem}[Partition function of $2d$ polymer as a field]\label{fieldDPRM}
 Let $(W_{n}^{\beta_n}(x))_{x\in\bbZ^2}$ be the field of partition functions of the directed polymer in $d=2$, indexed by the starting point $x\in\bbZ^2$. 
 Let $\beta_n$ be chosen as
 \begin{align}\label{choice_beta}
 \beta_n:=\frac{\hat\beta}{\sqrt{\sfR_n}}, \quad \text{with}
 \quad \sfR_n:=\sum_{i\leq n\,,\,x\in Z^2} q_i(x)^2 = \frac{\log n}{\pi}(1+o(1)).
 \end{align}
Let also $\phi\in C_b(\R^2)\cap L^1(\R^2)$ be a test function. Then, for $\hat\beta<1 $,
\begin{align}\label{eq:Wnphi}
 W_{n}^{\beta_n}(\phi):=\frac{\sqrt{\sfR_n}}{n} \sum_{x\in \bbZ^d} \big( W_{n}^{\beta_n}(x) -1\big) \phi\big( \tfrac{x}{\sqrt{n}}\big),
 \end{align}
 converges to a Gaussian variable with mean zero and variance
 \begin{align*}
 \sigma^2_{\hat\beta, \phi}= \frac{\hat\beta^2}{1-\hat\beta^2} \int_{R^2}\int_{R^2} \phi(x) K(x,y) \phi(y)\, \dd x\dd y, \qquad
 \text{with} \qquad K(x,y)=\int_0^{1} \frac{1}{2\pi t} e^{-\frac{|x-y|^2}{2t}} \,\dd t.
 \end{align*}
 \end{theorem}
 The analogue of this theorem in the case of $(\log W_n^{\beta_n}(x))_{x\in\bbZ^2}$ and
 for the solution of the KPZ equation has been proved in \cite{CSZ20, G18}. 
 The analogue for fields coming from solutions to semilinear stochastic her equation has
 been achieved in \cite{T22}.
 
We remark that the amplification in \eqref{eq:Wnphi} by the factor $\sqrt{R_n}\approx \sqrt{\log n}$, indicates that for $\hat\beta<1$, 
the polymer partition functions become asymptotically
uncorrelated. This is also an indication of a delocalisation of the polymer. Strong correlations
start to emerge at the critical temperature $\hat\beta_c=1$. This is the content of the
next theorem, proved in \cite{CSZ23}, which essentially says that the polymer partition,
viewed as a random measure, has a non-trivial scaling limit (without need of renormalisation)
at a {\it critical window} around $\hat\beta_c=1$, more precisely for $\beta_n$ such that
\begin{align}\label{def:critbeta}
\sigma(\beta_n)^2:=e^{\lambda(2\beta_n)-2\lambda(\beta_n)}-1=\frac{1}{\sfR_n}
\Big( 1+\frac{\vartheta+o(1)}{\log n}\Big),\qquad \text{for $\vartheta\in \bbR$}.
\end{align}
To formulate the result, recall the definition of the point-to-point partition function 
\eqref{eq:Zab1}, \eqref{eq:Zab2}.
To deal with parity issues, for $x \in \R^2$
we denote by $[\![x]\!]$ the closest point $z \in \Z^2_{\rm even}
:=\{(z_1, z_2)\in \Z^2: z_1+z_2 \mbox{ even}\}$;
for $s \in \R$ we define the even approximation
$[\![s]\!] := 2 \, \lfloor s/2 \rfloor \in \Z_{\rm even} := 2\Z$.
We then introduce the field
\begin{equation} \label{eq:rescZmeas}
	\cW^{\beta_n}_n = \bigg(W^{\beta_n}_{n;\, s,t}(\dd x, \dd y) :=
	\frac{n}{4} \, W_{[\![ns]\!], [\![nt]\!]}^{\beta_n,\,\omega}([\![\sqrt{n} x]\!],
	[\![\sqrt{n} y]\!]) \, \dd x \, \dd y
	\bigg)_{0 \le s \le t < \infty}.
\end{equation}
where $\dd x \, \dd y$ denotes the Lebesgue measure on $\R^2 \times \R^2$.
The rescaling $\frac{n}{4}$ is just to take into account the Lebesque scalings 
and, thus, it is not of a particular substance. 
The next theorem is mainly proved in \cite{CSZ23}; part (v) is proved in \cite{CSZ23b}, while (iv)
in \cite{CSZ24+}.
 \begin{theorem}[Critical 2d Stochastic Heat Flow]\label{th:mainSHF}
Fix $\beta_n$ in the critical window \eqref{def:critbeta}, for some $\vartheta\in \R$.
As $n\to\infty$, the family of random measures
$\cW^{\beta_n}_n =
(W^{\beta_n}_{n;\, s,t}(\dd x, \dd y))_{0\le s \le t < \infty}$ defined in \eqref{eq:rescZmeas}
converges in the sense of measures to a {\it unique} limit
\begin{equation*}
	\mathscr{Z}^\vartheta = (\mathscr{Z}_{s,t}^\vartheta(\dd x , \dd y))_{0 \le s \le t <\infty} \,,
\end{equation*}
which we call the \emph{Critical $2d$ Stochastic Heat Flow}. This is a  universal process, in that it does not depend on the law of the disorder $\omega$, and 
 has the following properties:
\begin{itemize}
\item[\rm (i)] it is translation invariant in law:
\begin{equation*}
	(\mathscr{Z}_{s+\sfa, t+\sfa}^\vartheta(\dd (x+\sfb) , \dd (y+\sfb)))_{0 \le s \le t <\infty}
	\stackrel{\rm dist}{=}
	(\mathscr{Z}_{s,t}^{\vartheta}(\dd x , \dd y))_{0 \le s \le t <\infty}
	\quad \ \forall \sfa \ge 0, \ \forall \sfb \in \R^2 \,,
\end{equation*}
\item[\rm(ii)] it satisfies the scaling covariant relation:
\begin{equation}\label{eq:scaling}
	(\mathscr{Z}_{\sfa s, \sfa t}^\vartheta(\dd (\sqrt{\sfa} x) , \dd (\sqrt{\sfa} y)))_{0 \le s \le t <\infty}
	\stackrel{\rm dist}{=}
	(\sfa\, \mathscr{Z}_{s,t}^{\vartheta+ \log \sfa}(\dd x , \dd y))_{0 \le s \le t <\infty}
	\quad \ \forall \sfa > 0 \,.
\end{equation}
\item[\rm (iii)] it is a log-correlated field, with covariance function as in \eqref{SHFcov} below,
\item[\rm (iv)] it is a.s. singular with respect to the Lebesque measure,
\item[\rm (v)] it is neither a Gaussian field nor an exponential of a Gaussian field,
i.e. it is not a Gaussian Multiplicative Chaos (GMC - see \cite{RV14} for a review on this model).
\end{itemize}
\end{theorem}
The covariance function of the Stochastic Heat Flow is given by
\begin{align}\label{SHFcov}
&\bbcov[\mathscr{Z}^\vartheta_{s,t}(\dd x, \dd y), \mathscr{Z}^\vartheta_{s,t}(\dd x', \dd y')]
	= \tfrac{1}{2} \, K_{t-s}^\vartheta(x,x'; y, y') \, \dd x \, \dd y \, \dd x' \, \dd y' \,, \qquad \text{with} \notag \\
	\\
&K_{t}^{\vartheta}(z,z'; w,w')
	\,=\, \pi \, \iiiint\limits_{\substack{0 < s < u < t \\ a,b \in \R^2}}
	\big\{ g_{\frac{s}{2}}(a-z) \, g_{\frac{s}{2}}(a-z') \big\}
	 \, G_\vartheta(u-s) \, g_{\frac{u-s}{4}}(b-a)   \notag \\
	& \qquad\qquad\qquad\qquad
	\times \big\{ g_{\frac{t-u}{2}}(w-b)\, g_{\frac{t-u}{2}}(w'-b) \big\} \,
	\dd s \, \dd u \, \dd a \, \dd b, \,  \notag \\
	\notag \\
	&G_\vartheta(t) = \int_0^\infty \frac{e^{(\vartheta-\gamma)s} \, s \, t^{s-1}}{\Gamma(s+1)} \, \dd s 
	\quad\text{and $g_t(x)$ the two-dimensional heat kernel.}\notag
\end{align}
The covariance kernel $K_{t-s}^\vartheta(x,x'; y, y')$
first appeared in the setting of the $2d$ stochastic heat equation \cite{BC98}.
Explicit but complicated expressions for the higher, mixed moments of the Critical 2d SHF also exist \cite{GQT21, CSZ19}.
A point to be remarked is that the $k^{th}$ moment grows too fast in $k$ (current lower bounds are of the 
form $e^{ck^2}$ \cite{CSZ23b}, 
while the conjectured growth is $e^{e^k}$) and, thus, cannot be used to determine the distributional 
properties of the field.  
A question of interest is:
\begin{question}
Provide an axiomatic characterisation of the Critical 2d SHF. Study its fine properties as a random measure and
explore its universality scope.
\end{question}

\subsection{Some points of the methods.}\label{sec4methods}
The approach to Theorems \ref{thm:marginal}, \ref{fieldDPRM} and \ref{th:mainSHF}
starts with a representation of the partition function in terms of an expansion in multilinear
polynomials of disorder. Such an expansion often  goes by the name {\it chaos expansion}
in stochastic analysis or {\it high temperature expansion} in statistical physics. Let us start describing it.
Denote by
 \begin{align}
 	\xi_{t,x} := \sigma(\beta)^{-1}	
	\big( e^{\beta \omega_{t,x} -\gl(\gb)} - 1 \big)  \quad \text{with} \quad
\sigma(\beta):=\sqrt{ e^{\lambda(2\beta) - 2\lambda(\beta)}-1 },
\end{align}
with $\sigma(\beta)$ chosen so that $\bbV ar(\xi_{t,x})=1$. 
Using the fact that 
\begin{align*}
e^{(\beta \omega_{t,x} -\gl(\gb))\ind_{\{S_t=x\}}} - 1 = 
\Big(e^{\beta \omega_{t,x} -\gl(\gb)} - 1\Big) \ind_{\{S_t=x\}},
\end{align*}
we can then write
\begin{align*}
\prod_{t=1}^n e^{(\beta \omega_{t,x}-\lambda(\beta) ) \ind_{\{S_t=x\}}}
&=\prod_{t=1}^n \Big( 1+ \big( e^{(\beta \omega_{t,x}-\lambda(\beta) ) \ind_{\{S_t=x\}}} -1\big)  \Big) \\
&=\prod_{t=1}^n \Big( 1+ \sigma(\beta) \xi_{t,x} \ind_{\{S_t=x\}}  \Big)\\
&= \sum_{k=1}^n \, \sumtwo{1\leq t_1<\cdots <t_k < n}{x_1,...,x_k\in \bbZ^d} \,
 \sigma(\beta)^k \prod_{i=1}^k    \xi_{t_i,x_i} \, \ind_{\{S_{t_i}=x_i\}}.
\end{align*}
Inserting this into the original expression for $W_n^\beta(x)$, we obtain its
chaos or high temperature expansion:
\begin{equation}\label{eq:poly}
\begin{split}
	W_{n}^\beta(x) = 1 + \sum_{k=1}^N \sigma(\beta)^k
	\sumtwo{0 =t_0 < t_1 < \ldots < t_k \le n}
	{x_0=x, \ x_1, \ldots, x_k \in \bbZ^2}
	\prod_{i=1}^k q_{t_i - t_{i-1}}(x_i - x_{i-1}) \, \xi_{t_i, x_i}, 
\end{split}
\end{equation}
	The above expansion for $W_n^{\beta_n}$ has the form of a multilinear polynomial. 
	This structure plays an important role in the approach as it allows for a robust limit theory.
	In particular, it allows for the use of two powerful tools, the {\it Lindeberg principle} and the 
	{\it Fourth Moment Theorem}. We will expose these below but first let us introduce the more 
	general form of a multilinear polynomial (a particular case of which is expansion \eqref{eq:poly}).
\begin{definition}[Multilinear polynomials]\label{def:mul-pol}
Let $S$ be an index set and $\xi:=(\xi_x \colon x\in S)$ be a family of i.i.d. random variables.
Consider also a kernel $\psi$ defined as a real function on the subsets of $S$. 
We define the multilinear polynomial $\Psi$ associated to the family $\xi$ and the kernel $\psi$, to be
\begin{align*}
\Psi(\xi):= \sum_{I\subset S, \,|I|<\infty} \psi(I) \prod_{x\in I} \xi_x.
\end{align*}
We also define the degree of $\Psi$, ${\rm deg}(\Psi):=\max\{ |I| \colon I\subset S \,\,\text{such that $\psi(I)\neq 0$} \}$.
\end{definition}
We also need the important notion of {\it influence}, which roughly quantifies how much a function
of several random variables depends on any individual one of them.
\begin{definition}[Influence]\label{def:influence}
Let $S$ be an index set, $\xi=(\xi_x\colon x\in S)$ be a family of i.i.d., real random variables and
$f\colon \R^S\to \R$ a function of $\xi$. The influence of the individual random variable $\xi_x, x\in S$
on $f$ is defined as
\begin{align*}
{\rm Inf}_x(f):=\bbE \big[ \bbV ar(f(\xi)) \,\big| \, \{\xi_y\}_{y\neq x} \, \big].
\end{align*} 
In the case that $f$ is a multilinear polynomial $\Psi$ with kernel $\psi$, it is easy to conclude that
the influence takes the form
\begin{align*}
{\rm Inf}_x(\Psi)
= \bbV ar\Big( \frac{\partial \Psi(\xi)}{\partial \xi_x}\Big)
=\sum_{I\ni x} \psi(I)^2.
\end{align*}
\end{definition}
Controlling the influence allows to control how much the distribution of a multilinear polynomial 
$\Psi$ will change, if the input changes from one family of random variables $\xi$ to another one $\zeta$.
The relevant result is the Lindeberg principle. 
\begin{theorem}[Lindeberg principle]\label{thm:Lind}
Let $\Psi$ a multilinear polynomial and $\xi=(\xi_x\colon x\in S)$ and $\zeta=(\zeta_x\colon x\in S)$
two independent families of i.i.d. random variables with mean $0$, variance $1$ and uniformly 
integrable second moments ( i.e. $\xi_x^2$ and $\eta_x^2$ are uniformly integrable). 
Then for any function $f\in C_b^3(\R;\R)$ and every $\epsilon>0$
 there exist $C_\epsilon$, depending on the sup-norm of the derivatives of $f$ and the degree of $\Psi$, ${\rm deg}(\Psi)$ 
 (see Definition \ref{def:mul-pol}), such that
 \begin{align*}
 \big| \bbE\big[f(\Psi(\xi)) \big] -\bbE\big[f(\Psi(\zeta)) \big] \big| \leq \epsilon + C_\epsilon \sqrt{\max_{x\in S} {\rm Inf}_x(\Psi)}.
 \end{align*}
\end{theorem}
Lindeberg principles have a long history. The philosophy was introduced by Lindeberg as direct way to 
prove the central limit theorem. More recently this principle has found a wide range of applications. 
The above theorem was established in \cite{CSZ17a}, Theorem 2.6 extending relevant work from \cite{MOO10}.
The significance of the result is that, if one has a sequence of multilinear functionals $\Psi_n$ and
one can show that asymptotically the influence is negligible, i.e. $\lim_{n\to \infty} \max_{x} {\rm Inf}_x(\Psi)\to 0$, as $n\to \infty$, then the asymptotic distribution of $\Psi_n(\xi)$ is stable and, thus, one can replace
$\xi$ by another set of variables that have more convenient distribution to work with, e.g. normal variables. 

We also need to record the following powerful theorem, which goes under the name {\it Fourth Moment Theorem} and identifies when asymptotic limits of multilinear polynomials are Gaussian.
The origins of the fourth moment theorem can be traced to \cite{dJ87, dJ90, NP05} but there has 
been a huge number of extensions and applications 
(see also \cite{CSZ17b}, Theorem 4.2 for  bit more general setting than what is presented below).
\begin{theorem}[Fourth moment theorem]\label{thm:fourth}
Let $\Psi_n$ be a sequence of 
multilinear polynomials on a family $\xi=(\xi_x\colon x\in S)$ of i.i.d. random variables
with mean $0$. Assume also $\Psi_n$ has variance $1$ and the each variable $\xi_x$ has asymptotically 
 negligible influence on $\Psi_n$.
Then the sequence $\Psi_n$ converges in distribution to a
standard normal random variable $\cN$ if $\bbE\big[ \Psi_n(\xi)^4 \big]\to 3$, as $n\to\infty$.
If, moreover, the family $\xi$ is a family of normal random variables, then this is an if-and-only-if statement.  
\end{theorem}
Let us now use the above toolbox to give a sketch of the proof of Theorem \ref{2dsubcritical}. 
An alternative proof has been provided by 
Caravenna-Cottini \cite{CC22} and Cosco-Donadini \cite{CD24}, who looked directly at $\log W_N^{\hat\beta}$ and provided
approximations for which the classical Lindeberg central limit theorem could be applied. 

\begin{proof}[{\bf  Proof of Theorem \ref{2dsubcritical} (Sketch)}]
Let us first discuss why the limit is $0$ for $\hat\beta\geq 1$. This will follow by approximating the partition function at 
critical value through its log-normal limit for $\hat\beta <1$ and using the monotonicity of the fractional moments as this follows from Proposition \ref{prop:mon}.
More precisely, let $\gamma\in (0,1)$. By the fact that for $\gamma\in(0,1)$ the fractional moment $\beta \to \bbE[(W_n^{\beta})^\gamma]$ is decreasing as a function of $\beta$, we have that for 
$\beta_n=\hat\beta /\sqrt{\sfR_n}$ with $\hat\beta<1$ :
\begin{align*}
\lim_{n\to\infty}\bbE[(W_n^{1/\sqrt{\sfR_n}})^\gamma] 
\leq \lim_{n\to\infty}\bbE[(W_n^{\hat\beta/\sqrt{\sfR_n}})^\gamma] 
= \bbE\big[ e^{\gamma\sigma_{\hat\beta} \cN- \tfrac{1}{2} \gamma\sigma_{\hat\beta}^2}\big]
= e^{\tfrac{(\gamma-1)\gamma}{2}\sigma^2_{\hat\beta}},
\end{align*}
where in the first equality we used the fact that $(W_N^{\beta_n})^\gamma$ is uniformly integrable for 
$\gamma\in(0,1)$ (since $\bbE\big[ W_N^{\beta_n}\big]=1$) and, thus, we can interchange limits and expectations. The last equality is then an explicit computation of the moment generating function of the normal distribution. Letting now $\hat\beta\uparrow 1$ and using the fact that $\sigma_{\hat\beta} \to \infty$ as $\hat\beta\uparrow 1$, we obtain that $\lim_{N\to\infty}\bbE[(W_n^{1/\sqrt{\sfR_n}})^\gamma)]=0$,
By monotonicity, we obtain that also the fractional moment of any $W_n^{\beta_n}$ corresponding to $\hat\beta\geq 1$ converges to $0$.
The convergence of the fractional moment to $0$ implies that $W_n^{\beta_n}$ tends to $0$ in distribution for all $\beta_n$ corresponding
to $\hat\beta\geq 1$. 
\vskip 2mm
Let us now sketch the log-normal convergence in the regime $\hat\beta<1$.
An important first observation has to do with the correct {\it {\bf time scale}} upon which one observes a change in the
fluctuations. To determine this time scale, one may look at the partition function of a system of length $tn$ for arbitrary $t>0$. 
Computing the variance of $W_n^{\beta_n}$ using the chaos expansion \eqref{eq:poly} and the orthogonality of its terms, one obtains that it is asymptotically independent of $t$ when $n\to\infty$. 
 One can be easily convinced about this fact by looking, for example, at the variance of the first
term in the chaos expansion, which behaves as 
\begin{align*}
\sigma(\beta_n)^2 \,\bbvar\Big( \sumtwo{1\leq t_1 \leq tn}{x_1\in \Z^2} q_{t_1}(x_1) \xi_{t_1,x_1}\Big) &= 
 \frac{\hat\beta^2}{\sfR_n}  \sumtwo{1\leq t_1 \leq tn}{x_1\in \Z^2} q_{t_1}(x_1)^2 = \hat\beta^2 \frac{R_{tn}}{R_n}
 \xrightarrow[n\to\infty]{} \hat\beta^2,
\end{align*}
which is independent of $t$, since $R_n$ is a slowly varying function.
Moreover, a similar computation shows that the contribution to the fluctuations from disorder $\xi_{i,x}$ sampled in the time interval
$[tn,n]$ is negligible, for any $t>0$ fixed.  
On the other hand, a similar computation shows a change in the fluctuations when at time scales
$n^t$, with $t>0$. These facts dictate that the meaningful time scale is not $tn$ but $n^t$ and that the partition function 
$W_n^{\beta_n}$ essentially depends only on disorder $\xi_{i,x}$ with $i/n\to 0$, as $n\to\infty$. 
We remark that this observation plays a crucial role in the study of asymptotic limits below the critical
temperature, not only for the directed polymer model but also for the other marginal models or stochastic
PDEs that fall in this framework \cite{CSZ18+, CSZ20, G18, DG22, DG23}.

To quantify the observation on the time scale,
 we decompose the summations over $t_1,...,t_k$ in the multilinear expansion \eqref{eq:poly}, over intervals
 $t_j-t_{j-1}\in I_{i_j}$, with $I_{i_j}=\big(n^{\frac{i_j-1}{M}}, n^{\frac{i_j}{M}} \big]$, $i_j\in\{1,...,M\}$ with $M$ being 
 a coarse graining parameter (which will eventually tend to infinity). 
 We can then rewrite the $k$-th term in the expansion \eqref{eq:poly} as  
\begin{align}\label{thetas}
&\frac{\hat\beta^k}{M^{k/2}}\sum_{1\leq i_1,...,i_k\leq M}  \Theta^{n,M}_{i_1,...,i_k} \qquad \text{where}\\
&\Theta^{n,M}_{i_1,...,i_k} :=
\left(\frac{M}{\sfR_N}\right)^{k/2}\sumtwo{t_j-t_{j-1}\in\, I_{i_j}\,\,\text{for }\,j=1,...,k}{x_1,...,x_k \in Z^d}\,\,\, \prod_{j=1}^k q_{t_j-t_{j-1}}(x_j-x_{j-1}) \,\xi_{t_j,x_j}.
\notag
\end{align}
For technical reasons, that should become obvious below, we are led to restrict the summation in \eqref{thetas} to the subset 
\begin{align*}
&i_1,...,i_k \in \{1,...,M\}_\sharp \qquad \text{with} \qquad \\
& \{1,...,M\}_\sharp := \{\boldsymbol{i}=(i_1,...,i_k) \colon |i_j-i_{j'}|\geq 2 ,\,\, \text{for all $j\neq j'$}  \}.
\end{align*}  
It is not difficult to justify, via an $L^2(\bbP)$ estimate, that this restriction has negligible error but we will omit the details.

We now observe that if an index $i_j$ is a running maximum for the $k$-tuple $\boldsymbol{i}:=(i_1,...,i_k)$, i.e. $i_j> \max\{i_1,...,i_{j-1}\}$ 
then $\big(n^{\frac{i_j-1}{M}}, n^{\frac{i_j}{M}} \big] \ni t_j \gg 
 t_{r} \in \big(n^{\frac{i_r-1}{M}}, n^{\frac{i_{r}}{M} }\big] $, for all
$r<j$, when $n\to\infty$.
This is the point where we also use the restriction into  $\{1,...,M\}_\sharp$, which imposes that indices
$i_j$ are well separated.
This  implies that
$ q_{t_j-t_{j-1}}(x_j-x_{j-1}) \approx q_{t_{j}}(x_j)$ for $t_j\in I_{i_j}$ and $t_{j-1}\in I_{i_{j-1}}$, where the
drop out of the spatial term $x_{j-1}$ makes use of the diffusive properties of the random walk. 
Thus, we will decompose the sequence $\boldsymbol{i}:=(i_1,...,i_k)$  according to its running maxima, i.e.
 $\boldsymbol{i}=(\boldsymbol{i}^{(1)},...,\boldsymbol{i}^{(\mathfrak{m})})$ identified by the properties: 
 \begin{itemize}
\item $\boldsymbol{i}^{(r)}:=(i_{\ell_{r}},...,i_{\ell_{r+1}-1})$ with
 $i_{1}=i_{\ell_1}<i_{\ell_2}<\cdots<i_{\ell_{\mathfrak{m}}}$ being the successive running maxima,
 \item $i_{\ell_{r}} > i_{\ell_{r}+1},...,i_{\ell_{r+1}-1}$, for $r=1,...,\mathfrak{m}$.
 \end{itemize}
  It can then be shown 
 that  \eqref{thetas} asymptotically factorizes for large $n$ and takes the form
\begin{align}\label{theta_deco}
 \frac{\hbeta^k}{M^{\frac{k}{2}}} \sum_{\boldsymbol{i}\in \{1,\ldots, M\}^k_\sharp}
\Theta^{n;M}_{\boldsymbol{i}^{(1)}} \Theta^{n;M}_{\boldsymbol{i}^{(2)}}\cdots \Theta^{n;M}_{\boldsymbol{i}^{(\mathfrak{m})}}.
\end{align}
The heart of the matter is to show that all the $\Theta^{n;M}_{\boldsymbol{i}^{(j)}}$ converge jointly,
 when $n\to\infty$ to standard normal variables. This is where we use the Fourth Moment Theorem
 \ref{thm:fourth}.
  Assuming this, let us see how we can obtain the convergence to the log-normal distribution in \eqref{eq:conve0} when $\hat\beta<1$. 
 
 We can start by replacing, using the Lindeberg Principle \ref{thm:Lind}, 
 the $\Theta^{N;M}_{\boldsymbol{i}^{(j)}}$ variables in \eqref{theta_deco} by standard normals, which we denote by $\zeta_{\boldsymbol{i}^{(j)}}$. Recall that $\boldsymbol{i}^{(j)}=(i_{\ell_{j}},...,i_{\ell_{j+1}-1})$ and since $i_{\ell_{j}}, j\geq 1$ are
 the running maxima we have $i_{\ell_{j}} > i_{\ell_{j}+1},...,i_{\ell_{j+1}-1}$.
 Then denoting by
  \begin{align*}
 \zeta_r(a):= \sum_{(a_2,...,a_r)\in \{1,...,a-1\}^{r-1}} \zeta_{(a,a_2,...,a_r)} ,
  \end{align*}
we have that $W_n^{\beta_n}$ is approximately (in the large $n$ limit)
 \begin{align*}
 W_{n}^{\beta_n} \approx 1+\sum_{k=1}^\infty \sum_{m=1}^k  \frac{\hat\beta^k}{M^{\tfrac{k}{2}}}
 \sumtwo{1\leq \ell_1<\cdots < \ell_m\leq k}{1\leq a_1<a_2<\cdots<a_m\leq M} \prod_{j=1}^m \zeta_{\ell_{j+1}-\ell_j}(a_j),
 \end{align*}
 where $m$ denotes the number of running maxima in the sequence $\boldsymbol{i}=(i_1,...,i_k)$ in \eqref{thetas} (thus determining the number of dominated sequences),
  $\ell_1,...,\ell_m$ denotes the location of the running maxima in $\boldsymbol{i}$ and $a_1,...,a_m$ denote the values of $i_{\ell_1},...,i_{\ell_m}$. We can continue by rewriting the
  above, via the change of variables $r_j:=\ell_{j+1}-\ell_j$ as
  \begin{align*}
  &1+
\sum_{k=1}^\infty\sum_{m=1}^k \frac{\hat\gb^k}{M^{\frac{k}{2}}}\,
\sumtwo{r_1,\dots r_m\in \bbN}{ r_1+\cdots+ r_m=k }\,\, \sum_{1\leq a_1<a_2<\cdots <a_m\leq M}
 \prod_{j=1}^m \,\,\zeta_{r_{j}}(a_j)  \notag
 \end{align*}
 and interchanging the $k$ and $m$ summations and using that $k=r_1+\cdots +r_m$, 
 we can write it as
 \begin{align} \label{bsZhb3}
&\quad 1 + \sum_{m=1}^\infty \
	\sum_{r_1, \ldots, r_m \in \N}
	\sum_{1 \le a_1 < a_2 < \ldots < a_m \le M} \  \prod_{j=1}^m
	 \frac{\hbeta^{r_j}}{M^{\frac{r_j}{2}}}\ \zeta_{r_j}(a_j) 
	  \notag\\
=&\quad 1 + \sum_{m=1}^\infty \
	\sum_{r_1, \ldots, r_m \in \N}
	\sumtwo{0 < t_1 < t_2 < \ldots < t_m \le 1}{t_1, \ldots, t_m \in \frac{1}{M}\N}
	\ \prod_{j=1}^m \frac{\hbeta^{r_j}}{M^{\frac{r_j}{2}}}\ \zeta_{r_j}(M t_j) 
	 \notag\\
=& \quad 1 + \sum_{m=1}^\infty
	\sumtwo{0 < t_1 < t_2 < \ldots < t_m \le 1}{t_1, \ldots, t_m \in \frac{1}{M}\N}
 	\ \prod_{j=1}^m \bigg\{\sum_{r\in\bbN}\frac{\hbeta^{r}}{M^{\frac{r}{2}}}\ \zeta_{r}(M t_j)\bigg\},
  \end{align}
In the above summation manipulations, the assumption $\hat\beta<1$, that ensures convergence in $L^2(\bbP)$
and the allows the interchange of summations, plays an important role.
  Since $(\hat\gb/\sqrt{M})^r \zeta_r(Mt)$ are normal random variables, independent for different values of $r\in\bbN$ and $t\in M^{-1}\bbN$, we have that the random variables
\[
\Xi_{M,t} := \sum_{r\in\bbN}\frac{\hbeta^{r}}{M^{\frac{r}{2}}}\ \zeta_{r}(M t), \qquad
t\in (0,1]\cap \frac{1}{M}\N,
\]
are also independent normal with mean zero and variance
$$
{\mathbb V}{\rm ar}(\Xi_{M,t}) =  \sum_{r\in\bbN}\frac{\hbeta^{2r}}{M^r}
{\mathbb V}{\rm ar}(\xi_r(Mt)) = \sum_{r\in\N} \frac{\hbeta^{2r}}{M^r} (Mt-1)^{r-1} = \frac{\hbeta^2}{M} \cdot \frac{1+\epsilon_M(t)}{1-\hbeta^2 t},
$$
with the error $\epsilon_M(t)$ easily seen to converge to $0$, uniformly in $t\in [0,1]$, as $M\to\infty$ for $\hbeta<1$.
We can, therefore, represent $\Xi_{M,t}$ in terms of a standard, one dimensional Wiener process $W$:
 \begin{align}\label{Xiencode}
\Xi_{M,t} = \frac{\hbeta (1+\epsilon_M(t))}{\sqrt{1-\hbeta^2 \, t}}
 \int_{t-\frac{1}{M}}^{t} \dd W_s
 = \int_{t-\frac{1}{M}}^{t}  \frac{\hbeta (1+\epsilon_M(t))}{\sqrt{1-\hbeta^2 \, t}} \dd W_s, 
 \qquad t\in [0,1]\cap \frac{1}{M}\N.
\end{align}
and we can rewrite \eqref{bsZhb3} as
\begin{align}\label{hatZMchaos}
1 + \sum_{m=1}^\infty
	 \sumtwo{0 < t_1 < t_2 < \ldots < t_m \le 1}{t_1, \ldots, t_m \in \frac{1}{M}\N}  \,\,\prod_{j=1}^m
	  \int_{t_{j}-\frac{1}{M}}^{t_j} \frac{\hbeta(1+\epsilon_M(t))}{\sqrt{1-\hbeta^2\, t_j}} \dd W_s.
\end{align}
So, for $\hbeta<1$, we have that \eqref{Xiencode} converges in $L^2(\bbP)$, for $M\to\infty$
(here, again, the condition $\hbeta<1$ is crucial), to
\begin{align*}
& 1 + \sum_{m=1}^\infty \ \
	\idotsint\limits_{0 < t_1 < \ldots < t_m < 1} \
	\prod_{j=1}^m \frac{\hbeta}{\sqrt{1-\hbeta^2\, t_j}} \dd W_{t_j} \\
&\hskip 3cm	=  \exp \Bigg\{ \int_0^1  \frac{\hbeta}{\sqrt{1-\hbeta^2\, t}} \dd W(t) -\frac{1}{2} \bbE\Big[ \Big( \int_0^1  \frac{\hbeta}{\sqrt{1-\hbeta^2\, t}} \dd W(t)\Big)^2\Big] \Bigg\} ,
\end{align*}
where the last equality holds by the properties of the Wick exponential \cite[\S 3.2]{J97}. Since $\int_0^1  \frac{\hbeta}{\sqrt{1-\hbeta^2\, t}} \dd W(t)$ is a Gaussian variable
with variance $\int_0^1\tfrac{\hat \beta^2 \,\dd t}{1-\hat\beta^2} = \log(1-\hat\beta^2)$, the result follows.
\end{proof}
The proof of Theorem \ref{fieldDPRM} follows the same lines as the proof just presented. 
However, the approach towards Theorem \ref{th:mainSHF} deviates significantly as at the 
critical temperature the exponential separation of the time scales is lost. 
In this case, the main contributions to the fluctuations
do not come  any more from noise $\{\omega_{i,x}\}$ with $i=o(n)$. Rather, fluctuations are driven by the
noise at scales of order $n$. Another important difference is that looking at the chaos expansion
\eqref{eq:poly} it is not any more the case that terms in chaos of order $k\leq n$ (meaning the $k^{th}$ term in expansion \eqref{eq:poly}) are discounted by a factor $\hat\beta^k$, as $\hat\beta=1$
at the critical case. 
On the contrary, it turns
out that the main contribution to the fluctuations comes from terms in the chaos of order $\log n$.
The fact that the main contribution to the fluctuations comes from chaoses of increasing order
appears to be related to the phenomenon of {\it noise sensitivity} \cite{G11, GS14}. This is also related to the question of whether
the limit (here in particular the Critical 2d SHF) is independent of the original driving noise. This is an interesting question for the
Critical 2d SHF, which is currently open.

 The methods that are employed 
towards Theorem \ref{th:mainSHF} still make use
of the Lindeberg principle but combined with a crucial coarse-graining decomposition. 
It is also crucial in the derivation of Theorem  \ref{th:mainSHF} that the variance of the coarse-grained disorder
has the {\it same} scaling as the microscopic critical scaling $\hat\beta=1$. We refer to
\cite{CSZ23} for details.

\begin{remark}\label{comparison}
{\rm
Let us close this section with a comparison to the choice of the change of measure in \eqref{Xchoice}.
The similarity with the chaos expansion \eqref{eq:poly} is apparent (note that for $\beta=\beta_n$,
it turns out that $\sigma(\beta_n)^2\approx \hat\beta^2/\sf R_n$). The choice of $\sfX$ in \eqref{Xchoice}
with $k$ of order $\log\log\ell$ in $d=2$ (as done in \cite{BL17}) 
is closely related to the fact that in $d=2$ and at the critical
temperature $\hat\beta=1$, the main contribution to the fluctuations of $W_n^{\beta_n}$ comes
from terms in the chaos expansion with growing order (in fact of order $\log n$ as discussed in the previous paragraph). 
The $\log\log$ in the choice of $\sfX$ in \eqref{Xchoice},
as opposed to a single log which captures the main contribution at criticality, 
is due to the fact that \cite{BL17} works very slightly off the critical temperature.
}
\end{remark}
\section{Heavy tail disorder}\label{sec:heavy}
In this section we will review a phase transition diagram, much of which is still conjectural, in dimension $1$ and in the
case when disorder has heavy tails determined by a tail exponent $\alpha \in (0,\infty)$ and suitable intermediate disorder
scaling $\beta_n=\hat\beta n^{-\gamma}$, with $\gamma\in [0,\infty)$. The phase diagram proposes a decomposition of the
parameter space $(\gamma, \alpha)\in [0,\infty)\times (0,\infty]$\footnote{the inclusion of $\alpha=\infty$ intends to also cover the 
case of disorder with exponential moments for different intermediate disorder scalings} 
to regimes and level lines corresponding to different fluctuation exponents $\chi,\xi$,
which are, roughly, defined to determine the asymptotic behaviours:
\begin{align*}
\log Z_n^{\beta_n} -\sff(\beta_n) \approx n^{\chi}\qquad \text{and} \qquad
\bbE \,\E^{\mu_n^{\beta_n}}\big[ |S_n| \big] \approx n^{\xi},
\end{align*}
as $n\to\infty$ and for suitable centering constant $\sff(\beta_n)$. Note that in this section we will be working with the unnormalised 
partition function $Z_n^{\beta_n}$, instead of $W_n^{\beta_n}$ as the log-moment generating function $\lambda(\beta)$,
which is used in the normalisation $W_n^{\beta_n}$, is not defined for heavy tail disorder.
\subsection{Heavy tails and intermediate disorder phase diagram.}
\begin{assumption}[Heavy tail disorder]\label{ass:heavy}
We assume that for every $(i,y)\in\N\times\Z$, disorder $\omega_{i,y}$ has the following properties:
\begin{itemize}
\item[{\rm (i)}] The cumulative distribution function $F(x):=\bbP(\omega_{i,y}\leq x)$ has regularly varying right tail, in the sense that 
\begin{align}\label{aF}
	\bar{F}(x):=1-F(x)=x^{-\ga}L(x),\quad \text{ for all } \,x>0,
\end{align}
for some exponent $\alpha>0$ and a slowly varying function $L(x)$. The latter means that for any $t>0$ 
we have $L(tx)/L(x)\to 1 $, as $x\to\infty$. 

\item[{\rm (ii)}] When $\alpha> 2$, we will assume that 
\begin{align*}
\bbE[\omega_{i,y}]=0,\quad\bbE[\omega_{i,y}^2]=1.
\end{align*}

\item[{\rm (iii)}] When $\alpha \leq 2$, we  assume that 
\begin{align*}
F(-x)=(c_-+o(1))\bar{F}(x),\quad \text{ as } \,x\to\infty,
\end{align*}
for some $c_-\ge 0$, meaning that, in this case, the left tail is dominated by the right tail. 
\end{itemize}
\end{assumption}
\begin{remark}
In the above set of assumptions $\alpha$ can extend to
$\alpha=\infty$, in which case, we assume finite exponential moments.
\end{remark}
We will also need the function 
\begin{align}\label{def:m}
m(t):=\inf\{ x \colon \bar F(x)\leq t \}
\end{align}
If $\bar F$ satisfies \eqref{aF} with an exponent $\alpha$, it is easy to derive that $m(t)$ has the
asymptotic behaviour $m(t)=t^{1/\alpha} L_0(t)$, for $t\to \infty$ for some slowly varying function 
$L_0(\cdot)$.

The phase space $(\gamma,\alpha)\in [0,\infty)\times (0,\infty]$ is divided into the following regions 
corresponding to different exponents:
\begin{align*}
\setlength\arraycolsep{0.01em}
\begin{array}{rlrl}
	R_1 &:= \big\{(\gamma,\alpha): \gamma>\frac14, \gamma\ge\frac{3}{2\alpha} \big\},
	&R_2 &:= \big\{(1/4,\ga): \alpha \geq 6 \big\}, \\[2mm]
	R_3 &:=\big\{(\gamma,\alpha): 0<\gamma<1/4, \ga \ge \frac{5-2\gamma}{1-\gamma}\big\}, 
	&R_4 &:= \big\{(0,\ga): \alpha> 5 \big\},\\[2mm]
	R_5 & := \big\{(\gamma,\ga): \alpha> 1/2,  \max\{0,\frac{2}{\alpha}-1, \frac{\ga-5}{\ga -  2}  \} < \gamma < \frac{3}{2\alpha} \big\}, \\[2mm]
	R_6 &:= \big\{(\gamma,\ga):  0< \alpha < 2, \gamma= \frac{2}{\alpha}-1 \big\} 
	\qquad\quad \text{     and }  
	&R_7 &:=  \big\{(\gamma,\ga): 0< \alpha < 2, 0 \le \gamma< \frac{2}{\ga}-1 \big\}.\\[2mm]
\end{array}
\end{align*}

The region $R_2\cup R_3\cup R_4\cup R_5\cup R_6$ 
can (conjecturally) be decomposed into level curves over which exponents $\xi$  and $\chi$ stay constant.
More precisely, for any fixed value of $\xi\in[1/2,1]$, the curve in the $(\gamma,\alpha)$ plane 
which gives rise to a specific value of $\xi$ and $\chi=2\xi-1$ is given by 
\begin{align}\label{eq:level}
\Big\{ (\ga,\gamma) \colon \xi= \frac{1+\ga(1-\gamma)}{2\ga-1}\,\,\text{and}\,\,\ga\le \frac{5-2\gamma}{1-\gamma} \Big\}
\,\,\bigcup \,\,\Big\{(\ga,\gamma) \colon \xi=\frac{2(1-\gamma)}{3}\,\,\text{and}\,\,\ga\ge \frac{5-2\gamma}{1-\gamma}  \Big\},
\end{align}
We also refer to Figure \ref{fig-heavy} for a depiction. The blue curve given by 
$\gamma=\max\{0,\frac{5-\alpha}{3\alpha} \}$ is conjecturally the curve in the $(\alpha,\gamma)$ diagram
where the exponents $(\chi,\xi)=(\frac{1}{3}. \frac{2}{3})$ should be attained, while the red curve
$\gamma=\max\{ \frac{1}{4}, \frac{3}{2\alpha}\}$ is (this has actually been proved and we will detail on this 
below) the curve that corresponds to a weak disorder regime with exponents $(\chi,\gamma)=(0,\frac{1}{2})$. Curves of similar shape, i.e. curved inside the red region and continuing straight inside the
blue, are conjecturally spanning the parameters $(\chi,\xi)$ with $\xi$ ranging between $[1/2,1]$ 
and $\chi=2\xi-1$. We note that the KPZ hyper-scaling relation $\chi=2\xi-1$ fails in region $R_1\cup R_7$:
 Inside $R_1$ we still have $\xi=\frac12$ but $\chi=\frac14-\gamma<2\xi-1$. 
 Similarly, in $R_7$ we have $\xi=1$ and $\chi= \frac{2}{\alpha}-\gamma>2\xi-1$. 
\begin{figure}
	\centering                                                        
	\includegraphics[height=70mm]{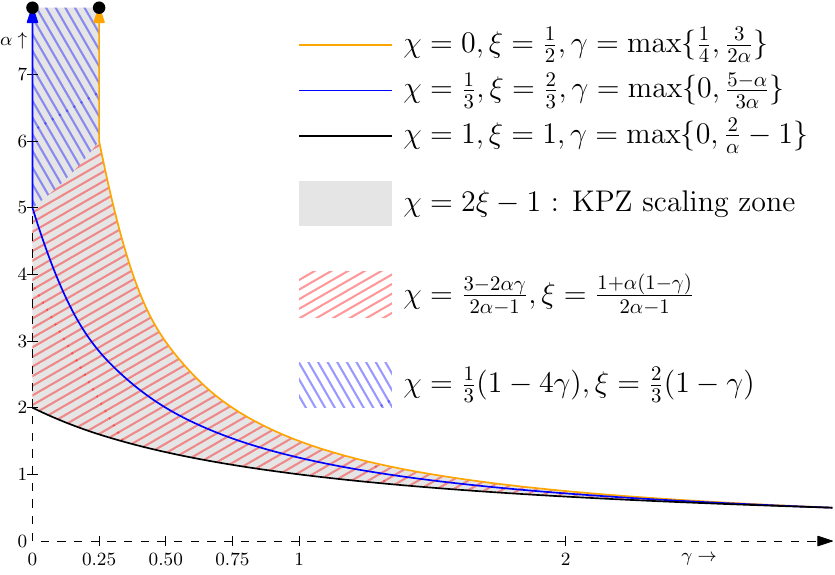}                                                 
	\caption{Phase diagram for fluctuation and transversal exponents} 
	\label{fig-heavy}                                                      
\end{figure}
\vskip 2mm
{\bf Heuristic derivation of the phase diagram.} We can distinguish three mechanisms that
drive fluctuations, which we will detail. The terminology was coined in the very interesting paper \cite{GDBR16} where we refer for conjectures, heuristics and numerics that go beyond what we present here.
We will discuss the mechanisms that drive fluctuations and give rise to the conjectured asymptotics
afterwards. The three regimes are:
\begin{itemize}
\item[{\bf I. Collective optimisation.}] In this case, the polymer is attracted to narrow corridors where 
the energy $H_n^\beta(S)$ is large. In this case, KPZ and Airy process mechanisms prevail.
\item[{\bf II. Elitist.}] This mechanism emerges when heavy tails of the disorder start having a
distinguishing effect and the polymer is attracted to a relatively small (but still significant) 
fraction of sites where disorder is exceptionally
high. In this situation energy starts being balanced by entropy. 
\item[{\bf III. Individual.}] Heavy tails become dominant and the polymer is attracted to a few 
(finite number of) sites where disorder takes large values. 
\end{itemize}

{\bf Heuristics on the Collective regime.}
Let us present the heuristic argument from \cite{AKQ10}.
As already mentioned, this is the situation that no sites carry exceedingly large disorder and fluctuations
are driven by a collective optimisation of the energy. This is very difficult to track and one 
resorts to the detailed information coming from one-dimensional KPZ asymptotics, which 
provide the process limit
\begin{align}\label{Airy_asym}
\Big\{ \frac{\log Z_n^\beta(0,x n^{2/3}) - n\sff(\beta)}{ c(\beta) n^{1/3}} \Big\}_{x\in \R}
\xrightarrow[n\to\infty]{d} \big\{{\rm Airy}_2(x)- x^2 \big\}_{x\in \R},
\end{align}
where the argument $(0,xn^{2/3})$ denotes a starting point $0$ and an end point $xn^{2/3}$
and ${\rm Airy}_2(\cdot)$ is the Airy-2 process. 
This asymptotic limit is largely open for polymer models at positive temperature, even in the integrable cases.
Currently, only the 1-point asymptotics have been established for the log-gamma and
O'Connell-Yor polymers  \cite{COSZ14, OC12, BCR13} and  formal asymptotics only for the
2-point function \cite{NZ17} (see also \cite{D23} for rigorous asymptotics for the 2-points function 
of the closely related six-vertex model). The analogue of \eqref{Airy_asym} for zero-temperature,
i.e. last passage percolation has been established for exponential or geometric random variables \cite{PS02, J03}.
Let us assume \eqref{Airy_asym}, and that it is valid in 
an intermediate disorder scaling $\beta_n=\hat\beta n^{-\gamma}$ where collective behaviour dominates.
Then we can expect
\begin{align*}
\log Z_n^{\beta_n}(0,x) - \log Z_n^{\beta_n}(0,0) 
\approx \beta_n \, n^{1/3}\big({\rm Airy}_2\big(\frac{x}{n^{2/3}} \big) -{\rm Airy}_2(0) \big) -\frac{x^2}{n}.
\end{align*}
We have looked at the difference $\log Z_n^{\beta_n}(0,x) - \log Z_n^{\beta_n}(0,0) $
since the Airy process is stationary and so we should have a reference point.
Taking advantage of the fact that the Airy process has locally the regularity of Brownian motion, we can
deduce that 
\begin{align*}
\log Z_n^\beta(x) - \log Z_n^\beta(0) 
\approx \beta_n \, x^{1/2} -\frac{x^2}{n}  = n^{-\gamma} x^{1/2} -\frac{x^2}{n} .
\end{align*}
Optimising over $x$ leads to $x\approx n^{\frac{2}{3}(1-\gamma)}$, which can be interpreted as the 
transversal fluctuation exponent being 
\begin{align}\label{collective-exp}
\xi=\frac{2}{3}(1-\gamma)\qquad \text{and, thus,} \qquad
\chi=2\xi-1 = \frac{1}{3}(1-4\gamma).
\end{align}

{\bf Heuristics on the Elitist regime.} 
We present here a heuristic proposed in \cite{DZ16}, which has the flavour of a {\it Flory type} argument
\cite{BBP07}.
Assume that $\bar F(x)= x^{-\alpha+o(1)}$ and that the polymer path fluctuates in a range
$n^\xi$ for some $\xi\in [\frac{1}{2},1]$. In other words, let us assume that the polymer path
spans the box $B_{n,\xi}:=\big([0,n]\times (-n^\xi,n^\xi)\big)\times \Z^2$. By standard results on order statistics
of heavy-tail random variables \cite{LLR83}, we have that, typically,
\begin{align}\label{max}
\max\{ \omega_v \colon v\in B_{n,\xi}\} \approx |B_{n,\xi}|^{\frac{1}{\alpha}} = n^{\frac{1+\xi}{\alpha}}. 
\end{align}
In fact, let us denote by $(\go^{(j)}, j=1,2,\ldots,|B_{n,\xi}|)$ the order statistics of the set $\{ \go_v\colon v\in B_{n,\xi} \} $, i.e. all the values of the disorder inside $B_{n,\xi}$ ordered from largest to smallest.
Let also $(t^{(j)}, x^{(j)})$ for $j=1,...,|B_{n,\xi}|$ be the site inside box $B_{n,\xi}$ 
where $j^{\text{th}}$ largest value of disorder is attained. Again, by standard results in order statistics,
 we have that, for any finite $k$
\begin{align}\label{eq:orderstat}
\{n^{-\frac{1+\xi}{\alpha} }\omega^{(j)}, \,n^{-1} t^{(j)}, \,n^{-\xi} x^{(j)}  \}_{j=1,...,k}
 \,\xrightarrow[n\to\infty]{d}\, \{ {\mvw}^{(j)},\, \mvt^{(j)}, \mvx^{(j)} \}_{j=1,...,k},
\end{align}
where $\mvw^{(j)}$ are non-trivial random variables and the locations 
$\mvt^{(j)}, \mvx^{(j)}$ are independent and uniformly distributed in $[0,1]^2$.
By moderate deviations for simple random walks, we have that
\begin{align}\label{moderate}
-\log \P(S_{n \mvt^{(j)}} = n^\xi \mvx^{(j)}) \approx n^{2\xi-1} \frac{(\mvx^{(j)})^2}{2 \mvt^{(j)}}, 
\qquad \text{for $n\to\infty$.}
\end{align}
If we assume that heavy tails drive fluctuations, the energy (from \eqref{max}) -entropy 
(from \eqref{moderate}) equilibrium gives
 \begin{align}\label{eq:en-en}
 \log Z_n^{\beta_n} \approx \beta_n n^{\frac{1+\xi}{\alpha}} - n^{2\xi-1}
 =n^{\frac{1+\xi}{\alpha}-\gamma}- n^{2\xi-1}.
 \end{align}
The optimum will be achieved if the exponent $\xi$ is such that the two terms in the right-hand
side balance out, which means that $\frac{1+\xi}{\alpha}-\gamma=2\xi-1$, giving the value for
the transveral exponent $\xi$:
\begin{align}\label{eq:xi-red}
\xi=\frac{1+\alpha(1-\gamma)}{2\alpha-1} .
\end{align} 
Given that, by definition of the exponent $\chi$, we have 
$n^{\chi+o(1)}\approx \log Z_n^{\beta_n}$,  \eqref{eq:en-en} 
also provides the relations
\begin{align}\label{eq:chi-red}
\chi=2\xi-1= \frac{3-2\alpha\gamma}{2\alpha-1}.
\end{align}

{\bf Transition between Collective and Elitist behaviours.}  When the disorder has gradually heavier
tails, i.e. decreases down from $\alpha=\infty$, 
there should be a point where the optimisation will favour one Elitist strategy over the Collective. 
The transition
point is determined when the exponents $\chi_{\rm collective}$ from \eqref{collective-exp} and
$\chi_{\rm elitist}$ from \eqref{eq:chi-red} are equal, leading to the curve separating the blue and the red
regions in Figure \ref{fig-heavy}, with the curve being given by
\begin{align*}
 (\alpha,\gamma) \colon \quad
\chi_{\rm collective}=\frac{1}{3}(1-4\gamma)= \frac{3-2\alpha\gamma}{2\alpha-1} = \chi_{\rm elitist}
\quad \Longleftrightarrow \quad
\alpha=\frac{5-2\gamma}{1-\gamma}.
\qquad 
\end{align*}
\vskip 4mm
Let us now state a theorem that concerns the behaviour in the collective regime but with an intermediate disorder scaling
that induces a diffusive behaviour for the polymer, i.e. it corresponds to an exponent $\xi=1/2$ and corresponds to the outer,
red curve of the diagram in Figure \ref{fig-heavy} for values of $\alpha>1/2$. The part of this curve for values 
of $\alpha<1/2$, which prompts into the individual regime will be described in Theorem \ref{thm:heavy}. 
Both theorems were proved in 
\cite{DZ16} and we refer there for the proofs and heuristics. Item (i) of Theorem \ref{thm:DZ} was conjectured in \cite{AKQ14}. 
\begin{theorem}\label{thm:DZ}
Let $\gamma>0$ and $\beta_n n^{\gamma} \to \hat \beta $, for some $\hat\beta\geq 0$.
Consider a heavy-tail disorder $\omega$ satisfying Assumptions \ref{ass:heavy} with certain exponent
$\alpha$. We have the following three asymptotic regimes for the partition $Z_n^{\beta_n}$:
\begin{itemize}
\item[\rm(i)] {\rm Regime $\alpha \geq 6$ and  $\gamma \geq \frac{1}{4}$}. 
If $\beta_n n^{1/4}\to \hat\beta $, with $\hat\beta>0$, then 
\begin{align*}
\log Z_n^{\beta_n} - n \log \bbE\Big[ e^{-\beta_n\omega_-} + \sum_{i=1}^4 \frac{\beta_n^i}{i!} \omega^i_+\Big] \xrightarrow[n\to\infty]{d} u_{\sqrt{2}\hat\beta}(0),
\end{align*}
with $u_{\sqrt{2}\hat\beta}(0)$ the solution to the stochastic heat equation \eqref{eq:SHE} evaluated at $0$ and with flat 
initial condition.
If $\beta_n n^{1/4} \to 0$, then 
\begin{align*}
\frac{1}{\beta_n n^{1/4}} 
\Big(
\log Z_n^{\beta_n} - n \log \bbE\Big[ e^{-\beta_n\omega_-} 
+ \sum_{i=1}^3 \frac{\beta_n^i}{i!} \omega^i_+\Big]
\Big)
 \xrightarrow[n\to\infty]{d} \cN(0,2\pi^{-1/2}),
\end{align*}
with $\cN(0,2\pi^{-1/2})$ a normal with mean $0$ and variance $2\pi^{-1/2}$.
\item[\rm (ii)] {\rm Regime $\alpha \in (2,6]$ and  $\gamma \geq \frac{3}{2\alpha}$}. 
If $\beta_n m(n^{3/2})\to\hat\beta \geq 0$, then 
\begin{align*}
\frac{1}{\beta_n n^{1/4}} 
\Big(
\log Z_n^{\beta_n} - n \log \bbE\Big[ e^{-\beta_n\omega_-} 
+ \sum_{i=1}^2 \frac{\beta_n^i}{i!} \omega^i_+\ + \frac{\beta_n^3}{3!} \omega^3_+ \ind_{\alpha>3} \Big]
\Big)
 \xrightarrow[n\to\infty]{d} \cN(0,2\pi^{-1/2}).
\end{align*}
\end{itemize}
\end{theorem}

{\bf Individual regime.} In this regime the process of picking the largest weights becomes more pronounced.
To state the first result for $\ga\in(\frac12,2)$, we need to consider a Poisson point process $\cP$ on $(w,t,x)\in \cS:=\R\times [0,1]\times \R$ with intensity measure $\eta(\dd w \dd t \dd x) =\frac12\ga |w|^{-1-\ga}(\ind_{w>0}+c_-\ind_{w<0})\dd w \dd t \dd x$.

With $g_t(x)$ being the heat kernel we also define the random variables
\begin{align*}
\cW_{0}^{(\alpha)} =
\begin{cases}
 \int_{\cS} w g_t(x) (\cP-\eta)(\dd w \dd t \dd x) &\text{ if } \ga\in (1,2)\\
\int_{\cS\cap\{ |w|>1\}} w \rho(x,t) \cP(\dd w \dd t \dd x)  +\int_{\cS\cap\{ |w|\le 1\}} w g_t(x) (\cP-\eta)(\dd w \dd t \dd x) &\text{ if } \ga=1\\
\int_{\cS} wg_t(x)\cP(\dd w \dd t \dd x) &\text{ if } \ga\in (0,1),
\end{cases}
\end{align*}
and for $\gb\in(0,\infty)$,
\[
\cW_{\gb}^{(\ga)} 
= \frac{1}{\gb} \int_{\cS} (e^{\gb w}-1-\gb w)g_t(x)\cP(\dd w \dd t \dd x) + \cW_{0}^{(\ga)}.
\]
For every $\gb\ge 0, \ga\in(1/2,2)$, the random variables $\cW_{\gb}^{(\ga)}$ are finite a.s. and 
have stable distribution (see Lemma 1.3 in \cite{DZ16}).

The first result in the regime $\alpha\in(1/2,2)$ considers the range $\gamma\geq 3/2\alpha$ where the
polymer still maintains diffusive behaviour. This result was obtained in \cite{DZ16}.
Recall $m(t)$ from definition \eqref{def:m}.
\begin{theorem}[$1/2<\alpha<2, \gamma \geq 3/2\alpha$]\label{thm:heavy}
Under Assumption \ref{ass:heavy} (iii), i.e. $\alpha\in(1/2,2)$ and $c_-\geq 0$.
Let $\beta_{n}$ be a sequence of real numbers such that $\beta_{n} m(n^{3/2})$ converges to 
$\beta\in [0,\infty)$, as $n\to\infty$. Then 
\[
	\frac{\sqrt{n}}{\beta_n m(n^{3/2})}\Bigl(\log Z_{n}^{\beta_n} - n\gb_n\E\big[\omega
	\ind_{|\omega|\leq m(n^{3/2})}\big]\,\ind_{\{\alpha=1\}}\Bigr) \xrightarrow[n\to\infty]{d} 2\cW_{\gb}^{(\alpha)} .
\]
\end{theorem}
The result of the above theorem says that the fluctuations of the partition function are energy dominated,
driven by catching the few largest weights. 
The situation below the curve $\gamma=\tfrac{3}{2\alpha}$ is governed by energy-entropy variational problems.
The following theorem was proved in \cite{BT18,BT19} and resolved a conjecture in \cite{DZ16}. 
First, we need to define the entropy of a path
in the class of a.e. differentiable curves by
\begin{align}\label{eq:Ent1}
{\rm Ent}(s):=\int_0^1 (\dot s(t))^2 \,\dd t .
\end{align}
\begin{theorem}
Consider Assumption \ref{ass:heavy} with $\alpha\in (1/2,2)$, $\beta_n = \hat\beta n^{-\gamma}$ with 
$\gamma\in (\tfrac{2}{\alpha}-1,\tfrac{3}{2\alpha})$ and $c_-=0$. If $\chi=2\xi-1$ and 
$\xi=\frac{1+\alpha(1-\gamma)}{2\alpha-1}$, which is in $(\frac{1}{2},1)$ for the considered range of 
$\gamma$, we have that
\begin{align*}
\frac{1}{n^\chi} \big( \log Z_n^{\beta_n} -n\beta_n \bbE[\omega]\ind_{\alpha\geq 3/2} \big)
\xrightarrow[n\to\infty]{d} \sup_{s\colon {\rm Ent}(s)<\infty} \Big\{ \beta\pi(s)- {\rm Ent}(s)\Big\},
\end{align*}
where, for the Poisson point process $\cP$ defined above, we define the energy
$\pi(s):=\sum_{(w,t,x)\in \cP} w\ind_{\{ (t,x)\in s \}}$.
\end{theorem}
Part of the conclusion in the above theorem is that the variational problem is well defined, i.e. finite, 
in the above range of $(\alpha,\gamma)$. As proven in \cite{BT19}, Theorem 2.4, this is not the case when 
$\alpha\in (0,1/2]$.
In this case, the limit is characterised by another variational problem, which takes into account the
fact that in this regime the path will move in a range of order $n$, i.e. ballistic, and the vatiational
problem involves the entropy induced from the large deviation behaviour, while \eqref{eq:Ent1} arises
from moderate deviations. The relevant theorem, which actually characterises the limit below the 
curve $\gamma=\frac{2}{\alpha}-1$ was proven in \cite{AL11} and we state it below. Before doing so, though, let us point that \cite{BT19} established several more phases in particular close to the line
$\gamma=\frac{2}{\alpha}-1$ and we refer there for details on the rich phenomenology. 

The theorem of Auffinger-Louidor \cite{AL11} is the following:
\begin{theorem}
Recall definition \eqref{def:m} and define
\begin{align*}
\widehat{\rm Ent}(s):=\int_0^1 e(\dot s(t)) \,\dd t, \quad \text{with}
\quad 
e(x):=\frac{1}{2}\big((1+x)\log(1+x)+(1-x)\log (1-x) \big).
\end{align*}
Under Assumption \ref{ass:heavy} with $\alpha\in (0,2)$ and $\beta_n$ such that
$\frac{1}{n} m(n^2)\beta_n\to\hat \beta \in (0,\infty]$, we have that
\begin{align*}
\frac{1}{m(n^2)\beta_n} \log Z_n^{\beta_n}\xrightarrow[n\to\infty]{d} 
 \sup_{s\colon {\rm Ent}(s)<\infty} \Big\{ \pi(s)- \frac{1}{\hat\beta}{\widehat {\rm Ent}}(s)\Big\}.
\end{align*} 
\end{theorem}
We mention that the above theorem is a generalisation to a positive temperature setting of 
result of Hambly-Martin \cite{HM07}.

We close this section with the natural question:
\begin{question}
Derive the limiting behaviours in the rest of the phase diagram. Do Airy and KPZ asymptotics extend 
when $\gamma=0$ and $\alpha>5$ ? 
\end{question}
 
 \subsection{Heavy tails and weak-to-strong disorder}\label{sec:viv}
 In this section we report on an interesting phase transition for a variation of the directed polymer
 model with disorder in the domain of attraction of stable laws. The result here is from the work of
 Viveros \cite{V21}. Similar phase transition was observed earlier in the case of the disordered pinning 
 model \cite{L16}.
 
 More precisely let us consider the partition function
 \begin{align}\label{hatzeta}
 \hat Z_n^{\beta}:= \prod_{i=1}^n (1+\beta \xi_{i,S_i}).
 \end{align}
 It is worth comparing $ \hat Z_n^{\beta}$ with \eqref{eq:poly} and the derivation there.
The assumptions on the disorder
 $(\xi_{i,x} \colon i\in \N, x\in \Z^d)$ are that they form an i.i.d. field and that
 \begin{equation}\label{heavy-cond}
 \begin{split}
 \xi_{i,x} \geq -1,\quad \bbE[\xi_{i,x}]=0 \quad \text{and} \quad
 \bbP(\xi_{i,x}>u) \sim C x^{-\gamma}, \quad \text{for $\gamma\in(1,2)$}.
 \end{split}
 \end{equation}
 The assumption that the tail exponent $\gamma$ is strictly less than $2$, implies that the 
 components in a chaos decomposition \eqref{eq:poly} do not posses second moment.
 
 The interesting phenomenon observed is that there is a critical value of the parameter $\gamma$,
 equal to $1+\frac{2}{d}$, below which there is no weak disorder phase (in the same sense as 
 we have encountered so far in Definition \ref{def:weak}) while a weak disorder phase exists for 
 $\gamma>1+\frac{2}{d}$. The relevant theorem from \cite{V21} is the following:
 \begin{theorem}
Assume \eqref{heavy-cond} for the partition function $\hat Z_n^{\beta}$ in \eqref{hatzeta}. If
$\gamma \leq 1+\frac{2}{d}$, then very strong disorder holds for all $\beta>0$ in every dimension 
$d\geq 1$. Moreover, in $d\geq 3$, if $\gamma<1+\frac{2}{d}$ and $\hat\sff(\beta)$ the
free energy of $\hat Z_n^\beta$, then it holds that
\begin{align*}
\lim_{\beta\to 0} \frac{\log \hat \sff(\beta)}{\log \beta} = \frac{2}{d}\frac{\gamma }{1+\frac{2}{d}-\gamma},
\end{align*}
with the limit being infinity if $\gamma=1+\frac{2}{d}$. 
On the contrary, if $\gamma>1+\frac{2}{d}$, then there exists $\beta_c>0$ such that weak disorder 
holds for $\beta<\beta_c$ and $d\geq 3$. 
\end{theorem}
The proof of this result follows the usual format of fractional moment method, as outlined in Section 
\ref{sec:frac}, and second-to-first moment estimates, in the a spirit similar to \cite{L10, BL17}.

Let us also mention a related work \cite{BCL23}, where the stochastic heat equation driven by a L\'evy noise
is studied. This can be viewed as a continuum model of directed polymer in heavy tail disorder.
The corresponding continuum polymer measure in a L\'evy environment was constructed in \cite{BL22}.

\begin{remark}{\rm 
It is very worth comparing the critical value $\gamma_c=1+\frac{2}{d}$ with the lower bound for 
$p^*(\beta)$ in Theorem \ref{thm:propp*}. We refer to a \cite{J22b} for an interesting discussion 
on some possible links between these two critical values. 
}
\end{remark}
\section{polymers on more general graphs}\label{sec:gengraph}

In this section we will mainly give pointers to studies of the directed polymer model on graph structures
other than $\Z^d$. Most of the graphs that we will discuss have a built-in renormalisation structure,
which allows for more independence and, thus, more exact computations. 
\vskip 2mm
{\bf Directed polymer on trees.}
We can consider the case of discrete, $d$-regular trees or continuous, Galton-Watson trees.
Tree structures can be thought of as a mean-filed limit of directed polymer on the lattice, when the
lattice dimension tends to infinity. One feature that trees have and which makes the study of polymers
easier is that two independent copies of the random walk will not meet again once they split and follow
different branches. 

Directed polymers on trees were studied by Derrida-Spohn \cite{DS88} and Chauvin-Rouault \cite{CR88}.
For a $b$-ary tree, i.e. a tree where every vertex branches out to $b$ subsequent vertices,
 the precise critical value $\beta_c$ has been determined as the solution to the equation
$\beta\lambda'(\beta)-\lambda(\beta) = \log b$, see \cite{C16}, Section 4.3. Chapter 4 in \cite{C16} also
discusses polymers on more general tree structures.

In the setting of continuous trees with binary branching at rate $\lambda$, 
the partition function of the directed polymer is related to Branching Brownian motion. 
The Laplace transform of the corresponding 
partition function $Z^{\beta,{\rm branch}}_t(x)$, 
$u_t(x):=\bbE\big[ \exp\big( -e^{-\beta x} Z^\beta_t(x)\big)\big]$,
solves the Kolmogorov-Petrovsky-Piscounov (KPP) equation
\begin{align*}
\partial_t u &= \frac{1}{2} \Delta u +\lambda(u^2-u), \quad t>0, \,\, x\in \R,\qquad
u(0,x)=e^{-\beta x}, \quad x\in \R.
\end{align*}
The KPP equation admits travelling wave solutions and it turns out the speed of the wavefront is related to the
free energy of the polymer partition function. We refer to \cite{DS88} for details. 
 \begin{figure}[t]
 \begin{center}
\begin{tikzpicture}[scale=.6]
\node[draw,circle,inner sep=1pt,fill] at (0,-2) {}; \node[draw,circle,inner sep=1pt,fill] at (0,2) {};
\draw (0,-2)--(0,2);
\draw[ultra thick, red] (0,-2)--(0,2);
\node[draw,circle,inner sep=1pt,fill] at (4,-2) {}; \node[draw,circle,inner sep=1pt,fill] at (4,2) {};
\node[draw,circle,inner sep=1pt,fill] at (2,0) {}; \node[draw,circle,inner sep=1pt,fill] at (6,0) {};
\draw (4,-2)--(2,0)--(4,2)--(6,0)--(4,-2);
\draw[ultra thick, red] (4,-2)--(2,0)--(4,2);
\node[draw,circle,inner sep=1pt,fill] at (10,-2) {}; \node[draw,circle,inner sep=1pt,fill] at (10,2) {};
\node[draw,circle,inner sep=1pt,fill] at (8,0) {}; \node[draw,circle,inner sep=1pt,fill] at (12,0) {};
 \node[draw,circle,inner sep=1pt,fill] at (10.5,-0.5) {};  \node[draw,circle,inner sep=1pt,fill] at (11.5,-1.5) {};
 \draw (12,0)--(11.5,-1.5)--(10,-2)--(10.5,-0.5)--(12,0);
  \node[draw,circle,inner sep=1pt,fill] at (9.5,-0.5) {};  \node[draw,circle,inner sep=1pt,fill] at (8.5,-1.5) {}; 
   \draw (10,-2)--(8.5,-1.5)--(8,0)--(9.5,-0.5)--(10,-2);
   \node[draw,circle,inner sep=1pt,fill] at (9.5,0.5) {};  \node[draw,circle,inner sep=1pt,fill] at (8.5,1.5) {}; 
   \draw (10,2)--(8.5,1.5)--(8,0)--(9.5,0.5)--(10,2);
    \node[draw,circle,inner sep=1pt,fill] at (10.5,0.5) {};  \node[draw,circle,inner sep=1pt,fill] at (11.5,1.5) {};
 \draw (12,0)--(11.5,1.5)--(10,2)--(10.5, 0.5)--(12,0);
\draw[ultra thick, red] (10,2)--(11.5,1.5)--(12,0)--(10.5, -0.5)--(10,-2);
\end{tikzpicture}
 \end{center}  
\caption{ \small The first three graphs in sequence of hierarchical lattices corresponding to $b=2, s=2$. The red paths depict the directed polymers on each lattice. }
\label{fig:hierarchy}
\end{figure}

\vskip 2mm
{\bf Directed polymer on hierarchical lattice.}
Hierarchical lattices present a graph structure which lies in between the solvable tree structure (where
independence is very pronounced) and the more complex $\Z^d$ lattice structure. It was originally introduced in the physics literature as a model amenable to renormalisation analysis in spin systems
\cite{M75, K76}.
Hierarchical lattices are constructed recursively in the following way. 
We start with two parameters: the branching parameter $b\in\{2,3,...\}$ and the segmenting parameter $s\in\{2,3,...\}$. The $0^{th}$ lattice $D_0^{b,s}$ is just an edge connecting two vertices, say $A,B$.
$D_1^{b,s}$ is obtained by replacing the singe edge in $D_0^{b,s}$ by $b$ branches from $A$ to $B$
and then splitting each branch into $s$ segments.
Assume we have constructed lattices $D_{n-1}^{b,s}$. The lattice $D_n^{b,s}$ is constructed 
by replacing each edge in $D_{n-1}^{b,s}$ by a copy of  $D_1^{b,s}$.
An example of a hierarchical lattice with $b=s=2$ is shown in Figure \ref{fig:hierarchy}.
The partition function $W_n^\beta$ is defined analogously to the $\Z^d$ case but now with the 
polymer running from $A$ to $B$ on lattice $D_{n}^{b,s}$.

The study of directed polymer on hierarchical lattices was initiated in works of 
Cook-Derrida \cite{CD89} and Derrida-Griffiths \cite{DG89}. In \cite{LM10} hierarchical lattice with site
disorder was studied and it was shown that: (i) strong disorder holds for all $\beta>0$, if $b\leq s$ and
(ii) a weak-to-strong phase transition takes place if $b>s$. The analogy between the relations between
hierarchical lattices with parameters $b,s$ and directed polymers on $\Z^d$ is that 
$b<s$ corresponds to $d=1$, $b=s$ corresponds to the
critical dimension $d=2$ and $b>s$ corresponds to $d\geq 3$. 

An intermediate disorder scaling for polymers on hierarchical lattices with critical set of parameters $b\leq s$, 
analogous to what was discussed in Section \ref{sec:interm} was initiated in \cite{ACK17, AC19, C20}.
In a very interesting series of work by Clark \cite{C21,C22,C23}, the critical case $b=s$ with a critical intermedaite
disorder scaling was studied and, in particular,
an analogue of the Critical 2d Stochastic Heat Flow was constructed in \cite{C22} and path properties
of the polymer in this setting were established in \cite{C21,C23}. 

We refer to \cite{LM10, ACK17, AC19, C20, C21,C22,C23} for details and further historical information.
\vskip 2mm
{\bf Directed polymer on more general graphs.} Work \cite{CoSeZe21} has initiated a general framework of
investigation of the phase transitions of directed polymer models on more irregular and exotic graphs.
These may include percolation clusters, biased Galton-Watson trees, canopy graphs\footnote{canopy graphs are infinite tree structures viewed from the top, eg consider a countable collection of points on the real line and for $d\in \N$ connect every $d$ consecutive points to a common parent a level above and then repeat the same procedure in all higher levels.} etc. 
 It is shown there that a rich phenomenology exists
in terms of existence or absence of weak-to-strong disorder regimes, or existence or absence of $L^2$
regimes depending on the graph structure, on the one hand, and properties of the underlying random walk, eg heat kernels, transience / recurrence etc., on the other.  A couple of interesting observations is that 
Proposition \ref{ui} on the equivalence of weak disorder and uniform integrability of the partition function
might fail in certain exotic graphs, while this is still the case in graphs which admit a Liouville property, i.e. 
bounded harmonic functions are constant. We refer directly to \cite{CoSeZe21} for details and open questions.

\end{document}